\documentclass[twoside,centertags]{amsart}

\newtheorem{innercustomthm}{Theorem}
\newenvironment{customthm}[1]
  {\renewcommand\theinnercustomthm{#1}\innercustomthm}
  {\endinnercustomthm}
  
  \usepackage{fancyhdr}            
\fancypagestyle{plain}{
    \lhead{}
    \fancyhead[R]{\thepage}
    \fancyhead[L]{}
    
    \fancyfoot{}
}

\usepackage[cmtip,color,all]{xy}
\usepackage{amsmath}
\usepackage{amsfonts}
\usepackage{amssymb}
\usepackage{amsthm}
\usepackage{graphicx}
\usepackage{mathrsfs}
\usepackage{hyperref}
\usepackage{tikz-cd}
\usepackage{xcolor}
\usepackage{bbold}
\usepackage{pifont}

\newtheorem{theorem}{Theorem}[subsection]
\newtheorem{corollary}[theorem]{Corollary}
\newtheorem{lemma}[theorem]{Lemma}
\newtheorem{proposition}[theorem]{Proposition}

\theoremstyle{definition}
\newtheorem{definition}[theorem]{Definition}

\theoremstyle{definition}
\newtheorem{remark}[theorem]{Remark}
\newtheorem{example}[theorem]{Example}

\newcommand{\C}{\mathscr{C}}

\newcommand{\Z}{\mathbb{Z}}

\newcommand{\map}{\mathrm{map}^h}
\newcommand{\Q}{\mathbb{Q}}
\newcommand{\F}{\mathbb{F}}
\newcommand{\M}{\mathcal{M}}
\newcommand{\N}{\mathcal{N}}
\newcommand{\cof}{\mathrm{Cof}}
\newcommand{\inj}{\mathrm{inj}}
\newcommand{\I}{\mathcal{I}}
\newcommand{\W}{\mathcal{W}}
\newcommand{\X}{\mathbb{X}}
\newcommand{\ho}{\mathrm{Ho}}
\newcommand{\Req}{\mathrm{R}\text{-}\mathrm{eq}}
\newcommand{\qi}{\mathrm{q.i.}}
\newcommand{\cateq}{\mathrm{cat.eq.}}
\newcommand{\eq}{\mathrm{eq.}}
\newcommand{\et}{\text{\'{E}t}}
\newcommand{\q}{\mathfrak{q}}
\newcommand{\y}{\mathfrak{i}_1}

\numberwithin{equation}{section}

\begin{document}
\title[The simplicial coalgebra of chains under three different notions of weak equivalence]{The simplicial coalgebra of chains under three different notions of weak equivalence}

\author[G. Raptis]{George Raptis}
\address{\newline G.R., Fakultät für Mathematik, Universität Regensburg, 93040 Regensburg, Germany}
\email{\href{mailto:georgios.raptis@mathematik.uni-regensburg.de}{georgios.raptis@mathematik.uni-regensburg.de}}

\author[M. Rivera]{Manuel Rivera}
\address{\newline M.R., Department of Mathematics, Purdue University}
\email{\href{mailto:manuelr@purdue.edu}{manuelr@purdue.edu}}

\begin{abstract}
We study the simplicial coalgebra of chains on a simplicial set with respect to three notions of weak equivalence. To this end, we construct three model structures on the category of reduced simplicial sets for any commutative ring $R$. The weak equivalences are given by: (1) an $R$-linearized version of categorical equivalences, (2) maps inducing an isomorphism on fundamental groups and an $R$-homology equivalence between universal covers, and (3) $R$-homology equivalences. Analogously, for any field $\F$, we construct three model structures on the category of connected simplicial cocommutative $\F$-coalgebras. The weak equivalences in this context are (1$’$) maps inducing a quasi-isomorphism of dg algebras after applying the cobar functor, (2$’$) maps inducing a quasi-isomorphism of dg algebras after applying a localized version of the cobar functor, and (3$’$) quasi-isomorphisms. Building on previous work of Goerss in the context of (3)--(3$'$), we prove that, when $\F$ is algebraically closed, the simplicial $\F$-coalgebra of chains defines a homotopically full and faithful left Quillen functor for each one of these pairs of model categories. More generally, when $\F$ is a perfect field, we compare the three pairs of model categories in terms of suitable notions of homotopy fixed points with respect to the absolute Galois group of $\F$.

\end{abstract}

\maketitle

\setcounter{tocdepth}{1}
\tableofcontents

\section{Introduction}  

\subsection{Overview} A central problem in algebraic topology is to understand how much information about a category of topological spaces, considered up to a specified notion of weak equivalence, is preserved by a particular functorial invariant. Quillen proved that a suitable version of the rational chains functor defines an equivalence of homotopy theories between simply-connected spaces, considered up to rational homotopy equivalence, and simply-connected cocommutative differential graded (dg) $\mathbb{Q}$-coalgebras, considered up to quasi-isomorphism \cite{Q69}. In the context of spaces of finite type, Sullivan developed an effective version of this theory, suitable for geometric applications, in terms of simply-connected commutative dg $\mathbb{Q}$-algebras and their minimal models \cite{S77}. 

\smallskip

For a field $\F$ of arbitrary characteristic, Goerss \cite{Go} considered spaces (simplicial sets) up to $\F$-homology equivalence and the comparison with simplicial cocommutative $\F$-coalgebras, considered up to quasi-isomorphism, via the simplicial $\F$-chains functor:
$$\F[-] \colon \mathsf{sSet} \to \mathsf{sCoCoalg}_{\F}, \ X \mapsto \F[X].$$ 
For any simplicial set $X$, the simplicial cocommutative $\mathbb{F}$-coalgebra of chains $\mathbb{F}[X]$ is given degreewise by the free $\F$-vector space on $X$ together with the coproduct induced by the diagonal map $X \to X \times X$. One of the main results of \cite{Go} shows that the functor $\F[-]$ is homotopically full and faithful when $\mathbb{F}$ is algebraically closed. This implies that any space may be recovered, up to $\F$-localization in the sense of Bousfield \cite{B75}, from its simplicial cocommutative $\F$-coalgebra of chains through the derived functor of the right adjoint of $\F[-]$. More explicitly, the right adjoint of $\F[-]$, also known as the \textit{functor of $\F$-points}, is given by
$$\mathcal{P} \colon \mathsf{sCoCoalg}_{\F} \to \mathsf{sSet}, \ C \mapsto \mathsf{sCoCoalg}_{\F}(\underline{\F}, C),$$
where $\underline{\F}$ denotes the constant simplicial object at $\F$. This result is of particular importance over $\overline{\F}_p$, the algebraic closure of the field $\mathbb{F}_p$ of $p$ elements. Both Quillen and Goerss used the framework of model categories to describe the necessary homotopical constructions. Furthermore, Mandell \cite{M01} proved a similar statement for finite type nilpotent spaces and the $E_{\infty}$-dg-algebra of singular cochains, and also obtained an intrinsic description of the essential image of this functor (see also \cite{Ma2} for the integral case). In another direction, Yuan \cite{Yu21} recently described a full and faithful integral model for finite type nilpotent spaces up to weak homotopy equivalence using the $E_{\infty}$-ring spectrum of spherical cochains. 

\smallskip

All these models fail to capture the information of the fundamental group in complete generality. In fact, for any commutative ring $R$, the $R$-localization of a space drastically changes the fundamental group in general. In order to retain the information of the fundamental group, we will consider instead spaces up to $\pi_1$-$R$-\textit{equivalence}. A $\pi_1$-$R$-equivalence is a map of (based, connected) spaces inducing an isomorphism on fundamental groups and an $R$-homology equivalence between the universal covers \cite{RWZTransactions, BS93}. For instance, a $\pi_1$-$\mathbb{Z}$-equivalence is the same as a weak homotopy equivalence. 

A key result for the present work is that a map of (based, connected) topological spaces is a $\pi_1$-$R$-equivalence if and only if the induced map of coaugmented dg coalgebras of normalized singular chains becomes a quasi-isomorphism of dg algebras upon applying the \textit{cobar functor}. Based on this result, we introduce in the present article a new notion of weak equivalence between connected simplicial cocommutative $R$-coalgebras, called $\widehat{\mathbb{\Omega}}$-\textit{quasi-isomorphism}, that turns out to be precisely the analogue of a $\pi_1$-$R$-equivalence between arbitrary (not necessarily fibrant) reduced simplicial sets. The class of $\widehat{\mathbb{\Omega}}$-\textit{quasi-isomorphisms} is detected by a functor from connected simplicial cocommutative coalgebras to dg algebras,
\[\widehat{\mathbb{\Omega}} \colon \mathsf{sCoCoalg}^0_R \to \mathsf{dgAlg}_R,\]
that is constructed by appropriately localizing the cobar construction of the normalized chains dg coalgebra of a connected simplicial cocommutative coalgebra at a suitable set of $0$-cycles. Then the functor of simplicial $R$-chains restricted to the category of reduced simplicial sets,
$$R[-] \colon \mathsf{sSet}_0 \to \mathsf{sCoCoalg}^0_{R},$$
preserves and detects weak equivalences: a map $f \colon S \to S'$ of reduced simplicial sets is a $\pi_1$-$R$-equivalence if and only if $R[f]$ is an $\widehat{\mathbb{\Omega}}$-\text{quasi-isomorphism} (Theorem \ref{LOmegareflects}). 

It turns out that there are good homotopy theories for these notions: we construct model category structures for reduced simplicial sets up to $\pi_1$-$R$-equivalence and, when $R=\mathbb{F}$ is a field, for connected simplicial cocommutative $\F$-coalgebras up to $\widehat{\mathbb{\Omega}}$-quasi-isomorphism. Moreover, the corresponding adjunction of categories
\begin{eqnarray}\label{chainsandpoints}
\F[-] \colon \mathsf{sSet}_0 \rightleftarrows \mathsf{sCoCoalg}^0_{\F} \colon \mathcal{P}
\end{eqnarray}
defines a Quillen adjunction between these model categories. Furthermore, we prove that the left Quillen functor $\F[-]$ is homotopically full and faithful when $\F$ is an algebraically closed field. We also study the Quillen adjunction \eqref{chainsandpoints} in the more general case where $\mathbb{F}$ is a perfect field with absolute Galois group $G$. In this case, we prove that the derived unit transformation of the Quillen adjunction \eqref{chainsandpoints} may be identified with the canonical map into the homotopy $G$-fixed points, interpreted appropriately in the homotopy theory of simplicial \emph{discrete} $G$-sets up to $\pi_1$-$\F$-equivalence. This is a significant improvement of the main result of \cite{RWZTransactions}, where a detection result was shown, namely, that for any field $\F$, two (fibrant) spaces $X$ and $Y$ are $\pi_1$-$\F$-equivalent if and only if the simplicial cocommutative coalgebras $\F[X]$ and $\F[Y]$ are $\mathbb{\Omega}$-quasi-isomorphic. The present article may be read independently of \cite{Go} and \cite{RWZTransactions}.

\smallskip

The homotopy theory of connected simplicial cocommutative $\F$-coalgebras up to $\widehat{\mathbb{\Omega}}$-quasi-isomorphism fits strictly between two other homotopy theories: (1) connected simplicial cocommutative coalgebras up to $\mathbb{\Omega}$-quasi-isomorphism, and (2) connected simplicial cocommutative coalgebras up to quasi-isomorphism. In the present article, we also construct a model category for (1) and explain its relation to a linearized version of the Joyal model category for reduced simplicial sets \cite{Joyal2, LuHTT, CHL}. On the other hand, the homotopy theory for (2) is the connected version of the model category studied by Goerss \cite{Go} and is related to the Bousfield model category of reduced simplicial sets up to homology equivalence. The main results of the present article are phrased in such a way that the parallelism between three corresponding Quillen adjunctions is emphasized. 

\smallskip

The main motivation for this work is to understand the strength of a particular invariant for homotopy types (in this case the simplicial cocommutative coalgebra of chains) with respect to different notions of weak equivalence. We are particularly interested in notions of weak equivalence arising from the homotopy theory of algebraic structures governed by a dg operad, since these usually yield flexible and explicit theories. One of our eventual goals is to develop a similar analysis for the invariant of the simplicial cocommutative coalgebra of \textit{integral} chains considered up to $\widehat{\mathbb{\Omega}}$-quasi-isomorphism (cf. \cite{Ma2} for the case of the $E_{\infty}$ dg algebra of integral cochains up to quasi-isomorphism). It was conjectured in \cite{RWZTransactions} that this invariant faithfully detects whether two homotopy types are equivalent. A refined form of this conjecture will be addressed in subsequent work.

\subsection{Summary of results}

We now summarize our main results. For any commutative ring $R$, we consider the following three notions of weak equivalence on the category $\mathsf{sSet}_0$ of reduced simplicial sets (simplicial sets with a single vertex).
\begin{enumerate}
\item Let $\Lambda (-;R): \mathsf{sSet}_0 \to \mathsf{dgAlg}_R$ be the restriction  to $\mathsf{sSet}_0$ of the left adjoint of the dg nerve functor $\mathsf{N}_{\text{dg}}: \mathsf{dgCat}_R \to \mathsf{sSet}$ from dg categories to simplicial sets (see Section \ref{cobar_sec2}). A map $f\colon X \to Y$ in $\mathsf{sSet}_0$ is called a \textit{categorical} $R$-\textit{equivalence} if it induces a quasi-isomorphism of dg $R$-algebras \[ \Lambda(f; R) \colon \Lambda(X; R) \xrightarrow{\simeq} \Lambda(Y; R).\] 
This is a linearized version of the notion of categorical equivalence (or Joyal equivalence) between reduced simplicial sets.  We denote by $\mathcal{W}_{J,R}$ the class of categorical $R$-equivalences. 
\item  A map $f\colon X \to Y$ in $\mathsf{sSet}_0$ is called a $\pi_1$-$R$-\textit{equivalence} if it induces an isomorphism between the fundamental groups \[\pi_1(|f|): \pi_1(|X|) \xrightarrow{\cong} \pi_1(|Y|)\] and an $R$-homology isomorphism between the universal covers
\[H_*(\overline{|f|};R): H_*(\overline{|X|};R) \xrightarrow{\cong} H_*(\overline{|Y|};R).\] We denote by $\mathcal{W}_{\pi_1\text{-}R}$ the class of  $\pi_1$-$R$-equivalences. 
\item A map $f\colon X \to Y$ in $\mathsf{sSet}_0$  is called an $R$-\textit{equivalence} if it induces an isomorphism in $R$-homology \[H_*(f;R): H_*(X;R) \xrightarrow{\cong} H_*(Y;R).\] We denote by $\mathcal{W}_R$ the class of $R$-equivalences. 
\end{enumerate}

The category of reduced simplicial sets is a convenient framework to make certain pointed constructions functorial (such as the fundamental group and the universal cover). Moreover, the reduced setting will be convenient in order to establish key connections between (1)--(2) and the cobar construction.
Note that the class of $\pi_1$-$R$-equivalences is strictly contained in the class of $R$-equivalences that induce a  $\pi_1$-isomorphism (see \cite[Example 4.35]{Ha}).

\smallskip

We construct three model category structures on $\mathsf{sSet}_0$, one for each of the above notions of weak equivalence.

\begin{customthm}{1}[see Theorem \ref{sset_modelcat}] \label{thm1} Let $R$ be a commutative ring. The category $\mathsf{sSet}_0$ admits three left proper combinatorial model category structures, denoted by $(\mathsf{sSet}_0, R\text{-cat-eq.})$, $(\mathsf{sSet}_0, \pi_1\text{-}R\text{-eq.})$, and $(\mathsf{sSet}_0, R\text{-eq.})$, which have the monomorphisms as cofibrations and $\mathcal{W}_{J,R}, \mathcal{W}_{\pi_1\text{-}R}$, and $\mathcal{W}_{R}$ as weak equivalences, respectively. Furthermore, we have strict inclusions $\mathcal{W}_{J,R} \subsetneq  \mathcal{W}_{\pi_1\text{-}R}\subsetneq \mathcal{W}_{R}.$  
\end{customthm} 

The model category $(\mathsf{sSet}_0, \pi_1\text{-}\mathbb{Z}\text{-eq.})$ is the classical Kan--Quillen model structure on reduced simplicial sets. Moreover, as part of our discussion of the model category $(\mathsf{sSet}_0, \pi_1\text{-}R\text{-eq.})$, we prove a result similar to the classical \textit{fracture theorem} \cite{Sul70, BK72, MPBook}, but now fully taking into account the fundamental group (Theorem \ref{fiberwise_fracture_thm}). More specifically, this result describes the weak homotopy type of a reduced simplicial set as a homotopy pullback in terms of its fibrant replacements in $(\mathsf{sSet}_0, \pi_1\text{-}\mathbb{Q}\text{-eq.})$ and $(\mathsf{sSet}_0, \pi_1\text{-}\mathbb{F}_p\text{-eq.})$ for all prime numbers $p$.

\smallskip

The main goal of the present article is to model each of the three homotopy theories of Theorem \ref{thm1} using connected simplicial cocommutative coalgebras via the functor of simplicial chains. For each of the three notions of weak equivalence in $\mathsf{sSet}_0$, we define an analogous notion in the category $\mathsf{sCoCoalg}^0_R$ of connected simplicial cocommutative coalgebras.
\begin{enumerate}
\item A map $f: C \to C'$ in $\mathsf{sCoCoalg}^0_R$ is called an $\mathbb{\Omega}$\textit{-quasi-isomorphism} if it becomes a quasi-isomorphism of dg algebras after applying $\mathbb{\Omega} \colon\mathsf{sCoCoalg}^0_R \to \mathsf{dgAlg}_R$, that is, the normalized chains functor followed by the cobar functor $\mathsf{Cobar}\colon \mathsf{dgCoalg}^c_R \to \mathsf{dgAlg}_R$  (see Section \ref{cobar_sec}). We denote by $\mathcal{W}_{\mathbb{\Omega}}$ the class of $\mathbb{\Omega}$-quasi-isomorphisms.
\item A map $f: C\to C'$ in $\mathsf{sCoCoalg}^0_R$ is called an $\widehat{\mathbb{\Omega}}$ \textit{-quasi-isomorphism} if it becomes a quasi-isomorphism of dg algebras after first localizing (in an appropriate way) $\mathbb{\Omega}(C)$ and $\mathbb{\Omega}(C')$ at suitable sets of $0$-cycles (see Section \ref{loc_sec}). We denote by $\mathcal{W}_{\widehat{\mathbb{\Omega}}}$ the class of $\widehat{\mathbb{\Omega}}$-quasi-isomorphisms. 
\item A map $f: C \to C'$ in $\mathsf{sCoCoalg}^0_R$ is called a \textit{quasi-isomorphism} if the induced map of normalized chains induces an isomorphism on homology. We denote by $\mathcal{W}_{\text{q.i.}}$ the class of quasi-isomorphisms. 
\end{enumerate}

\smallskip

When $R=\mathbb{F}$ is a field, we construct three corresponding model category structures on $\mathsf{sCoCoalg}^0_{\mathbb{F}}$.

\begin{customthm}{2}[see Theorem \ref{coalg_modelcat}] \label{thm2} Let $\mathbb{F}$ be a field. The category $\mathsf{sCoCoalg}^0_{\mathbb{F}}$ admits three left proper combinatorial model category structures, denoted by \begin{center}$(\mathsf{sCoCoalg}^0_{\mathbb{F}}, \mathbb{\Omega}\text{-}\qi)$, $(\mathsf{sCoCoalg}^0_{\mathbb{F}}, \widehat{\mathbb{\Omega}}\text{-}\qi)$, and $(\mathsf{sCoCoalg}^0_{\mathbb{F}}, \qi)$,
\end{center}
which have the injective maps as cofibrations and $\mathcal{W}_{\mathbb{\Omega}}$, $\mathcal{W}_{\widehat{\mathbb{\Omega}}}$, and $\mathcal{W}_{\qi}$ as weak equivalences, respectively. Furthermore, we have strict inclusions \[\W_{\mathbb{\Omega}} \subsetneq \W_{\widehat{\mathbb{\Omega}}} \subsetneq \W_\qi.\]
\end{customthm}

The proofs of Theorems \ref{thm1} and \ref{thm2} rely on a useful method for constructing combinatorial model category structures which is based on J. Smith’s recognition theorem and may be of independent interest. This general method is discussed in Section \ref{prelim_modelcat} and can be read independently of the rest of the article.

\smallskip

We then compare the model categories of Theorems \ref{thm1} and \ref{thm2}. For this comparison, a key result is that $\pi_1$-$R$-equivalences can be completely described in terms of $\widehat{\mathbb{\Omega}}$-quasi-isomorphisms; see Theorem \ref{LOmegareflects}. We show that we have three Quillen adjunctions (Proposition \ref{quillen_adj}):

\begin{equation}\label{Qadj1}
\F[-]\colon (\mathsf{sSet}_0, \F\text{-}\cateq) \rightleftarrows (\mathsf{sCoCoalg}^0_\F, \mathbb{\Omega}\text{-}\qi) \colon \mathcal{P}
\end{equation}
\begin{equation} \label{Qadj2}
\F[-]\colon (\mathsf{sSet}_0, \pi_1\text{-}\F\text{-}\eq) \rightleftarrows (\mathsf{sCoCoalg}^0_\F, \widehat{\mathbb{\Omega}}\text{-}\qi) \colon \mathcal{P}
\end{equation}
\begin{equation} \label{Qadj3}
\F[-]\colon (\mathsf{sSet}_0, \F\text{-}\eq) \rightleftarrows (\mathsf{sCoCoalg}^0_\F, \qi) \colon \mathcal{P}.
\end{equation}

\medskip

When $\F$ is an algebraically closed field, or more generally a perfect field, we prove the following statements about these Quillen adjunctions.
 
\begin{customthm}{3}[see Theorem \ref{comparisons1} and Corollary \ref{perfectfield_corollary}]\label{thm4}
If $\F$ is  algebraically closed, then $\F[-]$ is homotopically full and faithful with respect to each one of the three Quillen adjunctions \eqref{Qadj1}--\eqref{Qadj3}.

\indent More generally, suppose $\F$ is a perfect field with absolute Galois group $G$, and let $X$ be a reduced simplicial set. Then the derived unit transformation
$$X \to \mathbb{R}\mathcal{P}(\F[X])$$
of each one of the three Quillen adjunctions can be identified with the canonical map $$X \to (\delta(X))^{hG}$$ from $X$ to the homotopy $G$-fixed points of $\delta(X)$, where $\delta(X)$ denotes $X$ equipped with the trivial $G$-action and the homotopy $G$-fixed points functor $(-)^{hG}$ is interpreted appropriately in each case in the respective homotopy theory. 
\end{customthm}

In particular, Theorem \ref{thm4} says that when $\mathbb{F}$ is an algebraically closed field, one may functorially recover any reduced simplicial set $X$, up to $\pi_1$-$\mathbb{F}$-equivalence, from any simplicial cocommutative coalgebra that is $\widehat{\mathbb{\Omega}}$-quasi-isomorphic to the simplicial coalgebra of chains $\mathbb{F}[X]$.  The proof of Theorem \ref{thm4} makes use of the structure theory of coalgebras over a field and also relies on the construction of three corresponding model structures on the category of simplicial discrete $G$-sets (Theorem \ref{Gsset_modelcat}). The statement about the Quillen adjunction \eqref{Qadj3} (for $\mathsf{sSet}$) was shown by Goerss \cite{Go}; this result was extended to the context of simplicial presheaves of coalgebras (with respect to the local model structures) by the first-named author \cite{Ra} and it has also been shown for motivic homotopy theory in \cite{Gu}. 

\subsection{Organization of the paper} In Section \ref{prelim_sec}, we review some categorical and algebraic preliminaries and fix the notation and terminology. Section \ref{modelcategories_sec} is a recollection of several known model category structures used throughout the article. 

In Section \ref{loc_sec}, we discuss suitable notions of homotopical localization for (reduced) simplicial sets at a set of $1$-simplices and for dg algebras at a set of $0$-cycles. These notions are then used to define the localized cobar construction $\widehat{\mathbb{\Omega}}$. In Section \ref{sec_3_weq}, we define and study the three notions of weak equivalence for connected simplicial cocommutative coalgebras ($\mathbb{\Omega}$-quasi-isomorphisms, $\widehat{\mathbb{\Omega}}$-quasi-isomorphisms, and quasi-isomorphisms) and the three corresponding notions for reduced simplicial sets (categorical $R$-equivalences, $\pi_1$-$R$-equivalences, and $R$-equivalences). 

In Section \ref{prelim_modelcat}, we describe a general method for constructing combinatorial model category structures, based on ideas of J. Smith, which will be used later in Sections \ref{model_coalg} and \ref{sec_fields}. In Section \ref{model_coalg}, we establish the existence of three model structures on the category of reduced simplicial sets (Theorem \ref{thm1}), three model structures on the category of connected simplicial cocommutative coalgebras (Theorem \ref{thm2}), and obtain the three corresponding Quillen adjunctions \eqref{Qadj1}--\eqref{Qadj3}. In addition, in Section \ref{model_coalg}, we prove a fiberwise version of the fracture theorem for arbitrary pointed connected weak homotopy types that fits nicely in the context of the model categories $(\mathsf{sSet}_0, \pi_1\text{-}R\text{-eq.})$ for $R=\Q, \F_p$ ($p$ prime). 

Finally, in Section \ref{sec_fields}, we review some key facts about the structure theory of coalgebras and prove our main comparison results (Theorem \ref{thm4}) about the three Quillen adjunctions \eqref{Qadj1}--\eqref{Qadj3} in the case where $\F$ is an algebraically closed field (Theorem \ref{comparisons1}) or a perfect field (Theorem \ref{comparisons2} and Corollary \ref{perfectfield_corollary}). 

In Appendix \ref{sec_appendix}, we give a detailed proof of the fact that for the natural cylinder construction $C \mapsto \text{Cyl}(C)$ in connected simplicial coalgebras, the canonical projection map $\text{Cyl}(C) \to C$ is an $\mathbb{\Omega}$-quasi-isomorphism for any $C \in \mathsf{sCoCoalg}^0_{\F}$. This is a key step in the proof of the existence of the model category structure $(\mathsf{sCoCoalg}^0_{\F}, \mathbb{\Omega}\text{-q.i.})$ shown in Section \ref{model_coalg}. 

\subsection*{Acknowledgments} The authors would like to thank Hadrian Heine, Julian Holstein, Andrey Lazarev, Manfred Stelzer, Felix Wierstra, and Mahmoud Zeinalian for interesting discussions and fruitful exchanges. GR was partially supported by \emph{SFB 1085 -- Higher Invariants} (University of Regensburg), funded by the DFG. MR was supported by NSF Grant 210554 and the Karen EDGE Fellowship. MR would like to acknowledge the excellent working conditions of the Max Planck Institute for Mathematics in Bonn and of Université Sorbonne Paris Nord, where parts of this research were conducted. Finally, the authors thank the anonymous referees for their careful reading and helpful comments. 

\section{Preliminaries} \label{prelim_sec}
In this section, we recall some categorical and algebraic preliminaries.

\subsection{Simplicial objects}
Let $\mathbb{\Delta}$ be the category whose objects are the non-empty finite ordinals $\{ [n]=\{0 < \ldots < n\} \ | \ n  \in \mathbb{N} \}$ and the morphisms are given by order-preserving maps. The morphisms in $\mathbb{\Delta}$ are generated by coface maps $d^i\colon [n-1] \to [n]$ and codegeneracy maps $s^i\colon [n+1] \to [n]$ for $i=0,...,n$ and these maps satisfy the usual cosimplicial identities. 

A \textit{simplicial object} in a category $\mathsf{C}$ is a functor $F\colon \mathbb{\Delta}^{op} \to \mathsf{C}$, where $\mathbb{\Delta}^{op}$ denotes the opposite category of $\mathbb{\Delta}$. These form a category, denoted by $\mathsf{sC}$, with morphisms being the natural transformations. For any object $F$ in $\mathsf{sC}$ we write $F_n :=F([n])$. Thus any simplicial object $F$ in $\mathsf{sC}$ is determined  by the data of objects $F_0, F_1, F_2,...$ in $\mathsf{C}$ together with face maps $F(d^i)=d_i\colon F_n \to F_{n-1}$ and degeneracy maps $F(s^i)=s_i\colon F_n \to F_{n+1}$ in $\mathsf{C}$ satisfying the usual simplicial identities. 

\smallskip

Simplicial objects in $\mathsf{Set}$, the category of sets, are called \textit{simplicial sets}. We denote by $\mathsf{sSet}_0$ the full subcategory of $\mathsf{sSet}$ whose objects are all simplicial sets $S$ such that $S_0$ is a singleton. The objects of  $\mathsf{sSet}_0$ are called \textit{reduced simplicial sets}. 

\subsection{Simplicial coalgebras}
Fix a commutative ring $R$ and write $\otimes : = \otimes_R$. Let $\mathsf{Coalg}_R$ be the category of counital coassociative $R$-coalgebras. More precisely, the objects in $\mathsf{Coalg}_R$ are triples $(C, \Delta, \epsilon)$ where $C$ is an $R$-module, $\Delta\colon C \to C \otimes C$ is an $R$-linear map, called the \textit{coproduct}, satisfying the coassociativity condition 
$$(\Delta  \otimes \text{id}_C) \circ \Delta = (\text{id}_C \otimes \Delta) \circ \Delta,$$
and $\epsilon\colon C \to R$ is an $R$-linear map, called the \textit{counit}, that satisfies
$$(\text{id}_C \otimes \epsilon) \circ \Delta = \text{id}_C= (\epsilon \otimes \text{id}_C)\circ \Delta$$ 
(we used implicitly the canonical identifications $C \otimes R \cong C \cong R \otimes C$). The morphisms in $\mathsf{Coalg}_R$ are $R$-linear maps of $R$-modules preserving the coproduct and the counit. We denote by $\mathsf{CoCoalg}_R$ the category of \textit{cocommutative $R$-coalgebras}, namely, the full subcategory of $\mathsf{Coalg}_R$  consisting of those coalgebras $(C, \Delta, \epsilon)$ for which the coproduct $\Delta$ is cocommutative, i.e., it satisfies $\Delta = \tau \circ \Delta$, where the $R$-linear map $\tau\colon C \otimes C \to C \otimes C$ is the switch map determined by $\tau(x \otimes y )= y \otimes x$. 

\smallskip

In this article we will consider \textit{simplicial (cocommutative) coalgebras}, i.e., simplicial objects in $\mathsf{Coalg}_R$ ($\mathsf{CoCoalg}_R$), as models for different homotopy theories. We note that a simplicial coalgebra $C\colon \mathbb{\Delta}^{op} \to  \mathsf{Coalg}_R$ is equivalently described as a simplicial $R$-module $C$ equipped with maps of simplicial $R$-modules $\Delta\colon C \to C \otimes C$ and $\epsilon \colon C \to R$, where $(C \otimes C)_n = C_n \otimes C_n$ and $R$ here denotes the corresponding constant simplicial object, making $C$ into a counital coassociative coalgebra object. We denote by $\mathsf{sCoalg}^0_R$ the full subcategory of $\mathsf{sCoalg}_R$ consisting of those simplicial objects $C\colon \mathbb{\Delta}^{op} \to \mathsf{Coalg}_R$ for which $C_0 = (R, \Delta_R, \mathrm{id}_R)$ where $\Delta_R \colon R \xrightarrow{\cong} R \otimes R$ is defined by $\Delta_R(1)=1 \otimes 1$.  The objects of $\mathsf{sCoalg}^0_R$ are called \textit{connected simplicial coalgebras}. The category $\mathsf{sCoCoalg}^0_R$ is defined similarly. 

\subsection{Simplicial chains}

Any simplicial set $X \in \mathsf{sSet}$ gives rise to a simplicial cocommutative coalgebra $R[X]\colon \mathbb{\Delta}^{op} \to \mathsf{CoCoalg}_R$. The underlying $R$-module $R[X]_n$ is defined to be $R[X_n]$, the free $R$-module generated by the set $X_n$. The face and degeneracy maps of $R[X]$ are induced by those of $X$, by functoriality. The coproduct maps $$\Delta_n\colon R[X]_n \to R[X]_n \otimes R[X]_n$$ are induced functorially by the diagonal maps $$X_n \to X_n \times X_n, \ x \mapsto (x,x).$$ 
This construction gives rise to the \textit{simplicial (R-)chains functor}
$$R[-] \colon \mathsf{sSet} \to \mathsf{sCoCoalg}_R.$$
The simplicial $R$-chains functor $R[-]$ is clearly given by the corresponding functor $R[-] \colon \mathsf{Set} \to \mathsf{CoCoalg}_R$ by passing to simplicial objects in the respective categories. We also obtain a restricted functor
$$R[-] \colon \mathsf{sSet}_0 \to \mathsf{sCoCoalg}^0_R.$$ The functor of simplicial $R$-chains has a right adjoint
$$\mathcal{P}\colon \mathsf{sCoCoalg}_R\to  \mathsf{sSet},$$ called the \textit{functor of (R-)points} (or the \textit{set-like elements functor}), which is defined pointwise by
$$\mathcal{P}(C)_n = \text{Hom}_{\mathsf{CoCoalg}}(R, C_n),$$
where $R$ is considered as a cocommutative $R$-coalgebra (as indicated above). More explicitly, 
$$\mathcal{P}(C)_n \cong \{ x \in C_n \ | \ \Delta_n(x) = x \otimes x \text{ and } \epsilon(x)=1\}.$$ We also obtain a restricted functor
$$\mathcal{P} \colon \mathsf{sCoCoalg}_R^0 \to  \mathsf{sSet}_0.$$ 
When $R$ has no non-trivial idempotents, the unit of the adjunction $(R[-], \mathcal{P})$ is a natural isomorphism $$X \xrightarrow{\cong} \mathcal{P}(R[X])$$ for any simplicial set $X$. 
In particular, the simplicial $R$-chains functor $R[-]$ is full and faithful. Of course this also holds for the adjunction $R[-] \colon \mathsf{Set} \rightleftarrows \mathsf{CoCoalg}_R \colon \mathcal{P}$ before passing to simplicial objects. 

\subsection{Differential graded algebras and coalgebras}

A \textit{differential graded (dg) $R$-module}, or \textit{chain complex} for short, is given by a pair $(M, d)$ where $M$ is a $\Z$-graded $R$-module and $d\colon M \to M$ is an $R$-linear map of degree $-1$ satisfying $d \circ d = 0$. 
If $M=(M, d_M)$ and $N =(N, d_N)$ are chain complexes, then $M \otimes N$ is the chain complex with $(M \otimes N)_n= \bigoplus_{i+j=n} M_i \otimes N_j$ and differential $d_{M \otimes N}= d_M \otimes \text{id}_N + \text{id}_M \otimes d_N$. Throughout the article, we use the Koszul sign rule when applying graded maps to elements.

\smallskip

A \textit{dg unital associative $R$-algebra} $(A,d,\mu)$, or \textit{dg algebra} for short,  consists of a dg $R$-module $(A,d)$ and a degree $0$ unital associative product $\mu \colon A \otimes A \to A$ for which $d$ is a derivation, i.e., $d$ satisfies the following property:
$$d \circ \mu = \mu \circ (d \otimes \text{id}_A) + \mu \circ (\text{id}_A \otimes d).$$
(Writing $\mu(a \otimes b)=ab$ and using the Koszul sign rule, the above equation says
$d(ab)= d(a)b + (-1)^{|a|}ad(b)$, where $|a|$ denotes the degree of $a$.) $R$ is regarded as a dg algebra concentrated in degree $0$. We denote by $\mathsf{dgAlg}_R$ the category of ($\Z$-graded) dg algebras with morphisms the degree $0$ maps that preserve the differential and multiplicative structures. 

\smallskip

A \textit{dg counital coassociative $R$-coalgebra} $(N,\partial, \mathbf{\Delta})$, or \textit{dg coalgebra} for short,  consists of a dg $R$-module $(N, \partial)$ and a degree $0$ counital coassociative coproduct $\mathbf{\Delta} \colon N \to N \otimes N$ for which $\partial$ is a coderivation, i.e.,
$$\mathbf{\Delta} \circ \partial=  (\partial \otimes \text{id}_N) \circ \mathbf{\Delta} +  (\text{id}_N \otimes \partial) \circ \mathbf{\Delta}.$$ We denote by $\mathsf{dgCoalg}_R$ the category of dg coalgebras with morphisms the degree $0$ maps that preserve all the structure. 

A \textit{coaugmentation} of a dg coalgebra $N$ is a map of dg coalgebras $e: R \to N$, where $R$ is the dg coalgebra concentrated in degree $0$ with coproduct determined by the isomorphism $R \xrightarrow{\cong} R \otimes R$. A \textit{coaugmented dg coalgebra} is a dg coalgebra equipped with a coaugmentation. For any coaugmented dg coalgebra $(N, \partial, \mathbf{\Delta}, e)$, we write $\overline{N}:= N/e(R)$ and denote by $\overline{\partial}: \overline{N}\to \overline{N}$ and $\overline{\mathbf{\Delta}}: \overline{N} \to \overline{N}^{\otimes 2}$ the induced differential and coproduct, respectively. A \textit{conilpotent dg coalgebra} is a coaugmented dg coalgebra $(N,\partial, \mathbf{\Delta}, e)$ such that $$\overline{N}= \bigcup_{n=1}^{\infty} \text{ker}( \overline{\mathbf{\Delta}}^n ),$$ where $\overline{\mathbf{\Delta}}^n: \overline{N} \to  \overline{N}^{\otimes {n+1}}$ denotes the $n$-times iterated coproduct. Let $\mathsf{dgCoalg}^c_R$ denote the category whose objects are the conilpotent dg coalgebras and the morphisms are the maps of dg coalgebras that preserve the coaugmentations.

\subsection{Normalized chains and coalgebra structures}

Every simplicial counital coassociative coalgebra $C$ gives rise to a dg counital coassociative coalgebra $\mathcal{N}_*(C)$ defined as follows. Given a simplicial coassociative coalgebra $C$ with coproduct $\Delta\colon C \to C \otimes C$, let $(\mathcal{N}_*(C), \partial)$ be the chain complex obtained as the quotient of chain complexes $(N_*(C), \partial) / (D_*(C), \partial)$, where $N_n(C)= C_n$, the differential $$\partial= \sum_i(-1)^{i}d_i\colon N_{*}(C) \to N_{*-1}(C)$$ is given by the alternating sum of the face maps of $C$, and $D_*(C) \subseteq N_*(C)$ is the subcomplex generated by the images of the degeneracy maps of $C$. The chain complex $(\mathcal{N}_*(C), \partial)$ becomes a dg coalgebra when equipped with the coproduct
$$\mathbf{\Delta} \colon \mathcal{N}_*(C) \xrightarrow{\mathcal{N}_*(\Delta)} \mathcal{N}_*(C \otimes C) \xrightarrow{AW} \mathcal{N}_*(C) \otimes \mathcal{N}_*(C).$$
\noindent The map $AW$ is the \textit{Alexander-Whitney natural transformation}, which is defined for any $x \otimes y \in (C \otimes C)_n= C_n \otimes C_n$ by $$AW(x \otimes y)=  \sum_{p=1}^{n+1} d_{p+1} \circ \dots \circ d_n (x) \otimes \underbrace{d_0 \circ \dots \circ d_0}_{p}(y),$$
where each $d_i$ denotes the respective face map of $C$. Moreover, the counit $\epsilon \colon C \to R$ determines an associated counit $\mathcal{N}_*(\epsilon) \colon \mathcal{N}_*(C) \to \mathcal{N}_*(\underline{R}) = R$. The construction $$(C, \Delta, \epsilon) \mapsto (\mathcal{N}_*(C), \partial, \mathbf{\Delta}, \mathcal{N}_*(\epsilon))$$ extends to the \textit{normalized chains functor}
$$\mathcal{N}_*\colon \mathsf{sCoalg}_R \to \mathsf{dgCoalg}_R.$$ 
We have an induced functor
$\mathcal{N}_*\colon \mathsf{sCoalg}^0_R \to \mathsf{dgCoalg}^c_R,$
where we equip $\mathcal{N}_*(C)$ with the natural coaugmentation map $e: R \cong \mathcal{N}_0(C) \to \mathcal{N}_*(C)$ for any $C \in \mathsf{sCoalg}^0_R$. 
These functors preserve colimits, since these are computed in the categories of simplicial modules and chain complexes, respectively. 

\subsection{The cobar construction} \label{cobar_sec}

We now recall the definition of the \textit{cobar construction} (or the \textit{cobar functor}):
$$\mathsf{Cobar}\colon \mathsf{dgCoalg}^c_R \to \mathsf{dgAlg}_R.$$
For any $N \in  \mathsf{dgCoalg}^c_R$, the underlying graded algebra of $\mathsf{Cobar}(N)$ is the tensor algebra
$$T(s^{-1} \overline{N})= R \oplus s^{-1} \overline{N} \oplus  ( s^{-1}\overline{N} \otimes  s^{-1}\overline{N}) \oplus  ( s^{-1}\overline{N} \otimes s^{-1} \overline{N} \otimes  s^{-1}\overline{N}) \oplus... ,$$
where $s^{-1}$ denotes the functor which shifts the grading by $(-1)$, namely, $(s^{-1}\overline{N})_i=\overline{N}_{i+1}$. The differential $$D\colon T(s^{-1} \overline{N})\to T(s^{-1} \overline{N})$$ is defined by extending the map $$-s^{-1} \circ \overline{\partial} \circ s^{+1} + (s^{-1} \otimes s^{-1}) \circ \overline{\mathbf{\Delta}} \circ s^{+1} \colon s^{-1}\overline{N} \to T(s^{-1}\overline{N})$$ as a derivation to all of $T(s^{-1} \overline{N}).$  The equation $D\circ D=0$ is equivalent to the three properties: $\partial^2=0$, $\partial$ is a coderivation of $\mathbf{\Delta}$, and $\mathbf{\Delta}$ is coassociative.
The functor $\mathsf{Cobar}$ preserves colimits. 

For any $N \in \mathsf{dgCoalg}_R^c$, the dg algebra $\mathsf{Cobar}(N)$ is equipped with a natural augmentation $\mathsf{Cobar}(N) \to R$ given by the canonical projection map. Then the cobar construction may be regarded as a functor
\[ \mathsf{Cobar} \colon \mathsf{dgCoalg}_R^c \to \mathsf{dgAlg}_R^{a},\]
where $\mathsf{dgAlg}_R^{a}$ denotes the category of augmented dg algebras. This functor has a right adjoint
\[\mathsf{Bar} \colon \mathsf{dgAlg}_R^{a} \to \mathsf{dgCoalg}_R^c,\]
known as the \textit{bar construction}. 

\subsection{Loop spaces and the cobar construction} \label{cobar_sec2} We consider the following composition of functors 
$$\Lambda( - ; R):= \mathsf{Cobar} \circ \mathcal{N}_* \circ R[-]\colon \mathsf{sSet}_0 \to \mathsf{dgAlg}_R.$$ 
\noindent This functor can be identified with the restriction of the left adjoint of a well-known adjunction between simplicial sets and dg categories. We recall that a (small) dg $R$-category is a (small) category enriched in (the monoidal category of) chain complexes of $R$-modules. This is a many-object version of a dg $R$-algebra; specifically, a dg algebra is exactly a  dg category with a single object. Moreover, the category $\mathsf{dgAlg}_R$ of dg algebras is a full subcategory of the category $\mathsf{dgCat}_R$ of small dg categories. (Analogously, we may view an object in $\mathsf{sSet}$ as a many-object version of an object in $\mathsf{sSet}_0$.) Lurie \cite[Construction 1.3.1.6]{LuHA} constructed a right adjoint functor $$\mathsf{N}_{\text{dg}} \colon \mathsf{dgCat}_R \to \mathsf{sSet}$$ called the \emph{differential graded nerve functor}, and this (adjoint pair) also restricts to $\mathsf{N}_{\text{dg}} \colon \mathsf{dgAlg}_R \to \mathsf{sSet}_0$. 
The following result relates the cobar functor to the differential graded nerve functor.

\begin{theorem} \cite[Theorems 6.1 and 7.1]{RiZe16} \label{leftadjoint} The functor $\Lambda(-; R)\colon \mathsf{sSet}_0 \to \mathsf{dgAlg}_R$ is left adjoint to the functor $\mathsf{N}_{\text{dg}}\colon \mathsf{dgAlg}_R \to \mathsf{sSet}_0$.
\end{theorem} 



For any $X \in \mathsf{sSet}_0$, we denote by $C_*(\Omega |X|; R)$ the dg $R$-algebra of (normalized) singular $R$-chains on the based Moore loop space of the geometric realization $|X|$. The following is one of the main results of \cite{RiZe16} and extends a classical theorem of Adams to the non-simply-connected case. 

\begin{theorem}\cite[Proposition 8.2 and Corollary 9.2]{RiZe16} \label{nscAdams} Let $X \in \mathsf{sSet}_0$ be a Kan complex. Then the dg $R$-algebras $ \Lambda(X;R)$ and $C_*(\Omega |X|; R)$ are naturally quasi-isomorphic. 
\end{theorem}

\noindent For simplicity (and partly motivated by the last theorem), we denote by 
$$\mathbb{\Omega}\colon \mathsf{sCoCoalg}^0_R \to \mathsf{dgAlg}_R$$ the functor defined as the composition $\mathbb{\Omega} = \mathsf{Cobar} \circ \mathcal{N}_*.$ The functor $\mathbb{\Omega}$ preserves colimits.

\subsection{The fundamental bialgebra}

For any $C \in \mathsf{sCoCoalg}^0_R$, the dg algebra $\mathbb{\Omega}(C)$ may be naturally equipped with additional structure. In fact,  $\mathbb{\Omega}(C)$ has a natural coproduct 
$$\nabla \colon \mathbb{\Omega}(C) \to \mathbb{\Omega}(C) \otimes \mathbb{\Omega}(C)$$
making it a dg bialgebra. Specifically, since $C$ is a simplicial \textit{cocommutative} coalgebra, the dg coalgebra of normalized chains $\mathcal{N}_*(C)$ has a natural $E_2$-coalgebra structure. Furthermore, $\mathcal{N}_*(C)$ has an $E_{\infty}$-coalgebra structure that may be described explicitly through a coaction of the surjection operad (see \cite{McSmi}). The $E_2$-coalgebra part of the structure induces a coassociative coproduct on the cobar construction of the underlying $E_1$-coalgebra. This is described explicitly in \cite{Ka03} in terms of the corresponding structure maps of the surjection operad.

In degree $0$, which is the only case we need to consider for the present article, the coproduct
\[\nabla \colon \mathbb{\Omega}(C)_0 \to  \mathbb{\Omega}(C)_0 \otimes  \mathbb{\Omega}(C)_0\] is explicitly determined by the formulas $\nabla(1)=1 \otimes 1$ and
\begin{eqnarray} \label{nabla}
\nabla( s^{-1}\overline{x} ) = \sum_{(x)} s^{-1}\overline{x'} \otimes s^{-1}\overline{x''} + 1 \otimes s^{-1} \overline{x} + s^{-1} \overline{x} \otimes 1,
\end{eqnarray}
where $x \in C_1$, $\overline{x} \in \mathcal{N}_1(C)$ denotes the class of $x$ in the cokernel of the coaugmentation $R \to \mathcal{N}_*(C)$, and $s^{-1}$ is the degree shift functor. Here we have used Sweedler's notation $\Delta_1(x)= \sum_{(x)}x' \otimes x''$ for the coproduct of $C_1$. In the above formula, $s^{-1}\overline{x}$ is considered as a monomial of length $1$ and total degree $0$ in $\mathsf{Cobar}(\mathcal{N}_*(C))=\mathbb{\Omega}(C)$, and $\nabla$ may then be extended as an algebra map $\nabla \colon \mathbb{\Omega}(C)_0 \to \mathbb{\Omega}(C)_0 \otimes \mathbb{\Omega}(C)_0.$ In particular, if $\Delta_1(x)=x \otimes x$, then
\[ \nabla(s^{-1}\overline{x} + 1)=(s^{-1}\overline{x} + 1) \otimes (s^{-1}\overline{x} + 1).\]
 The induced coproduct
$$H_0(\nabla) \colon H_0(\mathbb{\Omega}(C)) \to H_0(\mathbb{\Omega}(C)) \otimes H_0(\mathbb{\Omega}(C))$$ endows the algebra $H_0(\mathbb{\Omega}(C))$ with a bialgebra structure that is functorial with respect to morphisms in $\mathsf{sCoCoalg}^0_R$. This bialgebra is called the \textit{fundamental bialgebra of $C$} \cite[Section 5.1]{RWZTransactions} and determines a functor
\[\pi \colon \mathsf{sCoCoalg}^0_R \to \mathsf{Bialg}_R, \ \ C \mapsto \pi(C) := H_0(\mathbb{\Omega}(C)).\]
The fundamental bialgebra admits an explicit description in the case where $C=R[S], S \in \mathsf{sSet}_0$. Recall that for any $S \in \mathsf{sSet}_0$, the \emph{homotopy category} $\tau(S)$ of $S$ is the monoid defined as the quotient of the free monoid $F(S_1)$ on $S_1$ by the relations that arise from the elements of $S_2$. Explicitly, there is a relation $\sigma_2 \cdot \sigma_0 \sim \sigma_1$ for any $\alpha \in S_2$ with $d_i(\alpha)=\sigma_i$ for $i =0,1,2$. For any $\sigma \in S_1$ we denote by $[\sigma]$ the $\sim$-equivalence class of $\sigma$ in $\tau(S)$. The unit of the monoid $\tau(S)$ corresponds to $[s_0(*)]$, where $\{*\}=S_0$ and $s_0: S_0 \to S_1$ is the degeneracy map. This defines a functor $\tau \colon \mathsf{sSet}_0 \to \mathsf{Mon}$ which is left adjoint to the (restriction of the) classical nerve functor $\mathsf{N} \colon \mathsf{Mon} \to \mathsf{sSet}_0$ on the category of monoids -- this adjunction is the restriction of the adjunction $\tau \colon \mathsf{sSet} \rightleftarrows \mathsf{Cat} \colon \mathsf{N}$ between simplicial sets and small categories. 

\begin{proposition}\label{fundamentalbialgebra}
For any $S \in \mathsf{sSet}_0$, there is a natural isomorphism of bialgebras $$\pi(R[S]) \xrightarrow{\cong} R[\tau(S)],$$
where the bialgebra structure on the left hand side is induced by (\ref{nabla}) and on the right hand side by the canonical monoid bialgebra structure.
\end{proposition}
\begin{proof}
This follows from the proof of \cite[Theorem 26 and Proposition 27]{RWZTransactions}. We outline the proof below for completeness. 
We define a map of $R$-algebras $$\phi: \mathbb{\Omega}(R[S])_0 \to R[\tau(S)]$$
by setting $\phi(1_{R}) = 1_{R}[s_0(*)]$ and
$$\phi: s^{-1} \overline{\sigma} \mapsto \big([\sigma] - 1_{R}[s_0(*)]\big)$$ for any non-degenerate $1$-simplex $\sigma \in S_1$. On monomials of arbitrary length in  $\mathbb{\Omega}(R[S])_0$, $\phi$ is defined by extending the above formula as an algebra map. A straightforward computation (see the proof of \cite[Theorem 26]{RWZTransactions}) yields that $\phi$ induces a well-defined map on $0$-th homology $$H_0(\phi): H_0(\mathbb{\Omega}(R[S])) \to R[\tau(S)].$$ This map is an isomorphism of algebras with inverse $\psi$ determined by $\psi(1_{R} [s_0(*)])=1_{R}$ 
and
$$\psi: \sigma \mapsto s^{-1}\overline{\sigma} + 1_{R}$$
for any non-degenerate $\sigma \in S_1$. The monoid bialgebra coproduct
\begin{eqnarray}\label{nablaS}
\nabla: R[\tau(S)] \to  R[\tau(S)]\otimes R[\tau(S)]
\end{eqnarray}
is determined by $\nabla(g) := g \otimes g$ for any $g \in \tau(S)$. A straightforward computation verifies that $H_0(\phi)$ intertwines the coproducts (\ref{nabla}) and (\ref{nablaS}), thus defining an isomorphism of bialgebras.  
\end{proof}

\smallskip

\noindent The following diagram summarizes the categories and functors discussed in this section.
\begin{equation*} \label{main_diagram} 
\begin{tikzcd}
 {\mathsf{Mon}} \\ {\mathsf{sSet}_0} & {\mathsf{sCoCoalg}^0_R} & & &\mathsf{Bialg_R} && \\
	{\mathsf{sSet}_0} & {\mathsf{sCoalg}^0_R} & {\mathsf{dgCoalg}^c_R} & {\mathsf{dgAlg}_R} & \mathsf{Alg}_R.
   \arrow["{R[-]}", from=2-1, to=2-2]
	\arrow[equals, from=2-1, to=3-1]
	\arrow["{R[-]}", from=3-1, to=3-2]
	\arrow[hook, from=2-2, to=3-2]
	\arrow["{\mathcal{N}_*}", from=3-2, to=3-3]
	\arrow["{\mathsf{Cobar}}", from=3-3, to=3-4]
    \arrow["\pi", from=2-2, to=2-5]
	\arrow["{\mathbb{\Omega}}", bend left=20, from=3-2, to=3-4]
    \arrow["H_0", from=3-4, to=3-5]
    \arrow["\text{forget}", from=2-5, to=3-5]
	\arrow["{\Lambda(-;R) = \mathbb{\Omega} \circ R[-]}"', bend right=24, from=3-1, to=3-4]
    \arrow["\tau", from=2-1, to=1-1]
    \arrow["{R[-]}", bend left=15, from=1-1, to=2-5]
\end{tikzcd}
\end{equation*}

\section{Review of some model categories}\label{modelcategories_sec}

In this section, we recall several known model categories that will be used throughout the article. We will assume basic knowledge of model category theory; standard references for the subject are \cite{Hir}, \cite{Ho}. For the theory of combinatorial model categories, see also \cite{Be}, \cite{D}, \cite[A.2.6]{LuHTT}, \cite{Ra_notes}. 

\subsection{The Joyal model structure on $\mathsf{sSet}$} Following \cite[1.1.5]{LuHTT}, we recall the definition of the functor
$$\mathfrak{C}\colon \mathsf{sSet} \to \mathsf{Cat}_{\mathsf{sSet}},$$ 
where $\mathsf{Cat}_{\mathsf{sSet}}$ denotes the category of (small) simplicial categories, i.e., categories enriched in (the cartesian monoidal category of) simplicial sets. For the standard $n$-simplex $\mathbb{\Delta}^n \in \mathsf{sSet}$, we define $\mathfrak{C}(\mathbb{\Delta}^n) \in  \mathsf{Cat}_{\mathsf{sSet}}$ by:
\begin{enumerate}
    \item $\text{Obj} \, \mathfrak{C}(\mathbb{\Delta}^n) = \lbrace 0, 1, \dots, n \rbrace$.
    \item If $i,j \in \{0, 1, \dots, n\}$, then $\mathfrak{C}(\mathbb{\Delta}^n)(i,j) = \begin{cases}
    \emptyset & \text{if }i > j \\
    N(P^n_{i,j}) \cong (\mathbb{\Delta}^1)^{ \times (j - i -1)} & \text{if } i < j\\
    \mathbb{\Delta}^0 & \text{if } i=j 
    \end{cases}$\\
    where $N$ denotes the nerve functor and  $P^n_{i,j}$ is the poset (regarded as category) of subsets $U \subseteq \{i,i+1,\dots, j\}$ with $i,j \in U$, ordered by inclusion of subsets. 
     \item For $0 \leq i_1 \leq i_2 \leq i_3 \leq n$, the composition:
    \begin{equation*}
        \mathfrak{C}(\mathbb{\Delta}^n)(i_2,i_3) \times \mathfrak{C}(\mathbb{\Delta}^n)(i_1,i_2) \to \mathfrak{C}(\mathbb{\Delta}^n)(i_1,i_3)
    \end{equation*}
    is induced by $P_{i_2,i_3} \times P_{i_1,i_2} \to P_{i_1, i_3}, \ (U, U') \mapsto U \cup U'$.
\end{enumerate}
The assigment $[n] \mapsto \mathfrak{C}(\mathbb{\Delta}^n)$ defines a cosimplicial object in $\mathsf{Cat}_{\mathsf{sSet}}$. Then the functor $\mathfrak{C} \colon \mathsf{sSet} \to \mathsf{Cat_{sSet}}$ is the (essentially) unique colimit-preserving extension of $\mathfrak{C}(\mathbb{\Delta}^{\bullet}) \colon \mathbb{\Delta} \to \mathsf{Cat_{sSet}}$, i.e., $\mathfrak{C}$ is given for every $S \in \mathsf{sSet}$ by
$$\mathfrak{C}(S)= \underset{\sigma\colon \mathbb{\Delta}^n \to S}{\text{colim}} \mathfrak{C}(\mathbb{\Delta}^n).$$
We have a functor $$\pi_0\colon \mathsf{Cat}_{\mathsf{sSet}} \to \mathsf{Cat}$$ given by applying the path components functor on each simplicial set of morphisms. For any $S \in \mathsf{sSet}$, $\pi_0\mathfrak{C}(S)$ recovers the \textit{homotopy category of $S$} -- this is the category with objects $S_0$ and morphisms generated by the set of $1$-simplices $S_1$ with relations given by the $2$-simplices of $S$ -- which defines the left adjoint to the usual nerve functor $\mathsf{N} \colon \mathsf{Cat} \to \mathsf{sSet}$.

We also recall that $|-|\colon \mathsf{sSet} \to \mathsf{Top}$ denotes the geometric realization functor from simplicial sets to topological spaces. A map of simplicial sets $f\colon S \to S'$ is a \textit{weak homotopy equivalence} if $|f|\colon |S| \to |S'|$ is a weak homotopy equivalence of topological spaces. 

 \begin{theorem} \cite[Section 6]{Joyal2}, \cite[Theorems 2.2.5.1 and 2.4.6.1]{LuHTT} \label{Jmodelstructure} There is a left proper combinatorial model category structure on $\mathsf{sSet}$ such that
 \begin{enumerate}
     \item a morphism $f\colon S \to S'$ is a weak equivalence if $\pi_0\mathfrak{C}(f)\colon \pi_0\mathfrak{C}(S) \to \pi_0\mathfrak{C}(S')$ is an essentially surjective functor and for every pair $x,y \in S_0$, the induced map $$\mathfrak{C}(f)\colon \mathfrak{C}(S)(x,y) \to \mathfrak{C}(S)(f(x), f(y))$$
     is a weak homotopy equivalence of simplicial sets;
     \item  a morphism is a cofibration if it is a monomorphism;
     \item the fibrant objects are the \textit{quasi-categories}, i.e., the simplicial sets $Q$ with the property that any map $f\colon \Lambda^n_k \to Q$, for $0 < k < n$, can be extended to $\widetilde{f}\colon \mathbb{\Delta}^n \to Q$.
 \end{enumerate}
\end{theorem}

This model category structure is called the \textit{Joyal model structure} on simplicial sets. We call a weak equivalence in the Joyal model structure a \textit{categorical equivalence}. There is also an induced model structure on $\mathsf{sSet}_0$, as shown in \cite[Lemma 3.2]{CHL}, where 
the weak equivalences and the cofibrations are defined in the same way.  

\subsection{The Kan--Quillen model structure on $\mathsf{sSet}$}

The following is a classical result due to Quillen \cite{Q67} (see also \cite{Ho}).

\begin{theorem} \label{KQmodelstructure} There is a proper combinatorial model category structure on $\mathsf{sSet}$ such that
\begin{enumerate}
    \item the weak equivalences are the weak homotopy equivalences;
    \item a morphism is a cofibration if it is a monomorphism;
    \item the fibrant objects are the \textit{Kan complexes}, i.e., the simplicial sets $K$ with the property that any map $f\colon \Lambda^n_k \to K$ can be extended to $\widetilde{f}\colon \mathbb{\Delta}^n \to K$ for every $0\leq k\leq n$.
\end{enumerate}
\end{theorem}
We call this model category structure the \textit{Kan--Quillen model structure} on simplicial sets. The Kan--Quillen model category is a model for the homotopy theory of spaces. There is an induced model structure on $\mathsf{sSet}_0$, which models connected pointed homotopy types, with the same weak equivalences and cofibrations (see, for example, \cite[Section V.6]{GJ}). 

The Kan--Quillen model structure is a left Bousfield localization of the Joyal model structure at the map $\mathbb{\Delta}^1\to \mathbb{\Delta}^0$. Thus, every categorical equivalence is a weak homotopy equivalence and every weak homotopy equivalence \textit{between Kan complexes} is a categorical equivalence. In fact, something slightly stronger also holds:
  
\begin{proposition}\label{Joyalsthm} 
Let $\mathcal{J}\colon \mathsf{sSet} \to \mathsf{sSet}$ denote a fibrant replacement functor in the Joyal model structure.  
\begin{enumerate}
\item If the homotopy category of $S \in \mathsf{sSet}$ is a groupoid, then $\mathcal{J}(S)$ is a Kan complex. 
\item If $f\colon S \to S'$ is a weak homotopy equivalence in $\mathsf{sSet}$ and the homotopy categories of $S$ and $S'$ are groupoids, then $f$ is a categorical equivalence. 
\end{enumerate}
\end{proposition}

\begin{proof} $(1)$ follows from \cite[Corollary 1.4]{Joyal} which states that a quasi-category is a Kan complex if and only if its homotopy category is a groupoid. $(2)$ follows from $(1)$ and from the fact that a weak homotopy equivalence between Kan complexes is a categorical equivalence.
\end{proof}

We can also describe the weak homotopy equivalences in terms of the functor $\mathfrak{C}$ as follows. Denote by $\mathsf{Gpd}_{\mathsf{sSet}}$ the category of \textit{simplicial groupoids}, i.e., simplicial objects in the category of groupoids $\mathsf{G}\colon \mathbb{\Delta}^{op} \to \mathsf{Gpd}$ with a constant simplicial set of objects.  
Let $$L\colon \mathsf{Cat}_{\mathsf{sSet}} \to \mathsf{Gpd}_{\mathsf{sSet}}$$ be the functor from simplicial categories to simplicial groupoids that formally inverts degreewise every morphism in a simplicial category, i.e.,  the left adjoint of the full and faithful embedding $\mathsf{Gpd}_{\mathsf{sSet}} \hookrightarrow \mathsf{Cat}_{\mathsf{sSet}}$. Then the weak homotopy equivalences may be described by localizing categorical equivalences as follows.

\begin{proposition} \cite{DK} \cite[Corollary 4.8]{MRZ}  A map $f\colon S \to S'$ in $\mathsf{sSet}$ is a weak homotopy equivalence if and only if $\pi_0L\mathfrak{C}(f)\colon \pi_0L\mathfrak{C}(S) \to \pi_0L\mathfrak{C}(S')$ is an essentially surjective functor of groupoids and for every pair $x,y \in S_0$, the induced map of simplicial sets
$$L\mathfrak{C}(f)\colon L\mathfrak{C}(S)(x,y) \to L\mathfrak{C}(S')(f(x), f(y))$$ is a weak homotopy equivalence. 
\end{proposition}

\subsection{The Bousfield model structure on $\mathsf{sSet}$}
The following model category is a special case of a well-known theorem due to Bousfield \cite{B75} (see also \cite{Hir, Be, Ra_notes}).
 
\begin{theorem}\label{Bousfieldmodelstructure}
Let $R$ be a commutative ring. There is a left proper combinatorial model category structure on $\mathsf{sSet}$ such that
\begin{enumerate}
\item a morphism $f\colon S \to S'$ is a weak equivalence if $$\mathcal{N}_*(R[f])\colon \mathcal{N}_*(R[S]) \to \mathcal{N}_*(R[S'])$$ is a quasi-isomorphism;
\item a morphism is a cofibration if it is a monomorphism;
\item a morphism is a fibration if it has the right lifting property with respect to the trivial cofibrations.
\end{enumerate}
\end{theorem}

We call the above model structure the \textit{Bousfield model structure} on simplicial sets and denote it by $(\mathsf{sSet}, \Req)$. We will call a weak equivalence in the Bousfield model structure an $R$-\textit{equivalence}. The Bousfield model structure is a left Bousfield localization of the Kan--Quillen model structure. A fibrant replacement in the Bousfield model structure is called an $R$-\textit{localization}. The Bousfield model structure induces also a model structure on $\mathsf{sSet}_0$, which is a left Bousfield localization of the corresponding Kan--Quillen model structure on $\mathsf{sSet}_0$. 

\subsection{Model structures on $\mathsf{sCoCoalg}_R$} \label{model_str_coalg_prelim}
Let $R$ be a commutative ring. Let $\kappa$ be a regular cardinal such that $\mathsf{CoCoalg}_R$ is locally $\kappa$-presentable. For example, this is satisfied if $\kappa > \mathrm{max}\{|R|, \aleph_0\}$ \cite{Ba, Ra} -- in fact, $\kappa = \aleph_1$ always suffices, see the recent preprint \cite{Po24}. Let $\I_{\kappa}$ denote the set of maps $i \colon A \to B$ in $\mathsf{sCoCoalg}_R$ between $\kappa$-presentable objects whose underlying map of simplicial $R$-modules is a monomorphism. The following result is a special case of the model category shown by the first-named author \cite{Ra} and generalizes the model category shown by Goerss \cite{Go} in the case of fields. 

\begin{theorem}\cite[Theorem A]{Ra} \label{modelcatqi} There is a left proper combinatorial model category structure on $\mathsf{sCoCoalg}_R$ such that
\begin{enumerate}
\item a morphism $f\colon C \to C'$ is a weak equivalence if $$\mathcal{N}_*(f)\colon \mathcal{N}_*(C) \to \mathcal{N}_*(C')$$ is a quasi-isomorphism;
\item the class of cofibrations $\cof(\I_{\kappa})$ is cofibrantly generated by the set $\I_{\kappa}$;
\item a morphism is a fibration if it has the right lifting property with respect to the trivial cofibrations.
\end{enumerate}
\end{theorem}

\noindent This model category will be denoted by $(\mathsf{sCoCoalg}_R, \qi)$. Strictly speaking, this model category depends on the choice of a sufficiently large regular cardinal $\kappa$. However, for any two such $\kappa < \kappa'$, the identity functor is a Quillen equivalence between the respective model categories. Thus, for simplicity, we have omitted $\kappa$ from the notation, even though some choice of $\kappa$ is implicitly assumed. We also refer to \cite{St} for a refinement of this model structure which applies to simplicial cocommutative \emph{flat} coalgebras over Prüfer domains.  

\smallskip

It is easy to see that the adjunction $R[-]\colon (\mathsf{sSet}, \Req) \rightleftarrows (\mathsf{sCoCoalg}_R, \qi) \colon \mathcal{P}$
is a Quillen adjunction between model categories. As shown in \cite{Ra}, this adjunction can also be used to induce a model structure on $\mathsf{sCoCoalg}_R$, transferred from the Kan--Quillen model structure on $\mathsf{sSet}$, which additionally makes the adjunction $(R[-], \mathcal{P})$ into a Quillen equivalence. 

\begin{theorem}\cite[Theorem 5.4]{Ra} \label{modelcatpts} Suppose that $R$ has no non-trivial idempotents. There is a proper combinatorial model category structure on $\mathsf{sCoCoalg}_R$ such that
\begin{enumerate}
\item a morphism $f\colon C \to C'$ is a weak equivalence if $\mathcal{P}(f)\colon \mathcal{P}(C) \to \mathcal{P}(C')$ is a weak homotopy equivalence;
\item a morphism $f\colon C \to C'$ is a fibration if $\mathcal{P}(f)\colon \mathcal{P}(C) \to \mathcal{P}(C')$ is a Kan fibration;
\item a morphism is a cofibration if it has the left lifting property with respect to the trivial fibrations; the class of cofibrations is cofibrantly generated by 
the set of morphisms $\{R[i_n] \ | \ i_n \colon \partial \mathbb{\Delta}^n \subseteq \mathbb{\Delta}^n, \ n \geq 0\}$.
\end{enumerate}
Moreover, with respect to this model structure on $\mathsf{sCoCoalg}_R$ and the Kan--Quillen model structure on $\mathsf{sSet}$,  the adjunction $$R[-]\colon \mathsf{sSet} \rightleftarrows \mathsf{sCoCoalg}_R \colon \mathcal{P}$$ is a Quillen equivalence.
\end{theorem}

\subsection{Model structures on $\mathsf{dgCat}_R$ and $\mathsf{dgAlg}_R$} \label{model_str_dgalg_prelim} First we recall the model category structure on the category $\mathsf{dgCat}_R$ of small dg categories over a commutative ring $R$ as shown by Tabuada \cite{Ta, Ta2} (see also \cite[1.3.1]{LuHA}). This model category will not be used in the present article, but it is closely related to a model category of dg algebras that will be important for later purposes. 

We denote by $H_0 \colon \mathsf{dgCat}_R \to \mathsf{Cat}_R$ the functor which is defined by applying $H_0$ to the morphism chain complexes, and $\mathsf{Cat}_R$ denotes the category of small categories enriched in $R$-modules. This is the homotopy category functor in the enriched context of dg categories.

\begin{theorem}\cite{Ta}
There is a right proper combinatorial model category structure on $\mathsf{dgCat}_R$ such that
\begin{enumerate} 
\item a morphism $F\colon \mathsf{C} \to \mathsf{C'}$ is a weak equivalence if the induced functor $$H_0(F)\colon H_0(\mathsf{C}) \to H_0(\mathsf{C'})$$ is essentially surjective, and for every pair of objects $x,y \in \mathsf{C}$ the induced map 
$$F\colon \mathsf{C}(x,y) \to \mathsf{C'}(F(x), F(y))$$ is a quasi-isomorphism of chain complexes;

\item a morphism $F\colon \mathsf{C} \to \mathsf{C'}$ is a fibration if for every pair of objects $x,y \in \mathsf{C}$, the chain map $F\colon \mathsf{C}(x,y) \to \mathsf{C'}(F(x), F(y))$ is degreewise surjective; moreover, for every isomorphism $F(x) \to z$ in $H_0(\mathsf{C'})$, there is a lift to an isomorphism $x \to y$ in $H_0(\mathsf{C})$. 
\end{enumerate}
\end{theorem}

\noindent We refer to \cite{Hol} for the question of left properness for this model category. 

This model category structure induces a model category structure on the category of dg algebras, where the latter are considered as dg categories with one object. The existence of this model category structure was shown (independently of dg categories) by Jardine \cite{Ja} (under different grading conventions) and it is also a special case of more general results obtained by Hinich \cite{Hin}. See also \cite[Section 9.1]{Po11} for another detailed account. 

\begin{theorem} \cite{Hin, Ja} \label{Pmodelstructure} There is a right proper combinatorial model structure on $\mathsf{dgAlg}_R$ such that
\begin{enumerate}
\item a morphism $f \colon (A, d) \to (A', d')$ is a weak equivalence if $f$ is a quasi-isomorphism of the underlying chain complexes;
\item a morphism $f \colon (A, d) \to (A', d')$ a fibration if $f$ is degreewise surjective.
\end{enumerate}
\end{theorem}

We will refer to this model category as the \textit{projective model structure} on dg algebras. It is easy to see that this model category is finitely combinatorial using the following descriptions of the cofibrations and trivial cofibrations. Let $T_n(x)$ denote the free dg algebra generated by an element $x$ in degree $n$, and $S_{n}(x)$ the free graded $R$-algebra $R\langle x \rangle$ generated by an element $x$ in degree $n$ and equipped with the trivial differential. Then the maps 
$$R \to T_n(x), \ n \in \Z,$$
generate the trivial cofibrations, and the maps
$$R\to T_n(x), \ R \to S_n(x), \ S_{n-1}(dx) \to T_n(x), \ n \in \Z,$$
generate the cofibrations. We refer to \cite[Remark 2.15]{BCL} and \cite[Section 2.4]{Re} for the question of left properness for this model category. 

Note that a map of dg algebras is a trivial fibration (resp. weak equivalence) in $\mathsf{dgAlg}_R$ if and only if it is so in $\mathsf{dgCat}_R$. As a consequence, the cofibrations in $\mathsf{dgAlg}_R$ are precisely the maps which are 
cofibrations in $\mathsf{dgCat}_R$. (The situation is analogous to the model category structures on $\mathsf{sSet}_0$ 
and $\mathsf{sSet}$.)

\begin{remark}[simplicial sets and dg categories] \label{joyalanddgalg}
 By \cite[Proposition 1.3.1.20]{LuHA} (see also \cite{Ta2}), the differential graded nerve functor $\mathsf{N}_{\text{dg}} \colon \mathsf{dgCat}_R \to \mathsf{sSet}$ is a right Quillen functor between the model category of Tabuada and the Joyal model category. It follows from the previous observations that the induced adjunction (Theorem \ref{leftadjoint})
 $$\Lambda( - ; R)\colon \mathsf{sSet}_0 \rightleftarrows \mathsf{dgAlg}_R\colon \mathsf{N}_{\text{dg}},$$ 
is again a Quillen adjunction where the model structure on $\mathsf{sSet}_0$ is induced by the Joyal model structure and the one on $\mathsf{dgAlg}_R$ is induced by the model structure on $\mathsf{dgCat}_R$ (as stated in Theorem~\ref{Pmodelstructure}). 
In particular, the map $$\Lambda(f; R) \colon \Lambda(S; R) \to \Lambda(S'; R)$$ is a quasi-isomorphism of dg algebras for every categorical equivalence $f \colon S \to S'$.
\end{remark} 

Lastly, we briefly mention the interaction between the homotopy theory of dg algebras and the cobar construction (Subsection \ref{cobar_sec}).

\begin{theorem}\label{barcobar} The counit transformation of the adjunction
$$\mathsf{Cobar} \colon \mathsf{dgCoalg}_R^c \rightleftarrows \mathsf{dgAlg}_R^{a} \colon \mathsf{Bar}
$$
is an objectwise quasi-isomorphism. 
\end{theorem}
\begin{proof}
See \cite[II.4]{HMS74}, \cite[Corollary 2.15]{Mu}. 
\end{proof}

We note that the functor $\mathsf{Cobar}$ does \textit{not} preserve quasi-isomorphisms in general. (We will see an important manifestation of this point in Section \ref{sec_3_weq}.)
On the other hand, this adjunction becomes a derived equivalence if we localize the category $\mathsf{dgCoalg}_R^c$ at the class of weak equivalences determined by $\mathsf{Cobar}$ (\emph{derived Koszul duality}). Indeed, it follows formally that the components of the unit transformation of the adjunction are sent to quasi-isomorphisms of dg algebras after applying $\mathsf{Cobar}$.

\begin{remark}[derived Koszul duality via model categories]\label{derived_koszul}For a suitable model category structure on $\mathsf{dgCoal}_R^c$ (where $R$ is a field), the adjunction 
$$\mathsf{Cobar} \colon \mathsf{dgCoalg}_R^c \rightleftarrows \mathsf{dgAlg}_R^{a} \colon \mathsf{Bar}
$$
becomes a Quillen equivalence (see \cite{Po11}). The model category structure on $\mathsf{dgAlg}_R^a$ is induced from the one of Theorem \ref{Pmodelstructure}. The cofibrations in $\mathsf{dgCoalg}_R^c$ are the injective maps; in particular, every object in $\mathsf{dgCoalg}_R^c$ is cofibrant. We emphasize that the notion of weak equivalence in  $\mathsf{dgCoalg}_R^c$ used for this Quillen equivalence is detected by the cobar functor and is strictly contained in the class of quasi-isomorphisms of conilpotent dg coalgebras. It follows that the (derived) unit transformation of the adjunction is an objectwise quasi-isomorphism.
\end{remark}

\section{Homotopical localization}\label{loc_sec}

In this section we describe explicit and functorial models for certain localizations in the contexts of reduced simplicial sets and of differential graded algebras. 

\subsection{Simplicial localization} 
Denote by $S^1 : = \mathbb{\Delta}^1/\partial \mathbb{\Delta}^1 \in \mathsf{sSet}_0$ the reduced simplicial set with exactly two non-degenerate simplices, one in each of the dimensions $0$ and $1$. The homotopy category of $S^1$ consists of a single object and a single non-identity endomorphism generates freely its endomorphisms. 
  
\begin{definition}\label{localization} A \textit{localization} (or \emph{weak group completion}) of $S^1$ is a map of reduced simplicial sets $$\iota\colon S^1 \hookrightarrow \X$$ which is a trivial cofibration in the Kan--Quillen model structure and the homotopy category of $\X$ is a groupoid (with one object). 
\end{definition}

\begin{example} \label{localization_examples}
The following are examples of localizations of $S^1$.
\begin{enumerate} 
\item Any fibrant replacement of $S^1$ in the Kan--Quillen model structure for $\mathsf{sSet}_0$. 
\item Let $J = (\bullet \rightleftarrows \bullet)$ denote the connected groupoid with two objects and no non-trivial automorphisms and let $(\sigma \colon \mathbb{\Delta}^1 \to N(J))$ be a non-degenerate $1$-simplex. For $(\mathbb{\Delta}^1 \to S^1)$ the non-degenerate $1$-simplex of $S^1$, we consider the pushout $\X = S^1 \cup_{\mathbb{\Delta}^1} N(J)$. Then the canonical map $\iota\colon S^1 \hookrightarrow \X$ is a localization of $S^1$. Note that $\X$ is not a quasi-category in this case (for example, there is no composition ``$[\sigma] \circ [\sigma]$'' in $\X$).
\end{enumerate}
\end{example}

Let $\mathsf{sSet}_0^+$ denote the category of \textit{marked reduced simplicial sets}. The objects of $\mathsf{sSet}_0^+$ are pairs $(S,W)$ where $S \in \mathsf{sSet}_0$ and $W \subseteq S_1$ is a subset of the $1$-simplices in $S$. A morphism $f\colon (S,W) \to (S',W')$ is a map of simplicial sets $f\colon S \to S'$ such that $f(W) \subseteq W'$. Given $(S,W) \in \mathsf{sSet}_0^+$, we will often identify the subset $W \subseteq S_1$ with the associated map $\underset{W}{\bigvee} S^1 \to S.$

\begin{definition} \label{sim_loc_def} Let $\iota\colon S^1 \hookrightarrow \X$ be a localization of $S^1$. The  \textit{simplicial localization functor} (or \emph{weak group completion functor}) \textit{with respect to $\iota\colon S^1 \hookrightarrow \X$}, 
$$\mathcal{K}_{\iota}\colon \mathsf{sSet}_0^+ \to \mathsf{sSet}_0,$$ 
is defined for every $(S, W) \in \mathsf{sSet}_0^+$ by the following pushout diagram of reduced simplicial sets
\[\begin{tikzcd}
	{\underset{W}{\bigvee}S^1} & S \\
	{\underset{W}{\bigvee}\X} & {\mathcal{K}_{\iota}(S,W)}
	\arrow[hook, from=1-1, to=1-2]
	\arrow[hook, from=1-2, to=2-2]
	\arrow[hook, from=2-1, to=2-2]
	\arrow["{\underset{W}{\vee}\iota}"', hook, from=1-1, to=2-1].
\end{tikzcd}\]
\end{definition}

Note that $\mathcal{K}_{\iota}(S,W)$ is a homotopy pushout in the Joyal model structure. The following proposition shows that the functor $\mathcal{K}_{\iota}$ is essentially independent of the choice of the localization $\iota$. 

\begin{proposition} \label{inftyloc} Let $\iota\colon S^1 \hookrightarrow \X$ and $\iota'\colon S^1 \to \X'$ be two localizations of $S^1$. Then, for any $(S,W) \in \mathsf{sSet}_0^+$, the reduced simplicial sets $\mathcal{K}_{\iota}(S,W)$ and $\mathcal{K}_{\iota'}(S,W) $ are connected by a natural zigzag of categorical equivalences under $S$.

In particular, the reduced simplicial sets $\X$ and $\X'$ are connected by a natural zigzag of categorical equivalences under $S^1$. The homotopy category of $\X$ is isomorphic to $\Z$ (as groups).
\end{proposition}
\begin{proof}
We consider the pushout
\[\begin{tikzcd}
	{S^1} & {\X'}  \\
	{\X} &  {\X \underset{S^1}{\cup} \X'.} 
	\arrow["\iota"', hook, from=1-1, to=2-1]
	\arrow["\iota' ", hook, from=1-1, to=1-2]
	\arrow[hook, from=1-2, to=2-2]
	\arrow[hook, from=2-1, to=2-2]
\end{tikzcd}\]
Every map in the diagram is a weak homotopy equivalence (and a monomorphism). The homotopy categories of the simplicial sets $\X, \X'$ and $\X \cup_{S^1} \X'$ are groups. In particular, the composite map $S^1 \to \X \cup_{S^1} \X'$ is again a localization of $S^1$. By Proposition \ref{Joyalsthm}(2), it follows that the natural zigzag of maps 
$$\X \to \X \cup_{S^1} \X' \leftarrow \X'$$
consists of categorical equivalences under $S^1$. The identification of their homotopy category follows from Example \ref{localization_examples}(2). Moreover, it follows that $\mathcal{K}_{\iota}(S,W)$ and $\mathcal{K}_{\iota'}(S,W)$ are connected by a natural zigzag of categorical equivalences under $S$, since they are defined by homotopy pushouts in the Joyal model structure. 
\end{proof}

\subsection{Localization of dg algebras} \label{derivedloc_subsection}


For a commutative ring $R$, we define the category $\mathsf{dgAlg}_R^{+}$ of \textit{marked dg algebras} as follows. The objects are pairs $(A,P)$, where $A$ is a dg algebra and $P \subseteq A_0$ is a set of $0$-cycles. The morphisms $f\colon (A,P) \to (A',P')$ are given by morphisms of dg algebras $f\colon A\to A'$ such that $f(P)\subseteq P'$. 

A simplicial localization of $S^1$ gives rise to a notion of localization for marked dg algebras as follows. Note that for an indexing set $P$ the dg algebra 
$$\Lambda (\underbrace{S^1 \vee \ldots \vee S^1}_{P};R)$$ 
is the free associative $R$-algebra $R \langle P \rangle$ generated by the set $P$ and equipped with the trivial differential.

\begin{definition}
\label{derivedlocalization_def} 

 Let $\iota\colon S^1 \hookrightarrow \X$ be a localization of $S^1$. The \textit{localization functor with respect to $\iota\colon S^1 \hookrightarrow \X$},  
\[\mathcal{L}_{\iota}\colon \mathsf{dgAlg}_R^{+} \to \mathsf{dgAlg}_R,\]
is defined for any $(A, P) \in  \mathsf{dgAlg}_R^{+} $ by the pushout diagram of dg algebras
\[\begin{tikzcd}
	{ \Lambda (\underset{P}{\bigvee}S^1;R)} & A\\
	{\Lambda (\underset{P}{\bigvee}\X;R)} & {\mathcal{L}_{\iota}(A,P)}
	\arrow[from=1-1, to=1-2]
	\arrow[from=1-2, to=2-2]
	\arrow[from=2-1, to=2-2]
 	\arrow[hook, from=1-1, to=2-1],
	\end{tikzcd}\]
where the top horizontal map $R \langle P \rangle \to A$ is induced by the inclusion $P \hookrightarrow A $ and the left vertical map is induced by $\iota \colon S^1 \to \X$.  We also write
\[\mathcal{L}_{\iota}(A,P)= \Lambda (\underbrace{\X \vee \ldots \vee \X}_P ;R) \underset{R\langle P \rangle }{\circledast}A,\] 
 where $A \underset{B}\circledast C$ denotes the pushout of a diagram of dg algebras $(A \leftarrow B \to C)$. 
\end{definition}

\begin{remark} \label{derivedlocalization} Whenever $A$ is \textit{left proper} as a dg algebra, for instance, when $A$ is cofibrant as a dg algebra or flat as an $R$-module (see \cite[Definition 2.5 and Theorem 2.14]{BCL}), the dg algebra $\mathcal{L}_{\iota}(A,P)$ is a model for the \textit{derived localization} of $A$ at $P$ as defined in \cite[Section 3]{BCL}. This means that $\mathcal{L}_{\iota}(A,P)$ is a model for the homotopy pushout \[A \circledast^{\mathbb{L}}_{R\langle P \rangle} R\langle P,P^{-1}\rangle,\] where $\mathsf{dgAlg}_R$ is equipped with the projective model structure. Here $R\langle P,P^{-1}\rangle$ denotes the algebra generated by the symbols $\{p,p^{-1} : p \in P\}$ modulo the ideal generated by the relations $pp^{-1}=1=p^{-1}p.$ In fact, the dg algebra map 
\[ R\langle P \rangle = \Lambda (\underbrace{S^1 \vee \ldots \vee S^1}_P;R) \to \Lambda (\underbrace{\X \vee \ldots \vee \X}_P;R)\] 
is a cofibrant replacement of $R\langle P\rangle \to R\langle P,P^{-1}\rangle$ in the induced model structure on the category $R\langle P \rangle \downarrow \mathsf{dgAlg}_R$. Thus, assuming that $A$ is a left proper dg algebra, it is enough to make this replacement to compute the desired homotopy pushout, see \cite[Remark 3.11]{BCL}. We also refer to \cite{CHL} for interesting connections between the derived localization of dg algebras and the homotopy theory of monoids.
\end{remark}

As in the case of reduced simplicial sets, the functor $\mathcal{L}_{\iota}$ is independent of the choice of localization $\iota$ in the following sense. 

\begin{proposition} \label{independenceofloc} Let $\iota\colon S^1 \hookrightarrow \X$ and $\iota'\colon S^1 \hookrightarrow \X'$ be two simplicial localizations of $S^1$ and let $(A,P) \in \mathsf{dgAlg}_R^{+}$, where $A$ is a left proper dg algebra. Then the dg algebras  $\mathcal{L}_{\iota}(A,P)$ and $\mathcal{L}_{\iota'}(A,P)$ are connected by a natural zigzag of quasi-isomorphisms under $A$. 
 \end{proposition}
 
 \begin{proof} By Proposition \ref{inftyloc}, we know that $\X$ and $\X'$ are naturally categorically equivalent under $S^1$, so $\underset{P}{\bigvee}\X$ and $\underset{P}{\bigvee}\X'$ are also connected by a natural zigzag of categorical equivalences under $\underset{P}{\bigvee} S^1$. This induces a natural zigzag of quasi-isomorphisms of dg algebras upon applying $\Lambda( - ; R)= \mathbb{\Omega} \circ R[-]$ (Remark \ref{joyalanddgalg}). Then the result follows from the fact that, assuming $A$ is left proper, the square of Definition \ref{derivedlocalization_def}  defines a homotopy pushout of dg algebras (see also Remark \ref{derivedlocalization}). 
 \end{proof}

Given a set of $0$-cycles $P$ in $A$, we denote by $[P]=\{[p] \in H_0(A) : p \in P\}$ the set of the homology classes of its elements. As the next proposition shows, the localization at the set $P$ of $0$-cycles only depends on $[P] \subseteq H_0(A)$.

\begin{proposition}\label{invariance}
Let $\iota \colon S^1 \to \X$ be a simplicial localization of $S^1$, $A$ a left proper dg algebra, and let $P, P' \subseteq A_0$ be two sets of $0$-cycles. If $[P] = [P']$, then the dg algebras $\mathcal{L}_{\iota}(A,P)$ and $\mathcal{L}_{\iota}(A,P')$ are canonically quasi-isomorphic.
\end{proposition}
\begin{proof} This follows from \cite[Theorem 3.10 and Remark 3.11]{BCL}.
\end{proof} 

Given a set $W \subseteq S_1$ of $1$-simplices in a simplicial set $S$, we consider an associated set of $0$-cycles:
\begin{eqnarray}\label{+1}
\overline{W}=\{ s^{-1}\overline{\sigma} + 1 \in \Lambda(S;R)_0 : \sigma \in W\},
\end{eqnarray}
 where $\overline{\sigma}$ denotes the class of $\sigma$ in the cokernel of the coaugmentation $R \to \mathcal{N}_*(C)$ and $s^{-1}$ the degree shift by $-1$ used in the definition of the cobar functor.
The functors $\mathcal{L}_{\iota}$ and $\mathcal{K}_{\iota}$ are compatible in the following way. 

\begin{proposition}\label{localizations}
For any $(S,W) \in \mathsf{sSet}_0^+$, there is a natural isomorphism of dg algebras $\mathcal{L}_{\iota} (\Lambda(S;R), \overline{W}) \cong \Lambda(\mathcal{K}_{\iota} (S,W);R).$
\end{proposition}
\begin{proof} 
Since $\Lambda(-;R)$ preserves pushouts (Theorem \ref{leftadjoint}), it sends the pushout square of Definition \ref{sim_loc_def}
\[\begin{tikzcd}
	{\underset{W}{\bigvee}S^1} & S \\
	{\underset{W}{\bigvee}\X} & {\mathcal{K}_{\iota}(S,W)}
	\arrow[from=1-1, to=1-2]
	\arrow[hook, from=1-2, to=2-2]
	\arrow[from=2-1, to=2-2]
	\arrow[hook, from=1-1, to=2-1]
\end{tikzcd}\]
to a pushout square of dg algebras
\[\begin{tikzcd}
	{ \Lambda (\underset{W}{\bigvee}S^1;R)} & \Lambda(S; R) \\
	{\Lambda (\underset{W}{\bigvee}\X;R)} & {\Lambda\big(\mathcal{K}_{\iota}(S,W); R\big).}	
 \arrow[from=1-1, to=1-2]
	\arrow[from=1-2, to=2-2]
	\arrow[from=2-1, to=2-2]
 	\arrow[hook, from=1-1, to=2-1]
	\end{tikzcd}\]
Using Proposition \ref{fundamentalbialgebra}, the top horizontal map of dg algebras  
$$R \langle W \rangle =\Lambda (\underbrace{S^1 \vee \ldots \vee S^1}_W;R) \to \Lambda(S; R)$$
is induced by the map (of underlying sets) $W \to \Lambda(S; R), \ \sigma \mapsto s^{-1}\overline{\sigma} + 1_R$. Thus, the last pushout square is identified with the pushout square for the localization $\mathcal{L}_{\iota} (\Lambda(S;R), \overline{W})$ (Definition \ref{derivedlocalization_def}) and the claim follows. 
\end{proof}

\subsection{The localized cobar construction} We define a functor $\mathfrak{X}\colon \mathsf{sCoCoalg}^0_R \to \mathsf{dgAlg}_R^{+}$ that ``marks" the cobar construction of a simplicial cocommutative coalgebra by a functorially defined set of $0$-cycles. Given $C \in \mathsf{sCoCoalg}^0_R$, we define 
\[ 
P_C = \{ a \in \mathbb{\Omega}(C)_0 : H_0(\nabla)([a])=[a] \otimes [a], \ \epsilon([a])=1 \},\]
where $[a] \in H_0(\mathbb{\Omega}(C))$ denotes the homology class of the $0$-cycle $a \in \mathbb{\Omega}(C)_0$ and $\epsilon$ is the counit of the fundamental bialgebra of $C$. In other words, $P_C$ consists of all $0$-cycles representing \emph{monoid-like elements} of the fundamental bialgebra $H_0(\mathbb{\Omega}(C))$ of $C$. Since the fundamental bialgebra construction is functorial, $C \mapsto P_C$ defines a functor as well. Therefore, we obtain a functor
\[ \mathfrak{X} \colon \mathsf{sCoCoalg}^0_R \to \mathsf{dgAlg}_R^{+}, \ \
\mathfrak{X}(C)=( \mathbb{\Omega}(C), P_C).\]


\noindent The next proposition describes the localization of $\mathfrak{X}(C)$ in the case where $C = R[S]$. We denote by $(-)^{\sharp} \colon \mathsf{sSet}_0 \to \mathsf{sSet}_0^{+}$ the functor given by the ``maximal" marking, i.e., $S^{\sharp}=(S,S_1).$

\begin{proposition} \label{localizations2}
Let $\iota \colon S^1 \to \X$ be a simplicial localization, $S$ a reduced simplicial set in $\mathsf{sSet}_0$, and assume that $R$ has no non-trivial idempotents. Then the dg algebras $\Lambda(\mathcal{K}_{\iota}(S^{\sharp});R)$ and $\mathcal{L}_{\iota}(\mathfrak{X}(R[S]))$ are connected by a natural zigzag of quasi-isomorphisms. 
\end{proposition}
\begin{proof} 
By Proposition \ref{localizations}, there is a natural isomorphism of dg algebras
\begin{eqnarray}\label{iso1}
\Lambda(\mathcal{K}_{\iota}(S, S_1);R)\cong \mathcal{L}_{\iota}(\Lambda(S;R),\overline{S_1}),
\end{eqnarray}
where $S_1$ is the set of $1$-simplices in $S$ and $\overline{S_1}$ is defined as in \eqref{+1} above. For any set $P$, let $F(P)$ denote the free monoid generated by the set $P$. The monoid $F(\overline{S_1})$ can be considered as a subset (monomials) of $\Lambda(S;R)_0$. We now argue that there is a natural quasi-isomorphism of dg algebras
\begin{eqnarray}\label{iso2}
\mathcal{L}_{\iota}(\Lambda(S;R),F(\overline{S_1})) \simeq \mathcal{L}_{\iota}(\Lambda(S;R),\overline{S_1}).
\end{eqnarray}
In fact, for any set $P$ there is a natural projection map 
\begin{eqnarray}\label{projection}
R \langle F(P) \rangle \to R \langle P \rangle
\end{eqnarray}
that is induced by the map (of underlying monoids) $F(P) \to R\langle P \rangle$ and $R\langle F(P) \rangle$ is the free associative $R$-algebra generated by (the underlying set) $F(P)$. This map makes $R \langle P \rangle$ into an $R \langle F(P) \rangle$-algebra. It follows from \cite[Lemma 3.7]{BCL} that there is a canonical quasi-isomorphism of homotopy pushouts (cf. Remark \ref{derivedlocalization})
\[ R \langle P \rangle \underset{R \langle F(P) \rangle}{\circledast^{\mathbb{L}}}   R \langle F(P), F(P)^{-1} \rangle \xrightarrow{\simeq} R \langle P \rangle \underset{R\langle P \rangle}{\circledast^{\mathbb{L}}} R \langle P, P^{-1} \rangle \cong R \langle P, P^{-1} \rangle. \]
This quasi-isomorphism may be realized by the canonical map of dg algebras (induced by \eqref{projection}) 
\[ 
R \langle P \rangle \underset{R \langle F(P) \rangle}{\circledast} \Lambda(\underbrace{\X \vee \ldots \X}_{F(P)}; R) \xrightarrow{\simeq} \Lambda(\underbrace{ \X \vee \ldots \X}_{P}; R)
\]
that model the respective homotopy pushouts. As a consequence, for any set $P$ of $0$-cycles in a cofibrant dg algebra $A$, the map \eqref{projection} induces a canonical quasi-isomorphism of dg algebras
\[ 
A \underset{{R\langle F(P) \rangle}}{\circledast} \Lambda(\underbrace{\X \vee \ldots \vee \X}_{F(P)};R) \xrightarrow{\simeq} A \underset{{R\langle P \rangle}}{\circledast}\Lambda(\underbrace{\X \vee \ldots \vee \X}_P;R).
\]
For $A=\Lambda(S;R)$ and $P=\overline{S_1}$, we obtain the desired quasi-isomorphism \eqref{iso2}.
Finally, we describe a natural quasi-isomorphism
\begin{eqnarray}\label{qiso}
\mathcal{L}_{\iota}(\Lambda(S;R),F(\overline{S_1})) \xrightarrow{\simeq} \mathcal{L}_{\iota}( \mathfrak{X}(R[S]))= \mathcal{L}_{\iota}(\Lambda(S;R),P_{R[S]})
\end{eqnarray} 
which together with \eqref{iso1} and \eqref{iso2} yield the desired conclusion. By Proposition \ref{fundamentalbialgebra}, the fundamental bialgebra $H_0(\Lambda(S;R))$ of $R[S]$ is isomorphic to $R[\tau(S)]$. The set $P_{R[S]}$ consists of all $0$-cycles in $\Lambda(S;R)_0$ representing the monoid-like elements of the bialgebra $H_0(\Lambda(S;R)) \cong R[\tau(S)]$. 
Since $R$ has no non-trivial idempotents, the monoid of monoid-like elements in the bialgebra $R[\tau(S)]$ is identified with (the monoid) $\tau(S)$. A canonical set of representatives for the elements in $\tau(S) \subseteq R[\tau(S)] \cong H_0(\Lambda(S;R))$ is given by $F(\overline{S_1}) \subseteq P_{R[S]}$. Then it follows from Proposition \ref{invariance} that the natural map \eqref{qiso} is a quasi-isomorphism and this completes the proof. 
\end{proof}

\begin{definition}\label{localizedcobardef}
We fix a simplicial localization and denote it by  $\iota \colon S^1 \to \mathbb{S}^1$. We define the \textit{localized cobar construction} to be the composition of functors
\[
\widehat{\mathbb{\Omega}} = \mathcal{L}_{\iota} \circ \mathfrak{X}\colon 
\mathsf{sCoCoalg}^0_R \to \mathsf{dgAlg}_R.\]
\end{definition}
\noindent Unraveling the definition, we have for any $C \in \mathsf{sCoCoalg}^0_R$
\begin{eqnarray}\label{pushout}
\widehat{\mathbb{\Omega}}(C)=
\mathbb{\Omega}(C) \underset{R\langle P_C \rangle}{\circledast} \mathbb{\Omega}(R[\underbrace{\mathbb{S}^1 \vee \ldots \vee \mathbb{S}^1}_{P_C}]).
\end{eqnarray}

\begin{remark}
A similar construction was proposed in \cite[Section 1.2]{HT10} under the name of \textit{extended cobar construction} in the context of dg coalgebras. However, the extended cobar construction is not functorial as defined, since it depends on a basis for the degree $1$ summand of the underlying graded $R$-module. On the other hand, for any simplicial cocommutative coalgebra $C$, we may obtain the desired set $P_C$ of degree $0$ elements in the cobar construction functorially through the monoid-like elements of the fundamental bialgebra $\pi(C)$.
\end{remark}

\section{Three notions of weak equivalence} \label{sec_3_weq}

In this section, we study the properties of the following three different notions of weak equivalence in $\mathsf{sCoCoalg}^0_R$.
\begin{itemize}
\item[(1)] The standard class $\W_{\qi}$ of \emph{quasi-isomorphisms} in $\mathsf{sCoCoalg}^0_R$ (or $\mathsf{sCoCoalg}_R$). The homotopy theory of simplicial $R$-coalgebras relative to $\W_{\qi}$ and its connection with the $R$-local homotopy theory of spaces have been studied in \cite{Go} (for fields) and in \cite{Ra} (for general presheaves of commutative rings) -- see Subsection \ref{model_str_coalg_prelim}.
\item[(2)] The class $\W_{\mathbb{\Omega}}$ of $\mathbb{\Omega}$-quasi-isomorphisms (or \emph{cobar quasi-isomorphisms}) in $\mathsf{sCoCoalg}^0_R$. 
This is a smaller class of weak equivalences than $\W_{\qi}$ and its properties have also been studied in \cite{RWZTransactions}.
\item[(3)] The class $\W_{\widehat{\mathbb{\Omega}}}$ of $\widehat{\mathbb{\Omega}}$-quasi-isomorphisms in $\mathsf{sCoCoalg}_R^0$ (or \emph{localized cobar quasi-isomorphisms}). This class of weak equivalences lies strictly between $\W_{\mathbb{\Omega}}$ and $\W_{\qi}$. This class determines a homotopy theory in $\mathsf{sCoCoalg}^0_R$, which is suitable for the comparison with the homotopy theory of reduced simplicial sets. 
\end{itemize}
We also relate each of the above notions to their corresponding counterparts for $\mathsf{sSet}_0$. 

\subsection{Quasi-isomorphisms} The following definitions are well known. 

\begin{definition} A map $f\colon C \to C'$ in $\mathsf{sCoCoalg}_R$ is a \textit{quasi-isomorphism} if the induced map of dg coalgebras $\mathcal{N}_*(f)\colon \mathcal{N}_*(C) \to \mathcal{N}_*(C')$ is a quasi-isomorphism. We denote the class of quasi-isomorphisms in $\mathsf{sCoCoalg}_R$ (or $\mathsf{sCoCoalg}^0_R$) by $\W_{\qi}$. 
\end{definition}

\begin{definition} A map $f\colon X \to X'$ in $\mathsf{sSet}$ is an $R$-\textit{equivalence}  if it induces an isomorphism $H_*(f;R)\colon H_*(X; R) \xrightarrow{\cong} H_*(X';R)$ on homology with $R$-coefficients. We denote the class of $R$-equivalences in $\mathsf{sSet}$ (or $\mathsf{sSet}_0$) by $\W_R$.
\end{definition}

We record the following obvious statement in order to emphasize the parallelism with analogous statements in the next sections. 

\begin{proposition} \label{qiandqi} A map $f\colon X \to X'$ in $\mathsf{sSet}$ is an $R$-equivalence if and only if $R[f]\colon R[X] \to R[X']$ is a quasi-isomorphism in $\mathsf{sCoCoalg}_R$.  
\end{proposition}

\subsection{$\mathbb{\Omega}$-quasi-isomorphisms} 
Recall the notation $\mathbb{\Omega}= \mathsf{Cobar} \circ \mathcal{N}_*$. 

\begin{definition}
A map $f\colon C \to C'$ in $\mathsf{sCoCoalg}^0_R$ is an $\mathbb{\Omega}$-\textit{quasi-isomorphism} if the map of dg algebras $\mathbb{\Omega}(f)\colon \mathbb{\Omega} (C) \to \mathbb{\Omega} (C')$ is a quasi-isomorphism. We denote by $\mathcal{W}_{\mathbb{\Omega}}$ the class of $\mathbb{\Omega}$-quasi-isomorphisms of connected simplicial cocommutative $R$-coalgebras. 
\end{definition}

This notion of weak equivalence is well studied in the context of dg coalgebras \cite{LH03} and has been used extensively in (derived) Koszul duality \cite{Po11}, \cite{LV12} (see also Subsection \ref{model_str_dgalg_prelim}). A first observation is that this notion is strictly stronger than quasi-isomorphism (see also Proposition \ref{inclusions} below).

\begin{proposition}\label{omegaqiandqi}  Let $C$ and $C'$ be connected simplicial cocommutative flat $R$-coalgebras. If $f\colon C\to C'$ is an $\mathbb{\Omega}$-quasi-isomorphism, then $f\colon C \to C'$ is a quasi-isomorphism, but the converse fails in general.
\end{proposition}
\begin{proof}
This is shown for fields in \cite[Proposition 2.4.2]{LV12}. The same proof applies here, too.
\end{proof}

The notion of $\mathbb{\Omega}$-quasi-isomorphism is related to the following linearized version of categorical equivalences between reduced simplicial sets.

\begin{definition} A map $f\colon S \to S'$ in $\mathsf{sSet}_0$ is a \textit{categorical $R$-equivalence} if the induced map of dg algebras $\Lambda(f;R)\colon \Lambda(S;R) \to \Lambda(S';R)$ is a quasi-isomorphism. We denote the class of categorical $R$-equivalences in $\mathsf{sSet}_0$ by $\W_{J, R}$ (the subscript ``$J$" stands for ``Joyal"). 
\end{definition} 

By Remark \ref{joyalanddgalg}, every categorical equivalence in $\mathsf{sSet}_0$ is also a categorical $R$-equivalence. The following is obvious from the definitions of the respective classes of weak equivalences. 

\begin{proposition} \label{RcateqandOmegaqi} A map $f\colon S \to S'$ in $\mathsf{sSet}_0$ is a categorical $R$-equivalence if and only if $R[f]\colon R[S] \to R[S']$ is an $\mathbb{\Omega}$-quasi-isomorphism of connected simplicial cocommutative $R$-coalgebras. In particular, if $f\colon S\to S'$ is a categorical equivalence then $R[f]\colon R[S] \to R[S']$ is an $\mathbb{\Omega}$-quasi-isomorphism. 
\end{proposition} 

\begin{remark}\label{nscAdams2}
Given $S \in \mathsf{sSet}_0$, we denote by $C_*(\Omega |S|;R)$ the dg algebra of normalized singular chains on the topological monoid of based (Moore) loops in $|S|$. For any simplicial localization $S \to \mathcal{K}_{\iota}(S^{\sharp})$, we have a natural quasi-isomorphism of dg algebras
$$\Lambda(\mathcal{K}_{\iota}(S^{\sharp});R) \simeq C_*(\Omega |S|; R).$$ 
In particular, there is a natural isomorphism of algebras $H_0( \Lambda (\mathcal{K}_{\iota}(S^{\sharp})) \cong R[\pi_1(|S|)].$ 
These follow easily by using the fact that $\mathcal{K}_{\iota}(S^{\sharp})$ is categorically equivalent to a Kan complex together with Proposition \ref{RcateqandOmegaqi}, and then applying Theorem \ref{nscAdams}. See also \cite[Corollary 4.2]{CHL}. 
\end{remark}

\subsection{$\widehat{\mathbb{\Omega}}$-quasi-isomorphisms} \label{derived_cobar_eq}
Let $\iota \colon S^1 \to \mathbb{S}^1$ be a fixed localization. We will make use of the localized cobar construction as defined in Definition \ref{localizedcobardef}. 

\begin{definition} A map $f\colon C \to C'$ in $\mathsf{sCoCoalg}^0_R$ is an $\widehat{\mathbb{\Omega}}$-\textit{quasi-isomorphism} if the induced map $\mathbf{\widehat{\mathbb{\Omega}}}(f)\colon \mathbf{\widehat{\mathbb{\Omega}}}(C) \to \mathbf{\widehat{\mathbb{\Omega}}}(C')$ is a quasi-isomorphism of dg algebras. We denote by $\W_{\widehat{\mathbb{\Omega}}}$ the class of $\widehat{\mathbb{\Omega}}$-quasi-isomorphisms in $\mathsf{sCoCoalg}^0_R$ . 
\end{definition}

The motivation for the above definition comes from the fact that a map of simplicial sets $f: S \to S'$ is a weak homotopy equivalence if and only if the induced map $$\mathcal{K}_{\iota}(f^{\sharp}) : \mathcal{K}_{\iota}(S^{\sharp}) \to \mathcal{K}_{\iota}(S'^{\sharp}),$$ obtained by inverting all $1$-simplices, is a categorical equivalence. Given the compatibility between simplicial localization and derived localization (Propositions \ref{localizations} and   \ref{localizations2}), we obtain immediately the following relationship between $\mathbb{\Omega}$-quasi-isomorphisms and $\widehat{\mathbb{\Omega}}$-quasi-isomorphisms for the simplicial chains of reduced simplicial sets. 

\begin{proposition} \label{omega-kan}  Let $f\colon S \to S'$ be a map of reduced simplicial sets. Then $R[f]\colon R[S] \to R[S']$ is an $\widehat{\mathbb{\Omega}}$-quasi-isomorphism if and only if $$R[\mathcal{K}_{\iota}(f^{\sharp})]\colon R[\mathcal{K}_{\iota}(S^{\sharp})] \to R[\mathcal{K}_{\iota}(S'^{\sharp})]$$ is an $\mathbb{\Omega}$-quasi-isomorphism.
\end{proposition}

Based on Theorem \ref{nscAdams}, we may give a concrete description of the $\widehat{\mathbb{\Omega}}$-quasi-isomorphisms arising from maps of reduced simplicial sets. This description is one of the key results for the comparison between the respective homotopy theories of reduced simplicial sets and connected simplicial cocommutative coalgebras.

\begin{definition} A map $f\colon S \to S'$ in $\mathsf{sSet}_0$ is a $\pi_1$-$R$-\textit{equivalence} if it induces an isomorphism between fundamental groups $$\pi_1(|f|)\colon \pi_1(|S|) \xrightarrow{\cong} \pi_1(|S'|)$$ and the induced map between the universal covers $$\overline{f}\colon \overline{|S|} \to \overline{|S'|}$$ is an $R$-equivalence. 
We denote the class of $\pi_1$-$R$-equivalences in $\mathsf{sSet}_0$ by $\W_{\pi_1\text{-}R}$. 
\end{definition} 

\begin{theorem}\label{LOmegareflects} 
A map of reduced simplicial sets $f\colon S \to S'$ is a $\pi_1$-$R$-equivalence if and only if $R[f]\colon R[S] \to R[S']$ is an $\widehat{\mathbb{\Omega}}$-quasi-isomorphism.
\end{theorem}
\begin{proof} 
We recall that for any $X \in \mathsf{sSet}_0$ there is a natural isomorphism of graded algebras 
$$\Theta_X\colon H_*(\Omega \overline{|X|}; R) \otimes R[\pi_1(X)] \xrightarrow{\cong} H_*(\Omega |X|;R).$$ Furthermore,  $\Theta_X$ restricts to an isomorphism of bialgebras $R[\pi_1(|X|)]  \cong H_0(\Omega |X|; R)$ in degree $0$. 

We write $\pi_1 = \pi_1(|S|)$ and $\pi'_1 = \pi_1(|S'|)$ in order to simplify the notation. We consider the following commutative diagram
\[
\xymatrix{
H_*(\Omega \overline{|\mathcal{K}_{\iota}(S^{\sharp})|};R) \otimes R[\pi_1] \ar[d] \ar[r]^-{\Theta}_(0.57){\cong} & H_*(\Omega |\mathcal{K}_{\iota}(S^{\sharp})|; R) \ar[d]^{H_*(\Omega|\mathcal{K}_{\iota}(f^{\sharp})|; R)}  & H_*( \Lambda(\mathcal{K}_{\iota}(S^{\sharp}));R) \ar[l]^{(\ref{nscAdams2})}_{\cong} \ar[d]^{H_*(\Lambda(\mathcal{K}_{\iota}(f^{\sharp})); R)} \\
 H_*(\Omega\overline{|\mathcal{K}_{\iota}(S'^{\sharp})|};R) \otimes R[\pi'_1] \ar[r]^-{\Theta}_(0.57){\cong} &H_*(\Omega |\mathcal{K}_{\iota}(S'^{\sharp})|; R) & \ar[l]_{\cong}^{(\ref{nscAdams2})} H_*( \Lambda(\mathcal{K}_{\iota}(S'^{\sharp})); R) 
}
\]
where every horizontal map is an isomorphism and the vertical maps are induced by the map $f$. Note that the vertical maps are isomorphisms in degree $0$ if and only if $f$ induces an isomorphism of group rings  $R[\pi_1(|f|)]: R[\pi_1] \to R[\pi_1']$, which holds if and only if $f$ induces an isomorphism $\pi_1(|f|): \pi_1 \to \pi_1'$ of fundamental groups.  

Since the canonical maps $S \to \mathcal{K}_{\iota}(S^{\sharp})$ and $S' \to \mathcal{K}_{\iota}(S'^{\sharp})$ are weak homotopy equivalences, the left vertical map is an isomorphism if and only if $f$ is a $\pi_1$-isomorphism and $\Omega\overline{|f|}$ is an $R$-equivalence. 
Moreover, $\overline{|f|}$ is a map of simply-connected spaces, so applying the Zeeman comparison theorem \cite{Z57} to the map of Serre spectral sequences of the respective path fibrations shows that $\Omega \overline{|f|}$ is an $R$-equivalence if and only if $\overline{|f|}$ is an $R$-equivalence.
Thus, the left vertical map is an isomorphism if and only if $f$ is a $\pi_1$-$R$-equivalence. 

On the other hand, by Proposition \ref{omega-kan}, the right vertical map is an isomorphism if and only if $R[f]$ is an $\widehat{\mathbb{\Omega}}$-quasi-isomorphism.  The result follows.
\end{proof} 

By Proposition~\ref{independenceofloc}, when $R=\mathbb{F}$ is a field, the quasi-isomorphism type of the dg algebra $\widehat{\mathbb{\Omega}}(C)$ is independent of the choice of the localization $\iota\colon S^1 \to \mathbb{S}^1$. This is the context in which we will work in later in Subsection~\ref{modelstronscocoalg}. Moreover, the notion of $\mathbb{\Omega}$-quasi-isomorphism is stronger than $\widehat{\mathbb{\Omega}}$-quasi-isomorphism. In fact, we have the following strict inclusions of classes of weak equivalences.
 
 \begin{proposition} \label{inclusions} Let $R=\F$ be a field. The classes of morphisms $\W_{\mathbb{\Omega}}$, $\W_{\widehat{\mathbb{\Omega}}}$, and $\W_{\text{q.i.}}$ in $\mathsf{sCoCoalg}^0_\F$ satisfy the following strict inclusions
 $$\W_{\mathbb{\Omega}} \subsetneq \W_{\widehat{\mathbb{\Omega}}} \subsetneq \W_{\qi}.$$
\end{proposition}
\begin{proof}
The dg algebra $\mathbb{\Omega}(C)$ is flat as a module over $\F$, consequently, it is left proper in the sense of \cite{BCL} (see \cite[Theorem 2.14]{BCL}). It follows that the dg algebra $\widehat{\mathbb{\Omega}}(C)$ is a model for the derived localization of $\mathbb{\Omega}(C)$ at the set of $0$-cycles given by $P_C$ (Remark \ref{derivedlocalization}). 
Suppose now that $f\colon C \to C'$ is a morphism in $\W_{\mathbb{\Omega}}$, that is, the induced map \[\mathbb{\Omega}(f)\colon \mathbb{\Omega}(C) \to \mathbb{\Omega}(C')\] is a quasi-isomorphism. Then the induced map 
\[H_0(\mathbb{\Omega}(f)) \colon H_0(\mathbb{\Omega}(C)) \xrightarrow{\cong} H_0(\mathbb{\Omega}(C')) \]
is an isomorphism of fundamental bialgebras and consequently induces an isomorphism between the respective monoids of monoid-like elements. Hence, the induced map $P_C \to P_{C'}$ yields a bijection \[[P_C] \xrightarrow{\cong} [P_{C'}]\] between the corresponding sets of homology classes. By the invariance of derived localization under quasi-isomorphisms of dg algebras (see \cite[Proposition 3.5]{BCL}, cf. Proposition \ref{invariance}), it follows that $\widehat{\mathbb{\Omega}}(f)\colon \widehat{\mathbb{\Omega}}(C) \to \widehat{\mathbb{\Omega}}(C')$ is a quasi-isomorphism. This shows $\W_{\mathbb{\Omega}} \subseteq \W_{\widehat{\mathbb{\Omega}}}$. The inclusion is clearly strict, since $\F[\iota] \colon \F[S^1] \to \F[\mathbb{S}^1]$ is an $\widehat{\mathbb{\Omega}}$-quasi-isomorphism, but not an $\mathbb{\Omega}$-quasi-isomorphism.

Now let $f\colon C \to C'$ be a map in  $\mathsf{sCoCoalg}^0_\F$ which induces a quasi-isomorphism \[\widehat{\mathbb{\Omega}}(f)\colon \widehat{\mathbb{\Omega}}(C) \to \widehat{\mathbb{\Omega}}(C').\] 
The bar construction sends quasi-isomorphisms of augmented algebras to quasi-isomorphisms of dg conilpotent coalgebras (Theorem \ref{barcobar} and Remark \ref{derived_koszul}), so the induced map
\begin{eqnarray} \label{barlocalizedcobar0}
\mathsf{Bar} (\widehat{\mathbb{\Omega}}(f))\colon \mathsf{Bar}(\widehat{\mathbb{\Omega}}(C)) \to \mathsf{Bar}(\widehat{\mathbb{\Omega}}(C'))
\end{eqnarray}
is a quasi-isomorphism. We claim that for any $C \in \mathsf{sCoCoalg}^0_R$, the canonical map 
\begin{eqnarray}\label{baroflocalizedcobar}
    \mathsf{Bar}(\mathbb{\Omega}(C)) \to \mathsf{Bar}(\widehat{\mathbb{\Omega}}(C))
    \end{eqnarray}
is a quasi-isomorphism. To see this, we first recall that the bar construction $\mathsf{Bar}$ preserves homotopy pushouts (Remark \ref{derived_koszul}). Thus, since $\widehat{\mathbb{\Omega}}(C)$ is defined by the homotopy pushout of dg algebras in \eqref{pushout}, it suffices to observe that the map
\begin{eqnarray*}\label{baroflocalizedcobar2}
  \mathsf{Bar} (\mathbb{\Omega} (R[\iota]) ) \colon \mathsf{Bar}(\mathbb{\Omega}(R[S^1])) \to \mathsf{Bar}(\mathbb{\Omega}(R[\mathbb{S}^1]))
    \end{eqnarray*}
is a quasi-isomorphism. This holds because the map $\mathsf{Bar} (\mathbb{\Omega} (R[\iota]) )$ fits in a commutative square 
\[\begin{tikzcd}
	{ \mathsf{Bar}(\mathbb{\Omega}(R[S^1]))} & {\mathsf{Bar}(\mathbb{\Omega}(R[\mathbb{S}^1]))} \\
	{\mathcal{N}_*(R[S^1])} & {\mathcal{N}_*(R[\mathbb{S}^1])}
	\arrow[from=1-1, to=1-2]
	\arrow["\simeq"', from=2-1, to=2-2]
	\arrow["\simeq", from=2-1, to=1-1]
	\arrow["\simeq"', from=2-2, to=1-2]
\end{tikzcd}\]
and the vertical maps, given by the unit transformation of the $(\mathsf{Cobar}, \mathsf{Bar})$ adjunction, are quasi-isomorphisms by derived Koszul duality (Theorem \ref{barcobar} and Remark \ref{derived_koszul}). 
The bottom horizontal map is a quasi-isomorphism since $\iota \colon S^1 \to \mathbb{S}^1$ is a weak homotopy equivalence. This shows the claim that \eqref{baroflocalizedcobar} is a quasi-isomorphism. Combining \eqref{barlocalizedcobar0} and \eqref{baroflocalizedcobar} and Remark \ref{derived_koszul}, we observe that in the commutative diagram
\[\begin{tikzcd}
	{\mathsf{Bar}(\widehat{\mathbb{\Omega}}(C))} & {\mathsf{Bar}(\widehat{\mathbb{\Omega}}(C'))} \\
	{\mathsf{Bar}(\mathbb{\Omega}(C))} & {\mathsf{Bar}(\mathbb{\Omega}(C'))} \\
	{\mathcal{N}_*(C)} & {\mathcal{N}_*(C')}
	\arrow["\simeq", from=1-1, to=1-2]
	\arrow["\simeq", from=2-1, to=1-1]
	\arrow[from=2-1, to=2-2]
	\arrow["\simeq"', from=2-2, to=1-2]
	\arrow["{\mathcal{N}_*(f)}"', from=3-1, to=3-2]
	\arrow["\simeq", from=3-1, to=2-1]
	\arrow["\simeq"', from=3-2, to=2-2],
\end{tikzcd}\]
all the maps labeled by $\simeq$ are quasi-isomorphisms of dg coalgebras. Thus, $\mathcal{N}_*(f)$ is a quasi-isomorphism as well, so it follows that  $\mathcal{W}_{\widehat{\mathbb{\Omega}}} \subseteq \mathcal{W}_{\text{q.i.}}.$
The inclusion is strict, since there are $\F$-equivalences $f \colon S \to S'$ in $\mathsf{sSet}_0$ which are not $\pi_1$-isomorphisms, so $R[f]$ is not an $\widehat{\mathbb{\Omega}}$-quasi-isomorphism by Theorem \ref{LOmegareflects}. 
\end{proof}

The analogous result for the respective classes of weak equivalences in $\mathsf{sSet}_0$ also holds for general commutative rings $R$.

\begin{proposition} \label{inclusions2} The classes of morphisms $\W_{J, R}$, $\W_{\pi_1\text{-}R}$, $\W_{R}$ in $\mathsf{sSet}_0$
 satisfy the following strict inclusions
 $$\W_{J, R} \subsetneq \W_{\pi_1\text{-}R} \subsetneq \W_{R}.$$
\end{proposition}
\begin{proof}
Suppose $f \in \W_{J, R}$. By Proposition \ref{RcateqandOmegaqi}, we have $R[f] \in \W_{\mathbb{\Omega}}$. The underlying graded $R$-module of the normalized chains on a simplicial set is flat over $R$, so $\mathbb{\Omega}(R[f])$ is a map between $R$-flat dg algebras. The same proof as for the first inclusion in Proposition \ref{inclusions} implies that $R[f] \in \W_{\widehat{\mathbb{\Omega}}}$. By Theorem \ref{LOmegareflects}, this means that $f \in \W_{\pi_1\text{-}R}$. This inclusion is clearly strict, since $\iota \colon S^1 \to \mathbb{S}^1$ is a $\pi_1$-$R$-equivalence that is not a categorical $R$-equivalence. 
The second strict inclusion is well known. 
\end{proof}

\section{A method for constructing model structures} \label{prelim_modelcat}

In this section we will discuss a useful and general method for constructing model structures on a locally presentable category $\C$ that are (left--)induced by a combinatorial model category $\M$ along an accessible functor $F \colon \C \to \M$. This method follows the proof in \cite[Section 4]{Ra} and is based on Smith's recognition theorem for combinatorial model category structures. We will review this method in a general abstract setting, as this will allow us to give a uniform treatment of the construction of the model structures in the next section. 

\subsection{Using Smith's recognition theorem} Let $\C$ be a locally presentable category, let $\M$ be a combinatorial model category, and let $F \colon \C \to \M$ be an accessible functor. Let $\kappa$ be a fixed regular cardinal such that: 
\begin{itemize}
\item[(a)] $\C$ is locally $\kappa$-presentable; 
\item[(b)] $\M$ is $\kappa$-combinatorial, that is, the category $\M$ is locally $\kappa$-presentable and there are generating sets of cofibrations and trivial cofibrations that are given by morphisms between $\kappa$-presentable objects;
\item[(c)] $F$ is $\kappa$-accessible. 
\end{itemize}
Then we consider the following classes of morphisms in $\C$.
\begin{itemize}
\item \emph{Weak equivalences:} $W_{F} = F^{-1}(W_{\M})$. In other words, the class of weak equivalences $W_{F}$ is the inverse image of the weak equivalences $W_{\M}$ in $\M$ under the functor $F$. The class of weak equivalences $W_{\M}$ is accessible and accessibly embedded in $\M^{\to}$ (see, e.g., \cite{Be, Ra2, RR}). Since $F$ is accessible, it follows that $W_{F}$ is also accessible and accessibly embedded in $\C^{\to}$. Moreover, $W_{F}$ satisfies the 2-out-of-3 property. 
\item \emph{Cofibrations:} Consider a set of morphisms in $\C$,
$$\hspace{2cm} I \subseteq \{ i \colon A \to B \ | \ A, B \text{ are }\kappa-\text{presentable, } F(i) \text{ is a cofibration in }\M\}.$$
The class of cofibrations $\cof(I)$ in $\C$ is the cofibrant closure of $I$, that is, the smallest class of morphisms in $\C$ which contains $I$ and is closed under pushouts, retracts, and transfinite compositions. 

For most applications, it will suffice to let $I$ be exactly the set of morphisms between $\kappa$-presentable objects (one from each isomorphism class) which become cofibrations in $\M$ after applying $F$. 
\end{itemize}
By Smith's recognition theorem (see \cite[Theorem 1.7]{Be}, \cite[Theorem 4.1]{Ra}), we have the following immediate conclusion.

\begin{theorem} \label{Smith_thm} 
Let $\C$, $\M$, and $F\colon \C \to \M$ be as above. Then the classes of morphisms in $\C$,
\begin{center}
Weak equivalences: $=W_{F}$ and Cofibrations: $=\cof(I),$
\end{center}
determine a combinatorial model category structure on $\C$ if and only if the following conditions are satisfied:
\begin{enumerate}
\item $I-\inj \subseteq W_{F}$.
\item The class of trivial cofibrations $\cof(I) \cap W_{F}$ is closed under pushouts and transfinite compositions. 
\end{enumerate}
\end{theorem}

\subsection{Techniques} For the applications of Theorem \ref{Smith_thm}, it is helpful to identify specific properties that ensure conditions (1) and (2) and 
are easy to verify in practice. First, we observe that the following conditions on $F$ ensure that condition (2) of Theorem \ref{Smith_thm} is satisfied.

\begin{proposition}  \label{Smith_thm2}
Let $\C$, $\M$, and $F \colon \C \to \M$ be as in Theorem \ref{Smith_thm}. Condition (2) holds if $F$ preserves small colimits. More generally, condition (2) holds  if the following weaker assumptions are satisfied: 
\begin{itemize}
\item[(i)] $F$ sends every pushout square in $\C$
$$
\xymatrix{
A \ar[r] \ar[d]_i & C \ar[d] \\
B \ar[r] & D
}
$$
where $i \colon A \to B$ is in $\cof(I) \cap W_{F}$, to a homotopy pushout square in $\M$.

\item[(ii)] $F$ sends every transfinite composition in $\C$, 
$$A_0 \to A_1 \to \cdots \to A_{\lambda} \to \cdots \to \mathrm{colim}_{\lambda} A_{\lambda},$$
where each morphism $A_i \to A_{i+1}$ is in $\cof(I) \cap W_{F}$, to a homotopy colimit diagram in $\M$.

\end{itemize}
\end{proposition}

Condition (1) of Theorem \ref{Smith_thm} is generally more difficult to verify in practice and some \emph{ad hoc} argument that uses the particularities of the context might be necessary. The following proposition states some general properties which ensure that condition (1) is satisfied (cf. \cite[Lemma 4.4]{Ra}).

\begin{proposition} \label{Smith_thm3}
Let $\C$, $\M$, and $F \colon \C \to \M$ be as in Theorem \ref{Smith_thm}. Suppose that every morphism $f \colon X \to Y$ between $\kappa$-presentable objects in $\C$ admits a factorization $f = p i$ such that:
\begin{itemize}
\item[(a)] $i \in \cof(I)$;
\item[(b)] $F(p)$ is a trivial fibration in $\M$. 
\end{itemize}
Then condition (1) of Theorem \ref{Smith_thm} is satisfied. 
\end{proposition}
\begin{proof}
The proof uses the same arguments as in \cite[Proposition 4.5]{Ra}.
\end{proof}

\begin{remark}
The technical assumptions of Proposition \ref{Smith_thm3} are often satisfied simply by using an appropriate mapping cylinder factorization. (This was the case in \cite[Lemma 4.4]{Ra}.) Property (b) is often the most tricky property to verify in practice. For that reason, it is desirable  to choose $\M$ appropriately. (For example, this was chosen to be the \emph{projective} model category of simplicial $R$-modules in \cite{Ra}.) Let us mention that in the case where $I$ is the maximal choice, i.e. the set of all morphisms between $\kappa$-presentable objects in $\C$ which become cofibrations in $\M$, then (a) clearly holds as long as: 
\begin{itemize}
\item[(a$_1$)] $F(i)$ is a cofibration in $\M$;
\item[(a$_2$)] the codomain of $i$ is $\kappa$-presentable in $\C$ (whenever $X$ and $Y$ are $\kappa$-presentable).  
\end{itemize}
\end{remark}  

\begin{proposition}
Let $\C$, $\M$, and $F \colon \C \to \M$ be as in Theorem \ref{Smith_thm} and suppose that conditions (1) and (2) of Theorem \ref{Smith_thm} are satisfied. Then $$F \colon \C \to \M$$ is a left Quillen functor (with respect to the model category structure of Theorem \ref{Smith_thm}) if and only if $F$ preserves small colimits.
\end{proposition}
\begin{proof} First, note that $F$ preserves weak equivalences by definition. By the (left) special adjoint functor theorem for locally presentable categories (see \cite{AR, LuHTT, NRS}), $F$ is a left adjoint if and only if $F$ preserves small colimits. In this case, it is immediate that $F$ also preserves cofibrations. 
\end{proof}

On the other hand, condition (1) of Theorem \ref{Smith_thm} is certainly satisfied when (for some set $I$) the class $I-\inj$ is already known to be contained in a class of weak equivalences $W_{\C}$ in $\C$ which is preserved by $F$. In particular, Smith's recognition theorem (\cite[Theorem 1.7]{Be}, \cite[Theorem 4.1]{Ra}) has the following immediate consequence for the existence of left Bousfield localizations of combinatorial model categories.

\begin{theorem} \label{loc_modelcat}
Let $(\C, \cof_{\C}, W_{\C}, \mathrm{Fib}_{\C})$ and $(\M, \cof_{\M}, W_{\M}, \mathrm{Fib}_{\M})$ be combinatorial model categories and let $F \colon \C \to \M$ be an accessible functor which preserves the weak equivalences. 

Then the left Bousfield localization of $\C$ with weak equivalences $W_F := F^{-1}(W_{\M})$ exists, and it is again a combinatorial model category, if and only if $\cof_{\C} \cap W_F$ is closed under pushouts and transfinite compositions. 

In particular, this holds if $F$ is additionally a left Quillen functor. 
\end{theorem}

\subsection{Homotopically full and faithful Quillen functors} For our main comparison results about spaces and simplicial coalgebras in Section \ref{sec_fields}, we will be interested in the following property of a Quillen adjunction. 

\begin{proposition} \label{equiv_char_hff}
Let $F \colon \M \rightleftarrows \N \colon G$ be a Quillen adjunction of model categories. The following are equivalent:
\begin{enumerate}
\item For every pair of cofibrant objects $X, Y \in \M$, the induced map of (derived) mapping spaces
$$\map_{\M}(X, Y) \to \map_{\N}(F(X), F(Y))$$
is a weak homotopy equivalence. 
\item  The derived unit map 
$$X \to G(F(X)^f)$$
is a weak equivalence for every cofibrant object $X \in \M$. (Here the morphism $F(X) \to F(X)^f$ is a functorial fibrant replacement in $\N$.)
\item The unit transformation $\mathrm{Id} \Rightarrow \mathbb{R}G \circ \mathbb{L} F$ of the derived adjunction 
$$\mathbb{L} F \colon \ho(\M) \rightleftarrows \ho(\N) \colon \mathbb{R}G$$
is a natural isomorphism.  
\item The left derived functor 
$$\mathbb{L}F \colon \ho(\M) \to \ho(\N)$$
is full and faithful. 
\end{enumerate}
\end{proposition}
\begin{proof} The equivalence of (2), (3) and (4) is well known and clearly (1) implies (4). For the remaining implication recall that, since $F$ and $G$ define a Quillen adjunction, there is an induced weak homotopy equivalence of (derived) mapping spaces 
$$\map_{\M}(F(X), Z) \simeq \map_{\N}(X, G(Z))$$
for a cofibrant object $X$ in $\mathcal{M}$ and a fibrant object $Z$ in $\mathcal{N}$ (see, for instance, \cite[Proposition 17.4.16]{Hir}, where mapping spaces, also known as homotopy function complexes, are constructed via (co)simplicial resolutions). Setting $Z = F(Y)^f$, we may deduce that (2) implies (1).
\end{proof}

\begin{definition}
Let $F \colon \M \rightleftarrows \N \colon G$ be a Quillen adjunction. We say that $F$ is \emph{homotopically full and faithful} if the equivalent conditions of Proposition \ref{equiv_char_hff} are satisfied.
\end{definition}

The following simple criterion will suffice for our applications in Section \ref{sec_fields}.

\begin{proposition} \label{criterion_hff}
Let $F \colon \M \rightleftarrows \N \colon G$ be a Quillen adjunction. Suppose that $F$ is full and faithful and $G$ preserves all weak equivalences. Then 
$F$ is homotopically full and faithful. 
\end{proposition}
\begin{proof}
$F$ is full and faithful if and only if the unit transformation $\mathrm{Id} \Rightarrow G \circ F$ is a natural isomorphism. Since $G$ preserves the weak equivalences, it is immediate that the Quillen adjunction $(F, G)$ satisfies condition (2) of Proposition \ref{equiv_char_hff}.
\end{proof}

\section{Model structures on $\mathsf{sSet}_0$ and  $\mathsf{sCoCoalg}^0_{\mathbb{F}}$ } \label{model_coalg}

In this section we will apply the method of Section \ref{prelim_modelcat} to establish the existence of three model category structures on reduced simplicial sets and their corresponding counterparts for the category of connected simplicial cocommutative coalgebras over a field. These model structures are associated with the classes of weak equivalences discussed in Section \ref{sec_3_weq}. 

Let $\iota\colon S^1 \to \mathbb{S}^1$ be a fixed localization of $S^1$. We will make use of the associated \textit{localized cobar construction} $\widehat{\mathbb{\Omega}} =\mathcal{L}_{\iota} \circ \mathfrak{X}\colon \mathsf{sCoCoalg}^0_{R} \to \mathsf{dgAlg}_{R}$ (see Definition \ref{localizedcobardef} and Subsection \ref{derived_cobar_eq}). 

\subsection{Three model structures on $\mathsf{sSet}_0$} The following result proves the existence of three model category structures on reduced simplicial sets: the first one is a linearized version of the Joyal model category (restricted to reduced simplicial sets); the second one is a left Bousfield localization of the first, which involves considering the fundamental group as in the Kan--Quillen model category; finally, the third one corresponds to the Bousfield model category for $R$-homology equivalences restricted to reduced simplicial sets. 

\begin{theorem}\label{sset_modelcat} Let $\mathsf{sSet}_0$ denote the category of reduced simplicial sets and let $R$ be a commutative ring.
\begin{enumerate}
\item There is a left proper combinatorial model category structure on $\mathsf{sSet}_0$ with the monomorphisms as cofibrations and the categorical $R$-equivalences $\W_{J, R}$ as weak equivalences. This model category is a left Bousfield localization of the Joyal model category structure on $\mathsf{sSet}_0$. 

We denote this model category by $(\mathsf{sSet}_0, R\text{-}\cateq)$.

\medskip

\item There is a left proper combinatorial model category structure on $\mathsf{sSet}_0$ with the monomorphisms as cofibrations and the  $\pi_1$-$R$-equivalences $\W_{\pi_1\text{-}R}$ as weak equivalences. This model category is a left Bousfield localization of the model category in (1) and of the Kan--Quillen model category structure on $\mathsf{sSet}_0$. 

We denote this model category by $(\mathsf{sSet}_0, \pi_1\text{-}R\text{-}\eq)$. 

\medskip

\item There is a left proper combinatorial model category structure on $\mathsf{sSet}_0$ with the monomorphisms as cofibrations and the $R$-equivalences $\W_{R}$ as weak equivalences. This model category is a left Bousfield localization of the model category in (2) and is induced by the Bousfield model category structure on $\mathsf{sSet}$ restricted to $\mathsf{sSet}_0$.

We denote this model category by $(\mathsf{sSet}_0, R\text{-}\eq)$.
\end{enumerate}  
\end{theorem}
\begin{proof}
We recall that $\Lambda( -; R)\colon \mathsf{sSet}_0 \to \mathsf{dgAlg}_R$ is a left Quillen functor when $\mathsf{sSet}_0$ is equipped with model structure induced by the Joyal model category and $\mathsf{dgAlg}_R$ has the model structure of Theorem~\ref{Pmodelstructure} (see Remark~\ref{joyalanddgalg}). Note that $\Lambda(-;R)$ preserves all weak equivalences, since every object in $\mathsf{sSet}_0$ is cofibrant. Then, (1) follows directly by applying Theorem~\ref{loc_modelcat} to the model categories $\C=\mathsf{sSet}_0$ and $\M= \mathsf{dgAlg}_R$, and the functor $F=\Lambda( - ; R)$.  

We will now prove (2). By Theorem \ref{LOmegareflects}, we have 
$$\W_{\pi_1\text{-}R}= (\widehat{\mathbb{\Omega}} \circ R[- ])^{-1}(\W_{\mathsf{dgAlg}_R}),$$ 
where $\W_{\mathsf{dgAlg}_R}$ denotes the class of quasi-isomorphisms of dg algebras. To prove (2),  we will apply Theorem \ref{loc_modelcat} in the case of the model categories
$$\C=  (\mathsf{sSet}_0, R\text{-}\cateq) \text{ (as in (1))}$$ 
$$\M= \mathsf{dgAlg}_R \ \text{ (Theorem \ref{Pmodelstructure})}$$ 
and the functor $F = \widehat{\mathbb{\Omega}} \circ R[- ]\colon \mathsf{sSet}_0 \to \mathsf{dgAlg}_R$. The class of $\pi_1$-$R$-equivalences consists precisely of the maps which are sent to quasi-isomorphisms by $F$ (Theorem \ref{LOmegareflects}). $F$ is a composition of functors which preserve filtered colimits, therefore $F$ also preserves filtered colimits. Moreover, $F$ preserves the weak equivalences by Proposition \ref{inclusions2}. 


Even though $F$ does not preserve pushouts in general, the required property that $\cof_{\C} \cap \W_{\pi_1\text{-}R}$ is closed under pushouts still holds, since given a pushout in $\mathsf{sSet}_0$
\[\begin{tikzcd}
	{ S } & {X}\\
  {S'} & {X'}
	\arrow[from=1-1, to=1-2]
	\arrow[from=2-1, to=2-2]
	\arrow["j", hook, from=1-2, to=2-2]
	\arrow["i", hook, from=1-1, to=2-1],
	\end{tikzcd}\]
where $i$ is a monomorphism and a $\pi_1\text{-}R$-equivalence, then we observe that $j$ is again a monomorphism and a $\pi_1$-$R$-equivalence. Indeed, $j$ is a $\pi_1$-isomorphism by the van Kampen theorem and the map of universal covers is a (homotopy) pushout of an $R$-equivalence covering $i$. (Alternatively, in view of Theorem \ref{LOmegareflects}, Proposition \ref{omega-kan} and Proposition \ref{Joyalsthm}, we could also pass to the simplicial localizations and apply the left Quillen functor $\Lambda(-;R)$.) We see similarly (directly) that $\cof_{\C} \cap \W_{\pi_1\text{-}R}$ is closed under transfinite compositions. This completes the proof of (2).

Finally, (3) follows similarly by applying Theorem~\ref{loc_modelcat} to the case where $\C = (\mathsf{sSet}_0, \pi_1\text{-}R\text{-}\eq)$ (as in (2)), $\M = \mathsf{Ch}_R$ is the category of (non-negatively graded) chain complexes of $R$-modules equipped with the projective model structure, and $F \colon \mathsf{sSet}_0 \to \mathsf{Ch}_R$ is the normalized chains functor. 
\end{proof}

We summarize the relationships between the various model category structures on $\mathsf{sSet}_0$ in the following diagram of left Quillen functors. Every arrow indicates the identity functor on $\mathsf{sSet}_0$. 
$$
\xymatrix{
(\mathsf{sSet}_0, \text{Joyal}) \ar[d] \ar[r] & (\mathsf{sSet}_0, \text{Kan--Quillen}) \ar[d] \\
(\mathsf{sSet}_0, R\text{-}\cateq) \ar[r] & (\mathsf{sSet}_0, \pi_1\text{-}R\text{-}\eq) \ar[r] & (\mathsf{sSet}_0, R\text{-}\eq). 
}
$$

\subsection{A fiberwise fracture theorem} 

Homotopy types may be studied in terms of their localizations (or completions) over $\mathbb{Q}$ and over $\mathbb{F}_p$ at each prime $p$. The \textit{fracture theorem} expresses a (nilpotent) homotopy type as a homotopy pullback of its $\mathbb{Q}$-localization, $\mathbb{F}_p$-localizations, and $\mathbb{Q}$-localization of its $\mathbb{F}_p$-localizations \cite{Sul70, BK72, MPBook}. We will show a version of the fracture theorem over $\pi_1$ for all homotopy types using the homotopy theories of Theorem \ref{sset_modelcat}(2). 

\smallskip

We denote by $X \to X_{R/\pi_1}$
a functorial fibrant replacement of $X \in \mathsf{sSet}_0$ in the model category $(\mathsf{sSet}_0, \pi_1\text{-}R\text{-}\eq)$ of Theorem \ref{sset_modelcat}. Let $P$ denote the set of prime numbers.

\begin{theorem} \label{fiberwise_fracture_thm}
For every $X \in \mathsf{sSet}_0$, the canonical square in $\mathsf{sSet}_0$
\begin{equation} \label{fib_fracture} \tag{*}
\xymatrix{
X \ar[r] \ar[d] & \prod_{p \in P} X_{\F_p/\pi_1} \ar[d] \\
X_{\Q/\pi_1} \ar[r] & \big(\prod_{p \in P} X_{\F_p/\pi_1}\big)_{\Q/\pi_1} 
}
\end{equation}
is a homotopy pullback (with respect to the Kan--Quillen model structure). Moreover, suppose we are given a homotopy pullback in $\mathsf{sSet}_0$
\begin{equation} \label{fib_square2} \tag{\(\dagger\)} 
\xymatrix{
X \ar[r] \ar[d] & \prod_{p \in P} X_p \ar[d]^f \\
X_0 \ar[r]^g & Y
}
\end{equation}
where the following hold:
\begin{itemize}
\item[(i)] $X_p$ is fibrant in $(\mathsf{sSet}_0, \pi_1\text{-}\F_p\text{-}\eq)$; 
\item[(ii)] $X_0$ and $Y$ are fibrant in $(\mathsf{sSet}_0, \pi_1\text{-}\Q\text{-}\eq)$; 
\item[(iii)] $f$ is a $\pi_1$-$\Q$-equivalence in $(\mathsf{sSet}_0, \pi_1\text{-}\Q\text{-}\eq)$;
\item[(iv)] for each $p \in P$, the composition
$$\pi_1(X_0) \xrightarrow{\pi_1(g)} \pi_1(Y)  \stackrel{\pi_1 (f)}{\cong} \pi_1( \prod_{p \in P} X_p) \to \pi_1(X_p)$$
is an isomorphism. 
\end{itemize}
Then the square can be identified up to weak homotopy equivalence with the square \eqref{fib_fracture} above. 
\end{theorem}

For the proof of Theorem \ref{fiberwise_fracture_thm}, we will make use of the following lemma about fibrant objects in $(\mathsf{sSet}_0, \pi_1\text{-}R\text{-}\eq)$. Note that every fibrant object in $(\mathsf{sSet}_0, R\text{-}\eq)$ is also fibrant in $(\mathsf{sSet}_0, \pi_1\text{-}R\text{-}\eq)$.

\begin{lemma} \label{technical_lemma}
Let $R$ be $\F_p$ (for some prime $p$) or $\Q$. Let $Z$ be a reduced Kan complex and let $p_1 \colon Z \to P_1(Z)$ be the map to the first Postnikov truncation of $Z$. If $Z$ is fibrant in $(\mathsf{sSet}_0, \pi_1\text{-}R\text{-}\eq)$, then the homotopy fiber of $p_1$ (with respect to the Kan--Quillen model structure) is also fibrant in $(\mathsf{sSet}_0, R\text{-}\eq)$.
\end{lemma}
\begin{proof}
We may assume that $p_1 \colon Z \to P_1(Z)$ is a fibration between fibrant objects in $(\mathsf{sSet}_0, \text{Kan--Quillen})$ (see \cite[Section V.6; esp. 6.6--6.8]{GJ}).  It is easy to see directly that $P_1(Z)$ is fibrant/local in $(\mathsf{sSet}_0, \pi_1\text{-} R\text{-}\eq).$ By general results on left Bousfield localizations \cite[Proposition 3.3.16]{Hir}, it follows that $Z$ is fibrant in $(\mathsf{sSet}_0, \pi_1\text{-}R\text{-}\eq)$ if and only if 
$p_1$ is a (local) fibration in $(\mathsf{sSet}_0, \pi_1\text{-}R\text{-}\eq)$. In this case, the (homotopy) fiber $\overline{Z}$ of $p_1$, which is simply a model for the universal cover of $Z$, is also fibrant in $(\mathsf{sSet}_0, \pi_1\text{-} R\text{-}\eq).$ 

Thus it suffices to prove the following special case of the lemma: a simply-connected $Z \in \mathsf{sSet}_0$ is fibrant in $(\mathsf{sSet}_0, \pi_1\text{-}R\text{-}\eq)$ if and only if $Z$ is fibrant in $(\mathsf{sSet}_0, R\text{-}\eq)$. A fibrant replacement $Z \to Z_R$ in $(\mathsf{sSet}_0, R\text{-}\eq)$ is also a $\pi_1\text{-}R$-equivalence, because $Z_R$ is again simply-connected. Since $(\mathsf{sSet}_0, R\text{-}\eq)$ is a left Bousfield localization of $(\mathsf{sSet}_0, \pi_1\text{-}R\text{-}\eq)$, it follows that $Z_R$ is also fibrant in $(\mathsf{sSet}_0, \pi_1\text{-}R\text{-}\eq)$; then the map $Z \to Z_R$ is a weak homotopy equivalence. Therefore, $Z$ is fibrant in $(\mathsf{sSet}_0, R\text{-}\eq)$, as required. 
 \end{proof}

\begin{proof}(Theorem \ref{fiberwise_fracture_thm})
We consider the canonical map from the square \eqref{fib_fracture} of Theorem \ref{fiberwise_fracture_thm} to the square \eqref{fib_fracture_P1} below obtained after applying the first Postnikov truncation pointwise; explicitly, this Postnikov truncation of \eqref{fib_fracture} can be identified with the following square (denoting by $BG$ a model for the reduced simplicial set with fundamental group $G$ and trivial higher homotopy groups):
\begin{equation} \label{fib_fracture_P1} \tag{**}
\xymatrix{
B \pi_1 X \ar[d]_{\simeq} \ar[r]^(.4){\Delta} & \prod_{p \in P} B \pi_1 X \ar[d]^{\simeq} \\
B \pi_1 X \ar[r]^(0.43){\Delta} & \prod_{p \in P} B \pi_1 X. 
}
\end{equation}
The square \eqref{fib_fracture_P1} is obviously a homotopy pullback (with respect to the Kan--Quillen model structure). Then consider the induced square which consists of the homotopy fibers of the canonical map of squares from \eqref{fib_fracture} to \eqref{fib_fracture_P1}; the resulting square is, up to weak homotopy equivalence, the square of the universal covers of \eqref{fib_fracture}. Using Lemma \ref{technical_lemma}, this square agrees with the classical fracture homotopy pullback of the universal covering $\overline{X}$ of $X$ (see \cite[Theorem 13.1.4]{MPBook}, \cite{BK72}, \cite{Sul70}). Then the first claim of Theorem \ref{fiberwise_fracture_thm} follows from looking at the long exact sequences of homotopy groups. 

For the second claim, we consider again the map of squares from \eqref{fib_square2} to the (homotopy commutative) square:
\begin{equation} \label{fib_fracture_P2} \tag{\(\dagger\dagger\)}
\xymatrix{
B \pi_1 X_0 \ar@{=}[d] \ar[r] & \prod_{p \in P} B \pi_1 X_p \ar[d]^{\simeq}_{\pi_1(f)} \\
B \pi_1 X_0 \ar[r]^{\pi_1(g)} & B \pi_1 Y, 
}
\end{equation}
which, using the assumptions (iii)--(iv), can be identified up to homotopy equivalence with the homotopy pullback
\begin{equation} \label{fib_fracture_P3} \tag{\(\dagger\dagger\)}
\xymatrix{
B \pi_1 X_0 \ar[d]_{\simeq} \ar[r]^(.4){\Delta} & \prod_{p \in P} B \pi_1 X_0 \ar[d]^{\simeq} \\
B \pi_1 X_0 \ar[r]^(0.4){\Delta} & \prod_{p \in P} B \pi_1 X_0. 
}
\end{equation}
Passing to the homotopy fibers of the maps from \eqref{fib_square2} to \eqref{fib_fracture_P2}, we obtain a homotopy pullback:
\begin{equation} \label{fib_fracture_P4} \tag{\(\sim\)}
\xymatrix{
X' \ar[d] \ar[r] & \prod_{p \in P} \overline{X}_p \ar[d] \\
\overline{X}_0 \ar[r] & \overline{Y}. 
}
\end{equation}
Then it follows from Lemma \ref{technical_lemma} and the corresponding uniqueness property of the classical fracture square (see \cite[Theorem 13.1.5]{MPBook}) that \eqref{fib_fracture_P4} is the classical fracture homotopy pullback of the space $X'$. In particular, it follows that $X'$ is simply-connected, so the map $X \to B \pi_1 X_0$ is $2$-connected, $X'$ is (a model for) the universal cover of $X$, and $X \to X_0$ is a $\pi_1$-isomorphism. Therefore, we conclude that the maps $X \to X_0$ and $X \to X_p$ are fibrant replacements in the respective model categories. 
\end{proof}

\subsection{Three model structures on $\mathsf{sCoCoalg}^0_\mathbb{F}$} \label{modelstronscocoalg}
The following result proves the existence of three model category structures on $\mathsf{sCoCoalg}^0_\mathbb{F}$ where $\mathbb{F}$ is a field. These model structures  correspond to the classes of weak equivalences $\mathcal{W}_{\mathbb{\Omega}}, \W_{\widehat{\mathbb{\Omega}}},$ and $\W_{\qi}$, respectively, as discussed in Section \ref{sec_3_weq}. 

We say that that a morphism $f \colon C \to C'$ in $\mathsf{sCoCoalg}^0_{\F}$ is \emph{injective} if the underlying map of simplicial $\F$-vector spaces is injective degreewise. It turns out that -- in the cocommutative case over a field -- the injective maps in $\mathsf{sCoCoalg}^0_{\F}$ (or just $\mathsf{CoCoalg}_{\F}$) are exactly the monomorphisms \cite{NT, Ag}.

\begin{theorem} \label{coalg_modelcat} Let $\mathbb{F}$ be a field and let $\mathsf{sCoCoalg}^0_\mathbb{F}$ denote the category of connected simplicial cocommutative $\mathbb{F}$-coalgebras. 
\begin{enumerate}
\item There is a left proper combinatorial model category structure on $\mathsf{sCoCoalg}^0_\F$ with the injective maps as cofibrations and the $\mathbb{\Omega}$-quasi-isomorphisms $\mathcal{W}_{\mathbb{\Omega}}$ as weak equivalences. 

We denote this model category by $(\mathsf{sCoCoalg}^0_\F, \mathbb{\Omega}\text{-}\qi)$. 

\medskip

\item   There is a left proper combinatorial model category structure on $\mathsf{sCoCoalg}^0_\F$ with the injective maps as cofibrations and the $\widehat{\mathbb{\Omega}}$-quasi-isomorphisms $\W_{\widehat{\mathbb{\Omega}}}$ as weak equivalences. This model category is a left Bousfield localization of the model category in (1).

We denote this model category by $(\mathsf{sCoCoalg}^0_\F, \widehat{\mathbb{\Omega}}\text{-}\qi)$. 

\medskip

\item  There is a left proper combinatorial model category structure on $\mathsf{sCoCoalg}^0_\F$ with the injective maps as cofibrations and the quasi-isomorphisms $\W_{\qi}$ as weak equivalences. This model category is a left Bousfield localization of the model category in (2). Moreover, it is induced by the model category of Theorem \ref{modelcatqi} restricted to $\mathsf{sCoCoalg}^0_\F$ (for $\kappa = \aleph_0$). 
 
We denote this model category by $(\mathsf{sCoCoalg}^0_\F, \qi)$. 
\end{enumerate}
\end{theorem}
\begin{proof}
We will prove (1) using Theorem \ref{Smith_thm}. Since every $\F$-coalgebra is a filtered colimit of its finite-dimensional subcoalgebras \cite{Sw69}, it follows easily that the categories $\mathsf{CoCoalg}_{\F}$ and $\mathsf{sCoCoalg}^0_{\F}$ are locally finitely presentable.  Let $I$ denote the set of injective maps $i \colon A \to B$ in $\mathsf{sCoCoalg}^0_\F$ between finitely presentable objects. We claim that $\mathrm{Cof}(I)$ consists of the class of injective maps in $\mathsf{sCoCoalg}^0_{\F}$. Clearly every map in $\mathrm{Cof}(I)$ is injective. Conversely, given an injective map $i \colon C \to D$ in $\mathsf{sCoCoalg}^0_{\F}$, we apply the small object argument to obtain a factorization:
$$C \xrightarrow{j} E \xrightarrow{q} D,$$
such that $j \in \mathrm{Cof}(I)$ and $q \in I - \inj$. Then we consider the following lifting problem:
$$
\xymatrix{
C \ar[d]_i \ar[r]^j & E \ar[d]^q \\
D \ar@{=}[r] & D
}
$$
and the poset $\mathcal{Z}$ whose elements are pairs $(Z, h)$ given by an ``intermediate'' simplicial coalgebra $(C \subseteq Z \subseteq D)$ and an ``intermediate'' solution to the lifting problem $h \colon Z \to E$. We define $(Z_1, h_1) \leq (Z_2, h_2)$ if $Z_1 \subseteq Z_2$ and $h_2$ extends $h_1$. Clearly $\mathcal{Z} \neq \varnothing$ and every chain in $\mathcal{Z}$ has an upper bound. By Zorn's lemma, $\mathcal{Z}$ has a maximal element $(C', h)$. If 
$C' \neq D$, then there is a simplicial finite-dimensional subcoalgebra $D' \subseteq D$ such that the injective map $C' \cap D' \subseteq D'$ is in $I$ and $C' \cap D' \neq D'$; the existence of $D'$ uses the fact that every (simplicial) $\F$-coalgebra is the filtered colimit of its finite-dimensional subcoalgebras. Since $q \in I - \inj$, we can then extend $h$ to the simplicial ``intermediate'' subcoalgebra $C' \subsetneq C' \oplus_{C' \cap D'} D' \subseteq D$, contradicting the maximality of $(C', h)$. Therefore, $C' = D$ and the lift $h$ exhibits $i$ is a retract of $j$, so $i \in \mathrm{Cof}(I)$, as claimed. 

Let $\mathsf{dgAlg}_{\F}$ be the finitely combinatorial model category of Theorem \ref{Pmodelstructure}. We will apply the method of Theorem \ref{Smith_thm} to the functor $\mathbb{\Omega} \colon \mathsf{sCoCoalg}^0_{\F} \to \mathsf{dgAlg}_{\F}$. We recall that $\mathbb{\Omega}$ preserves colimits; therefore, it is finitely accessible. 
In addition,  using the explicit description of the cofibrations in $\mathsf{dgAlg}_{\F}$, we see directly that for any injective map $i \colon A \to B$ in $\mathsf{sCoCoalg}^0_\F$, the map $\mathbb{\Omega}(i): \mathbb{\Omega}(A) \to \mathbb{\Omega}(B)$ is a cofibration in the model category $\mathsf{dgAlg}_\F$ (Remark \ref{derived_koszul}, cf. \cite[Section 9]{Po11}).

Condition $(2)$ of Theorem \ref{Smith_thm} is satisfied using Proposition \ref{Smith_thm2} and the fact that $\mathbb{\Omega}$ preserves colimits. Condition (1) of Theorem \ref{Smith_thm} will be shown using the 
criterion in Proposition \ref{Smith_thm3} for the standard mapping cylinder factorization. Namely, the cylinder object is given by the functor
$$\text{Cyl}\colon \mathsf{sCoCoalg}^0_\F \to \mathsf{sCoCoalg}^0_\F, \ C \mapsto \text{Cyl}(C),$$
defined by the pushout
\begin{equation} \label{cylinder} \tag{Cyl}
\xymatrix{
\F \otimes \F[\mathbb{\Delta}^1] \ar[d] \ar[r]^{e \otimes \mathrm{id}} & C \otimes \F[\mathbb{\Delta}^1] \ar[d] \\
\F \otimes \F[\mathbb{\Delta}^0] \cong \F \ar[r] & \text{Cyl}(C)
} 
\end{equation}
where $\F$ is considered as a constant simplicial $\F$-coalgebra and $e\colon \F \to C$ denotes the (implicit) coaugmentation map. In other words, as simplicial $\F$-vector space, we have:
$$\text{Cyl}(C)=\F \oplus \text{coker}(e \otimes \text{id}_{\F[\mathbb{\Delta}^1]}\colon \F\otimes \F[\mathbb{\Delta}^1] \to C \otimes \F[\mathbb{\Delta}^1]).$$

The $n$-simplices of $\mathbb{\Delta}^1$ may be labeled by sequences $[0...01...1]$ consisting of $r$ consecutive $0$'s and $s$ consecutive $1$'s for some non-negative integers $r$ and $s$ satisfying $r+s=n+1$. We denote any such simplex by $[0^r 1^s] \in (\mathbb{\Delta}^1)_n$. 
Using this notation, we have natural inclusion maps 
$i_{\epsilon}\colon C \to \text{Cyl}(C)$, for $\epsilon=0,1,$  by $i_{\epsilon}( x) = x \otimes [\epsilon^{n+1}]$ for any $x \in C_n$. 
We also have a natural projection map $q \colon \text{Cyl}(C) \to C$ determined by $q(x \otimes [0^r1^s])=x$. 

For every map $f \colon C \to C'$ in $\mathsf{sCoCoalg}^0_\F$, we define the mapping cylinder of $f$ 
by the pushout
$$
\xymatrix{
C \ar[d]^f \ar[rr]^(0.45){0 \oplus_\F \mathrm{id}_C} && C \oplus_\F C \ar[d]^{\mathrm{id}_C \oplus f} \ar[rr]^{i_0 \oplus_\F i_1} && \text{Cyl}(C) \ar[d]^h \\
C' \ar[rr] && C \oplus_\F C' \ar[rr]^s && \text{M}(f).
}
$$
Then there is a canonical factorization of $f$,
$$C \xrightarrow{i := h \circ i_0} \text{M}(f) \xrightarrow{p} C',$$
where $p$ is induced by $f$, the identity map on $C'$, and $\text{Cyl}(C) \xrightarrow{q} C \xrightarrow{f} C'$.  

We claim that the conditions (a)--(b) of Proposition \ref{Smith_thm3} are satisfied for this factorization. First, we see from the construction that the map $i$ is injective (because $s$ is). Moreover, assuming that $C$ and $C'$ are finitely presentable in $\mathsf{sCoCoalg}^0_\F$, then so is $\text{M}(f)$; so Proposition \ref{Smith_thm3}(a) is satisfied. 

For Proposition \ref{Smith_thm3}(b), we will show that the map $p \colon \text{M}(f) \to C'$ becomes a trivial fibration in $\mathsf{dgAlg}_\F$ after applying $\mathbb{\Omega}$. First, since $p$ admits a section $$s' \colon C' \to C \oplus_\F C' \xrightarrow{s} \text{M}(f),$$ it follows that $\mathbb{\Omega}(p)$ is surjective (a fibration in $\mathsf{dgAlg}_{\F}$). Then it suffices to show that the section $s'$ is an $\mathbb{\Omega}$-quasi-isomorphism. To see this, it suffices to prove that $$\mathbb{\Omega}(i_1)\colon \mathbb{\Omega}(C) \to \mathbb{\Omega}(\text{Cyl}(C))$$ is a trivial cofibration of dg algebras, since $\mathbb{\Omega}(s')$ is a pushout of this map. Moreover, since $ q \circ i_1 = \mathrm{id}_C$, it suffices to verify that the map $q\colon \text{Cyl}(C) \to C$ is an $\mathbb{\Omega}$-quasi-isomorphism. We prove this in Proposition \ref{projcobarquasiiso} of Appendix~\ref{sec_appendix}. 

To prove $(2)$, we will apply Theorem \ref{loc_modelcat} to the functor $$\widehat{\mathbb{\Omega}} \colon \mathsf{sCoCoAlg}^0_\F \to \mathsf{dgAlg}_\F$$
with respect to the model category structures of (1) and Theorem \ref{Pmodelstructure}. First, it is easy to see that $\widehat{\mathbb{\Omega}} = \mathcal{L}_{\iota} \circ \mathfrak{X}$ is the composition of finitely accessible functors; therefore, $\widehat{\mathbb{\Omega}}$ is finitely accessible. By Proposition~\ref{inclusions},  we have $\mathcal{W}_{\mathbb{\Omega}} \subseteq \mathcal{W}_{\widehat{\mathbb{\Omega}}}$, so $\widehat{\mathbb{\Omega}}$ preserves weak equivalences.
Thus, the assumptions of Theorem \ref{loc_modelcat} are satisfied. 

It remains to show that the class $\mathrm{Cof}(I) \cap \mathcal{W}_{\widehat{\mathbb{\Omega}}}$ is closed under transfinite compositions and pushouts using Proposition \ref{Smith_thm2}. The case of transfinite compositions is immediate from the fact that $\widehat{\mathbb{\Omega}}$ preserves filtered colimits, and the observation that a transfinite composition of quasi-isomorphisms in $\mathsf{dgAlg}_{\F}$ is again a quasi-isomorphism. 
We now argue for the case of pushouts. For every cofibration $i \colon A \hookrightarrow B$ in $\mathsf{sCoCoalg}^0_{\F}$, the induced map 
$$\mathbb{\Omega}(i) \colon \mathbb{\Omega}(A) \to \mathbb{\Omega}(B)$$
is a cofibration in $\mathsf{dgAlg}_{\F}$ (see Subsection \ref{model_str_dgalg_prelim}) and also the map $P_A \hookrightarrow P_B$ is injective. As a consequence, the map $(\mathcal{L}_{\iota}\mathfrak{X})(i)$ can be written as the composition of cofibrations
$$\mathcal{L}_{\iota}(\mathfrak{X}(A)) \to \mathcal{L}_{\iota}(B, i(P_A)) \to \mathcal{L}_{\iota} (\mathfrak{X}(B)),$$
therefore, it is again a cofibration in $\mathsf{dgAlg}_{\F}$. Suppose now that $i \colon A \hookrightarrow B$ is in $\mathrm{Cof}(I) \cap \mathcal{W}_{\widehat{\mathbb{\Omega}}}$ and let $f \colon A \to C$ be an arbitrary map in $\mathsf{sCoCoAlg}^0_\F$. We need to show that the top left horizontal map in the diagram below, induced by the map $C \to D := C \oplus_A B$,
\[\begin{tikzcd}
	{\mathcal{L}_{\iota} (\mathfrak{X}(C)) } & {\mathcal{L}_{\iota} (\mathfrak{X}(C \underset{A}{\oplus} B)) } & \mathcal{L}_{\iota}\big(\mathbb{\Omega}(C) \underset{\mathbb{\Omega}(A)}{\circledast} \mathbb{\Omega}(B), P_D\big) \\
	& {\mathcal{L}_{\iota} (\mathfrak{X}(C)) \underset{\mathcal{L}_{\iota} (\mathfrak{X}(A))}{\circledast} \mathcal{L}_{\iota} (\mathfrak{X}(B))} \\
 \mathbb{\Omega}(C) && \mathbb{\Omega}(C) \underset{\mathbb{\Omega}(A)}{\circledast} \mathbb{\Omega}(B) = \mathbb{\Omega}(D)
	\arrow[from=1-1, to=1-2]
	\arrow["\simeq", from=2-2, to=1-2]
	\arrow["\simeq", from=1-1, to=2-2]
    \arrow["\mathrm{loc}", from=3-1, to=1-1]
    \arrow[from=3-1, to=3-3]
    \arrow["\mathrm{loc}", from=3-3, to=2-2]
    \arrow["\cong", from=1-3, to=1-2]
    \arrow["\mathrm{loc}", from=3-3, to=1-3]
\end{tikzcd}\]
is a quasi-isomorphism of dg algebras.
The left vertical map is given by localization at $P_C$ and the right vertical map is given by the respective localization for the pushout of dg algebras $\mathbb{\Omega}(D)$ at the set $P_D$.
The top diagonal map is a quasi-isomorphism since $(\mathcal{L}_{\iota}\mathfrak{X})(i)$ is a trivial cofibration. The bottom diagonal map is a localization at the set of monoid-like elements that arise from $B$ and $C$. Since every monoid-like element in the  bialgebra $H_0(\mathcal{L}_{\iota}(\mathfrak{X}(C)))$ is invertible, the same holds also for the bialgebra structure of the zeroth homology in the middle of the diagram.  So every monoid-like element, in particular those from $[P_D]$, becomes invertible in the bialgebra structure of the zeroth homology in the middle. Hence the middle vertical map must also be a quasi-isomorphism between derived localizations of $\mathbb{\Omega}(D)$ at $[P_D]$ (see \cite[Section 3]{BCL}). This shows that the top left horizontal map is a quasi-isomorphism, as claimed, and completes the proof of (2).


The proof of $(3)$ is obtained similarly (and more easily) by applying Theorem~\ref{loc_modelcat} to the normalized chains functor,
$$(\mathsf{sCoCoalg}^0_\F, \widehat{\mathbb{\Omega}}\text{-}\qi) \xrightarrow{\mathcal{N}_*} \mathsf{dgCoalg}^c_{\F} \xrightarrow{\text{forgetful}}  \mathsf{Ch}_{\F},$$ 
where $\mathsf{Ch}_{\F}$ denotes the category of chain complexes of $\F$-vector spaces, equipped with the usual combinatorial model structure that is determined by the monomorphisms and the quasi-isomorphisms. Note that the normalized chains functor preserves weak equivalences, since $\W_{\widehat{\mathbb{\Omega}}} \subseteq \W_{\text{q.i.}},$ by Proposition~\ref{inclusions}. 
\end{proof}

\begin{remark}
The proof of Theorem \ref{coalg_modelcat}(3) is based on the existence of the model category $(\mathsf{sCoCoalg}^0_{\F}, \widehat{\mathbb{\Omega}}\text{-}\qi)$ and differs from the proof of Theorem \ref{modelcatqi} given in \cite{Ra}. On the other hand, following directly the proof of Theorem \ref{modelcatqi} (based on Theorem \ref{Smith_thm} and Propositions \ref{Smith_thm2}--\ref{Smith_thm3}) and using the mapping cylinder factorization from the proof of Theorem \ref{coalg_modelcat}(1), one can show that the model category of Theorem \ref{coalg_modelcat}(3) exists for arbitrary commutative rings $R$. 

We note that the proof of Theorem \ref{coalg_modelcat}(1) used special properties about simplicial coalgebras and dg algebras over fields; for instance, that an injective map $A \to B$ induces a cofibration of dg algebras $\mathbb{\Omega}(A) \to \mathbb{\Omega}(B)$. 
\end{remark}

The three pairs of model categories established in Theorems \ref{sset_modelcat} and \ref{coalg_modelcat} fit into Quillen adjunctions as follows. 

\begin{proposition} \label{quillen_adj} Let $\F$ be a field. The adjunctions
\begin{equation} \label{quillenadj1} 
\F[-]\colon (\mathsf{sSet}_0, \F\text{-}\cateq) \rightleftarrows (\mathsf{sCoCoalg}^0_\F, \mathbb{\Omega}\text{-}\qi) \colon \mathcal{P}
\end{equation}
\begin{equation} \label{quillenadj2}
\F[-]\colon (\mathsf{sSet}_0, \pi_1\text{-}\F\text{-}\eq) \rightleftarrows (\mathsf{sCoCoalg}^0_\F, \widehat{\mathbb{\Omega}}\text{-}\qi) \colon \mathcal{P}
\end{equation}
\begin{equation} \label{quillenadj3}
\F[-]\colon (\mathsf{sSet}_0, \F\text{-}\eq) \rightleftarrows (\mathsf{sCoCoalg}^0_\F, \qi) \colon \mathcal{P}
\end{equation}
are Quillen adjunctions. 
\end{proposition}
\begin{proof} The left adjoint functor $\F[-]\colon \mathsf{sSet}_0 \to \mathsf{sCoCoalg}^0_\F $ sends monomorphisms to injective maps, that is, $\F[-]$ preserves cofibrations. 
Moreover, using Propositions \ref{RcateqandOmegaqi} and \ref{qiandqi} (for \eqref{quillenadj1} and \eqref{quillenadj3}) and Theorem \ref{LOmegareflects} (for \eqref{quillenadj2}), it follows that $\F[-]$ preserves the weak equivalences. 
\end{proof}

\begin{remark}
The left adjoint $\mathbb{\Omega} \colon \mathsf{sCoCoalg}^0_{\F} \to \mathsf{dgAlg}_{\F}$ is a left Quillen functor with respect to the model category structures of Theorem \ref{coalg_modelcat}(1) and Theorem \ref{Pmodelstructure}. 
\end{remark}

\section{Comparison between $\mathsf{sSet}_0$ and $\mathsf{sCoCoalg}^0_\F$} \label{sec_fields}

\subsection{Coalgebraic preliminaries} \label{coalgebra_prelim} We recall some fundamental results about the structure theory of coalgebras over a perfect field. We refer the reader to \cite{Sw69, Go} for more details. 

Let $\F$ be a field. A cocommutative $\F$-coalgebra $A$ is \emph{simple} if it has no non-trivial subcoalgebras; simple $\F$-coalgebras are necessarily finite-dimensional. The \'etale subcoalgebra $\et(A)$ of a cocommutative $\F$-coalgebra $A$ is the (direct) sum of all simple subcoalgebras of $A$.

If $\F$ is algebraically closed, then $\F$ itself is the only simple $\F$-coalgebra up to isomorphism. In this case, the \'etale subcoalgebra of a cocommutative $\F$-coalgebra $A$ can be identified with the canonical counit map
\begin{equation} \label{etale_part_closed}
\F[\mathcal{P}(A)] \subseteq A
\end{equation}
which is associated with the adjunction $\F[-] \colon \mathsf{Set} \rightleftarrows \mathsf{CoCoalg}_{\F} \colon \mathcal{P}$. 

More generally, if $\F$ is a perfect field, the \'etale subcoalgebra of a cocommutative $\F$-coalgebra $A$ can be identified as follows. Let $\F \subseteq \overline{\F}$ be the algebraic closure of $\F$ and let $G$ denote the (profinite) absolute Galois group. Then the \'etale subcoalgebra of $\overline{A} = A \otimes_{\F} \overline{\F}$ is canonically identified with the $\overline{\F}$-subcoalgebra $\overline{\F}[\mathcal{P}(\overline{A})]$ generated by the $\overline{\F}$-points of $\overline{A}$. Moreover, this set of $\overline{\F}$-points
$$\mathcal{P}(\overline{A}) = \mathsf{CoCoalg}_{\overline{\F}}(\overline{\F}, A \otimes_\F \overline{\F})$$ 
admits a natural $G$-action, which extends to a $G$-action on the $\overline{\F}$-coalgebra $\overline{\F}[\mathcal{P}(\overline{A})]$; this $G$-action is compatible with the $G$-action on $\overline{A}$. The $G$-fixed points of $\overline{\F}[\mathcal{P}(\overline{A})]$ define an $\F$-subcoalgebra of $A$ and there is a natural isomorphism of $\F$-coalgebras:
\begin{equation} \label{etale_part}
\et(A) \cong \overline{\F}[\mathcal{P}(\overline{A})]^G.
\end{equation}
The inclusion of the \'etale subcoalgebra can be expressed also in this case in terms of the counit of an adjunction. Let $\mathsf{Set}(G)$ denote the category of $G$-sets (= discrete topological spaces with a continuous $G$-action). We recall that a $G$-action on a set $X$ is continuous if and only if the stabilizer of each element $x \in X$ is an open subgroup of $G$. This is equivalent to requiring that
\begin{equation}
X = \underset{H \in \mathcal{O}_{G}^{op}}{\text{ colim } }X^H,
\end{equation}
where $X^H$ denotes the $H$-fixed points of $X$ and $\mathcal{O}_G^{op}$ denotes the opposite of the category of open subgroups $H$ of $G$ with inclusions as morphisms. Note that the poset $\mathcal{O}_G^{op}$ is filtered. We recall that every open subgroup of a profinite group has finite index.

For any $G$-set $S$, the $G$-fixed points of the $\overline{\F}$-coalgebra $\overline{\F}[S]$ form an $\F$-coalgebra (cf. \eqref{technical_iso} below), so we obtain a functor
$$\overline{\F}[-]^G \colon \mathsf{Set}(G) \to \mathsf{CoCoalg}_{\F},$$
which admits a right adjoint, given by a functor of $\overline{\F}$-points, defined by
$$\mathcal{P}_G \colon \mathsf{CoCoalg}_{\F} \to \mathsf{Set}(G), \ \ A \mapsto \mathsf{CoCoalg}_{\overline{\F}}(\overline{\F}, A \otimes_{\F} \overline{\F}).$$
Then the inclusion of the \'etale subcoalgebra $\et(A) \subseteq A$ can be identified with the counit of this adjunction. For any $G$-set $S$, there is a canonical isomorphism of $\F$-coalgebras (cf. the proof of \cite[Lemma 4.3]{Go}):
\begin{equation} \label{technical_iso} 
\overline{\F}[S]^G \otimes_{\F} \overline{\F} \cong \overline{\F}[S]
\end{equation}
and it follows that the unit transformation of the adjunction $(\overline{\F}[-]^G, \mathcal{P}_G)$,
$$S \to \mathcal{P}_G(\overline{\F}[S]^G),$$
is a natural isomorphism. In particular, the left adjoint $\overline{\F}[-]^G$ is full and faithful. 

We will need the following fundamental property of the \'etale subcoalgebra. For the proof, we refer the reader to \cite[Section 2]{Go}, \cite[Section 4]{Nik}, \cite[Section 2]{Gu}. 

\begin{theorem} \label{decomposition_thm}
Let $A$ be a cocommutative $\F$-coalgebra over a perfect field $\F$. The inclusion $\et(A) \subseteq A$ admits a natural retraction of coalgebras. 
\end{theorem}

\subsection{Algebraically closed fields} Based on Theorem \ref{decomposition_thm}, we can now apply the criterion of Proposition \ref{criterion_hff} to the Quillen adjunctions of Proposition \ref{quillen_adj}. 

\begin{theorem} \label{comparisons1} Let $\mathbb{F}$ be an algebraically closed field. The left Quillen functors 
\begin{equation} \label{quillen-adj1} \tag{1}
\F[-]\colon (\mathsf{sSet}_0, \F\text{-}\cateq) \rightarrow (\mathsf{sCoCoalg}^0_\F, \mathbb{\Omega}\text{-}\qi) 
\end{equation} 
\begin{equation} \label{quillen-adj2} \tag{2}
\F[-]\colon (\mathsf{sSet}_0, \pi_1\text{-}\F\text{-}\eq) \rightarrow (\mathsf{sCoCoalg}^0_\F, \widehat{\mathbb{\Omega}}\text{-}\qi)
\end{equation}
\begin{equation} \label{quillen-adj3} \tag{3}
\F[-]\colon (\mathsf{sSet}_0, \F\text{-}\eq) \rightarrow (\mathsf{sCoCoalg}^0_\F, \qi) 
\end{equation}
are homotopically full and faithful.
\end{theorem} 
\begin{proof} The proof is similar to the proof in \cite{Go} which essentially treated case \eqref{quillen-adj3}. For any field $\mathbb{F}$, the functor $\mathbb{F}[-] $ is (1-categorically) full and faithful. Thus, by Proposition \ref{criterion_hff}, it suffices to show that 
$\mathcal{P}: \mathsf{sCoCoalg}_{\F}^0 \to \mathsf{sSet}_0$ preserves all weak equivalences in each one of the cases. Given a map $f: C \to C'$ in $\mathsf{sCoCoalg}^0_{\mathbb{F}}$ and assuming that $\F$ is algebraically closed, it follows from Theorem \ref{decomposition_thm} and \eqref{etale_part_closed} that the map $\F[\mathcal{P}(f)]$ is a retract of $f$. Therefore, if $f$ is a weak equivalence (in any of the three cases), then $\F[\mathcal{P}(f)]$ is again a weak equivalence. Hence $\mathcal{P}(f)$ is a weak equivalence (in any of the three cases) as required, using Proposition \ref{RcateqandOmegaqi} (for \eqref{quillen-adj1}), 
Theorem \ref{LOmegareflects} (for \eqref{quillen-adj2}), and Proposition \ref{qiandqi} (for \eqref{quillen-adj3}). 
\end{proof}

\begin{remark}
An analogue of Theorem \ref{comparisons1}(3) for simplicial presheaves with respect to the local model category structures was shown in \cite{Ra} and for the motivic homotopy theory in \cite{Gu}.
\end{remark} 

\begin{corollary} \label{completeinvariant}
Let $\F$ be an algebraically closed field and let $X$ and $Y$ be reduced simplicial sets. 
\begin{enumerate}
\item Suppose $\ \F[X] \cong \F[Y]$ in $\mathrm{Ho}(\mathsf{sCoCoalg}^0_{\F}, \mathbb{\Omega}\text{-}\qi)$. Then 
$$X \cong Y \text{ in } \mathrm{Ho}(\mathsf{sSet}_0, \F\text{-}\cateq).$$ 
\item Suppose $\ \F[X] \cong \F[Y]$ in $\mathrm{Ho}(\mathsf{sCoCoalg}^0_{\F}, \widehat{\mathbb{\Omega}}\text{-}\qi)$. Then 
$$X \cong Y \text{ in } \mathrm{Ho}(\mathsf{sSet}_0, \pi_1\text{-}\F\text{-}\eq).$$ 
\item Suppose $\ \F[X] \cong \F[Y]$ in $\mathrm{Ho}(\mathsf{sCoCoalg}^0_{\F}, \qi)$. Then 
$$X \cong Y \text{ in } \mathrm{Ho}(\mathsf{sSet}_0, \F\text{-}\eq).$$ 
\end{enumerate}
\end{corollary}

\begin{remark}
A version of Corollary~\ref{completeinvariant}(2) was shown in \cite{RWZTransactions} assuming that $X$ and $Y$ are Kan complexes and using the standard (non-localized) cobar construction. 
\end{remark}

\subsection{Perfect fields} Let $\F$ be a perfect field, let $\F \subseteq \overline{\F}$ be its algebraic closure and let $G$ denote the (profinite) absolute Galois group of $\F$. The analogue of Theorem \ref{comparisons1} for perfect fields will make use of the adjunction (see Subsection \ref{coalgebra_prelim}): 
$$\overline{\F}[-]^G \colon \mathsf{sSet}(G)_0 \rightleftarrows \mathsf{sCoCoalg}^0_\F \colon \mathcal{P}_G.$$
Here $\mathsf{sSet}(G)_0$ denotes the category of reduced simplicial $G$-sets. We recall that the left adjoint $\overline{\F}[-]^G$ is full and faithful. 

We denote by $\delta \colon \mathsf{sSet}_0 \rightleftarrows \mathsf{sSet}(G)_0 \colon (-)^G$ the adjunction induced by the trivial $G$-action functor $\delta$ and the $G$-fixed points functor $(-)^G$. Our analysis of the Quillen adjunction (in all three cases) $\F[-] \colon \mathsf{sSet}_0 \rightleftarrows \mathsf{sCoCoalg}^0_{\F} \colon \mathcal{P}$
is based on its factorization as the composition of adjunctions
$$\F[-] \colon \mathsf{sSet}_0 \overset{\delta}{\underset{(-)^G}\rightleftarrows} \mathsf{sSet}(G)_0  \overset{\overline{\F}[-]^G}{\underset{\mathcal{P}_G}\rightleftarrows} \mathsf{sCoCoalg}^0_{\F} \colon \mathcal{P}.$$
We will construct three auxiliary model structures on $\mathsf{sSet}(G)_0$ in order to set the adjunctions above into a homotopical setting. We say that a map $f\colon S \to S'$ in $\mathsf{sSet}(G)_0$ is a categorical $\F$-equivalence (resp. $\pi^G_1$-$\F$-equivalence, $\F$-equivalence) if the functor $\overline{\F}[-]^G$ sends $f$ to an $\mathbb{\Omega}$-quasi-isomorphism (resp. $\widehat{\mathbb{\Omega}}$-quasi-isomorphism, quasi-isomorphism) in $\mathsf{sCoCoalg}^0_\F$. The terminology is justified by the following lemma. We denote by $$U \colon \mathsf{sSet}(G)_0 \to \mathsf{sSet}_0$$ the forgetful functor that forgets the $G$-action.

\begin{lemma} \label{G-weak-eq}
Let $f \colon S \to S'$ be a map in $\mathsf{sSet}(G)_0$. Then the following hold:
\begin{itemize}
    \item[(1)] $f$ is a categorical $\F$-equivalence if and only if $U(f)$ is a categorical $\F$-equivalence. 
    \item[(2)] if $f$ is a $\pi_1^G$-$\F$-equivalence, then $U(f)$ is a $\pi_1$-$\F$-equivalence.
    \item[(3)] $f$ is an $\F$-equivalence if and only if $U(f)$ is an $\F$-equivalence. 
    \end{itemize}
\end{lemma}
\begin{proof}
Note that the functor $(-) \otimes \overline{\F} \colon \mathsf{sCoCoalg}^0_{\F} \to \mathsf{sCoCoalg}^0_{\overline{\F}}$ preserves and detects quasi-isomorphisms. It also preserves and detects $\mathbb{\Omega}$-quasi-isomorphisms because we have an isomorphism $\mathbb{\Omega}(C \otimes \overline{\F}) \cong \mathbb{\Omega}(C) \otimes \overline{\F}$ for any $C \in \mathsf{sCoCoalg}^0_\F$. On the other hand, for $C \in \mathsf{sCoCoalg}^0_{\F}$, the associated fundamental $\overline{\F}$-bialgebra $H_0(\mathbb{\Omega}(C \otimes \overline{\F})) \cong H_0(\mathbb{\Omega}(C)) \otimes \overline{\F}$ could possibly have new monoid-like elements that do not arise from monoid-like elements of the fundamental $\F$-bialgebra $H_0(\mathbb{\Omega}(C))$ of $C$ -- this does not occur for $C$ of the form  $\F[S]$. Still, since localization preserves quasi-isomorphisms, it follows that $(-) \otimes \overline{\F}$ preserves $\widehat{\mathbb{\Omega}}$-quasi-isomorphisms. Then the claims (1)--(3) follow easily from the natural isomorphism \eqref{technical_iso} and the characterizations of the three classes of weak equivalences in $\mathsf{sSet}_0$ shown in Section \ref{sec_3_weq}.
\end{proof}

The next theorem establishes three corresponding model category structures on $\mathsf{sSet}(G)_0$ together with three pairs of associated Quillen adjunctions.

\begin{theorem} \label{Gsset_modelcat}  Let $\F$ be a perfect field, let $\F \subseteq \overline{\F}$ be its algebraic closure and let $G$ denote the (profinite) absolute Galois group of $\F$.
\begin{enumerate}
\item There is a left proper combinatorial model category structure on $\mathsf{sSet}(G)_0$ with the monomorphisms as cofibrations and the categorical $\F$-equivalences, denoted by $\W(G)_{J, \F}$, as weak equivalences. 

We denote this model category by $(\mathsf{sSet}(G)_0, \F\text{-}\cateq)$.

\medskip

\item There is a left proper combinatorial model category structure on $\mathsf{sSet}(G)_0$ with the monomorphisms as cofibrations and the $\pi^G_1$-$\F$-equivalences, denoted by $\W(G)_{\pi^G_1\text{-}\F}$, as weak equivalences. This model category is a left Bousfield localization of the model category in (1).

We denote this model category by $(\mathsf{sSet}(G)_0, \pi^G_1\text{-}\F\text{-}\eq)$. 

\medskip

\item There is a left proper combinatorial model category structure on $\mathsf{sSet}(G)_0$ with the monomorphisms as cofibrations and the $\F$-equivalences, denoted by $\W(G)_{\F}$, as weak equivalences. This model category is a left Bousfield localization of the model category in (2).

We denote this model category by $(\mathsf{sSet}(G)_0, \F\text{-}\eq)$.
\end{enumerate}  
Moreover, the adjunctions
\begin{equation} \label{quillen-adjG1} \tag{1}
\delta \colon (\mathsf{sSet}_0, \F\text{-}\cateq) \rightleftarrows (\mathsf{sSet}(G)_0, \F\text{-}\cateq) \colon (-)^G  
\end{equation} 
\begin{equation} \label{quillen-adjG2} \tag{2}
\delta \colon (\mathsf{sSet}_0, \pi_1\text{-}\F\text{-}\eq) \rightleftarrows (\mathsf{sSet}(G)_0, \pi^G_1\text{-}\F\text{-}\eq) \colon (-)^G 
\end{equation}
\begin{equation} \label{quillen-adjG3} \tag{3}
\delta \colon (\mathsf{sSet}_0, \F\text{-}\eq) \rightleftarrows (\mathsf{sSet}(G)_0,  \F\text{-}\eq) \colon (-)^G 
\end{equation}
and 
\begin{equation} \label{quillen-adjG1} \tag{1}
\overline{\F}[-]^G \colon (\mathsf{sSet}(G)_0, \F\text{-}\cateq) \rightleftarrows (\mathsf{sCoCoalg}^0_\F, \mathbb{\Omega}\text{-}\qi) \colon \mathcal{P}_G  
\end{equation} 
\begin{equation} \label{quillen-adjG2} \tag{2}
\overline{\F}[-]^G \colon (\mathsf{sSet}(G)_0, \pi^G_1\text{-}\F\text{-}\eq) \rightleftarrows (\mathsf{sCoCoalg}^0_\F, \widehat{\mathbb{\Omega}} \text{-}\qi) \colon \mathcal{P}_G 
\end{equation}
\begin{equation} \label{quillen-adjG3} \tag{3}
\overline{\F}[-]^G \colon (\mathsf{sSet}(G)_0, \F\text{-}\eq) \rightleftarrows (\mathsf{sCoCoalg}^0_\F, \qi) \colon \mathcal{P}_G 
\end{equation}
are Quillen adjunctions. 
\end{theorem}
\begin{proof} We will apply Theorem \ref{Smith_thm} to the left adjoint functor $\overline{\F}[-]^G \colon \mathsf{sSet}(G)_0 \to \mathsf{sCoCoalg}^0_{\F}$ and the model categories of Theorem \ref{coalg_modelcat}. First note that $\overline{\F}[-]^G$ preserves colimits; in particular, it is accessible. 

We recall that the larger category $\mathsf{sSet}(G)$ is a (category of simplicial objects in a) Grothendieck topos; either directly or by using this fact, it is easy to conclude that $\mathsf{sSet(G)}_0$ is also locally presentable. Moreover, the class of monomorphisms in $\mathsf{sSet}(G)$, as in any topos, is cofibrantly generated 
by a set of monomorphisms $\widetilde{\mathcal{I}}$. For every $\widetilde{i} \colon A \to B$ in $\widetilde{\mathcal{I}}$, let $i \colon A/A_0 \to B/B_0$ be the induced monomorphism in $\mathsf{sSet}(G)_0$, where $(-)_0$ denotes the $0$-simplices as a constant simplicial object. Let $\mathcal{I}$ denote the set of monomorphisms obtained from the morphisms in $\widetilde{\mathcal{I}}$ in this way. We observe that a map in $\mathsf{sSet}(G)$ between reduced simplicial $G$-sets has the right lifting property with respect to $\widetilde{\mathcal{I}}$ if and only if it is has the right lifting property with respect to $\mathcal{I}$. Therefore, by the retract argument or otherwise, the set of monomorphisms $\mathcal{I}$ is a generating set for the monomorphisms in $\mathsf{sSet}(G)_0$. In more detail, for every monomorphism $f$ in $\mathsf{sSet}(G)_0$, there is a factorization $f = p i$ where $i \in \mathrm{Cof}(\mathcal{I})$ and 
$p \in \mathcal{I}-\text{inj}$ using the small object argument; then, as observed above, $p$ has the right lifting property with respect to $f$, hence $f$ is a retract of $i$.  

Then it suffices to prove the stronger claim that $\mathcal{I}-\text{inj}$ consists of maps which are categorical equivalences of underlying reduced simplicial sets (after forgetting the $G$-actions). Then the result will follow from Theorem \ref{Smith_thm} for all three cases simultaneously, since each one of the three classes of weak equivalences contains the class of categorical equivalences between underlying reduced simplicial sets (see also Lemma \ref{G-weak-eq}). 

Let $p \colon X \to Y$ be a morphism in $\mathsf{sSet}(G)_0$ which has the right lifting property with respect to the monomorphisms. Using similar arguments as above, it follows that $p$ has the right lifting property in $\mathsf{sSet}(G)$ with respect to the monomorphisms in $\mathsf{sSet}(G)$. For every open subgroup $H$ of $G$ and any $n \geq 0$, the map $i \times \mathrm{id} \colon \partial \Delta^n \times G/H \to \Delta^n \times G/H$ is a monomorphism in $\mathsf{sSet}(G)$. Thus the right lifting property of $p$ implies that the map 
$$p^H \colon X^H \to Y^H$$
is a trivial fibration of simplicial sets for any such $H$. Using the continuity of the $G$-action on $X$ and $Y$, it follows that $U(p)$ is the filtered colimit of the maps $p^H$ where $H$ ranges over the open subgroups of $G$ (with finite index). Since filtered colimits of trivial fibrations of simplicial sets are again trivial fibrations, we conclude that $U(p)$ is a trivial fibration between (reduced) simplicial sets; therefore, it is also a categorical equivalence, as required. 

It now follows immediately from the definition of cofibrations and weak equivalences that the adjunctions
$$\delta \colon \mathsf{sSet}_0 \rightleftarrows \mathsf{sSet}(G)_0 \colon (-)^G \text{ and } \overline{\F}[-]^G \colon \mathsf{sSet}(G)_0 \rightleftarrows \mathsf{sCoCoalg}^0_\F \colon \mathcal{P}_G  
$$ are Quillen adjunctions in all three contexts. 
\end{proof}

\begin{remark}
The model category structure on $\mathsf{sSet}(G)$ analogous to Theorem \ref{Gsset_modelcat}(3) was also constructed by Goerss \cite{Go2}. Note that the above proof works for any class of weak equivalences $\W$ in $\mathsf{sSet}(G)_0$ (or in $\mathsf{sSet}(G)$) such that: 
\begin{itemize} 
\item[(i)] $\W$ is detected by the class of weak equivalences $W_{\mathcal{M}}$ in a combinatorial model category $\mathcal{M}$  via a functor $F \colon \mathsf{sSet}(G)_0 \to \mathcal{M}$;
\item[(ii)] $F$ preserves small colimits and sends monomorphisms to cofibrations;
\item[(iii)] $\W$ contains the maps which define categorical equivalences between the underlying simplicial sets. 
\end{itemize}
For example, this applies to the forgetful functor $U \colon \mathsf{sSet}(G)_0 \to \mathsf{sSet}_0$ and the model categories of Theorem \ref{sset_modelcat} for any commutative ring $R$ and profinite group $G$. Specialized to a perfect field $\F$ as above, the resulting model category agrees in the cases (1) and (3) with the corresponding one from Theorem \ref{Gsset_modelcat} (cf. Lemma \ref{G-weak-eq}). In the case (2), the model category resulting from $U$ defines a left Bousfield localization of the corresponding model category from Theorem \ref{Gsset_modelcat}(2). In this case, the functor $\overline{\F}[-]^G \colon \mathsf{sSet}(G)_0 \to \mathsf{sCoCoalg}^0_{\F}$ becomes a left Quillen functor if the target category is equipped with the left Bousfield localization of $(\mathsf{sCoCoalg}^0_\F, \widehat{\mathbb{\Omega}} \text{-}\qi)$ with weak equivalences those maps that become $\widehat{\mathbb{\Omega}}$-quasi-isomorphisms after tensoring first with $\overline{\F}$ (cf. \eqref{technical_iso} and the proof of Lemma \ref{G-weak-eq}).
\end{remark}

We can now state the analogue of Theorem \ref{comparisons1} for perfect fields. 

\begin{theorem} \label{comparisons2}
Let $\F$ be a perfect field and let $\F \subseteq \overline{\F}$ be its algebraic closure with (profinite) absolute Galois group $G$. Then the left Quillen functors
\begin{equation} \label{quillen--adj1} \tag{1}
\overline{\F}[-]^G \colon (\mathsf{sSet}(G)_0, \F\text{-}\cateq) \rightarrow (\mathsf{sCoCoalg}^0_\F, \mathbb{\Omega}\text{-}\qi) 
\end{equation} 
\begin{equation} \label{quillen--adj2} \tag{2}
\overline{\F}[-]^G \colon (\mathsf{sSet}(G)_0, \pi^G_1\text{-}\F\text{-}\eq) \rightarrow (\mathsf{sCoCoalg}^0_\F, \widehat{\mathbb{\Omega}}\text{-}\qi)
\end{equation}
\begin{equation} \label{quillen--adj3} \tag{3}
\overline{\F}[-]^G \colon (\mathsf{sSet}(G)_0, \F\text{-}\eq) \rightarrow (\mathsf{sCoCoalg}^0_\F, \qi) 
\end{equation}
are homotopically full and faithful.
\end{theorem} 
\begin{proof} The proof is similar to the proof of Theorem \ref{comparisons1}. We recall that the functor $\overline{\F}[-]^G$ is (1-categorically) full and faithful (see Subsection \ref{coalgebra_prelim}). Then, by Proposition \ref{criterion_hff}, it suffices to show that 
$\mathcal{P}_G: \mathsf{sCoCoalg}_{\F}^0 \to \mathsf{sSet}(G)_0$ preserves all weak equivalences. 
Let $f: C \to C'$ in $\mathsf{sCoCoalg}^0_{\mathbb{F}}$ be a weak equivalence (in any of the three cases). By Theorem \ref{decomposition_thm} and \eqref{etale_part}, the induced map 
$$\overline{\F}[\mathcal{P}_G(f)]^G \colon \overline{\F}[\mathcal{P}_G(C)]^G \to \overline{\F}[\mathcal{P}_G(C')]^G$$
is a retract of $f$. So $\overline{\F}[\mathcal{P}_G(f)]^G$ is again a weak equivalence. This means that the map $\mathcal{P}_G(f)$ is a weak equivalence in 
$\mathsf{sSet}(G)_0$ (in any of the three cases) and finishes the proof. 
\end{proof}

\begin{remark}
An analogue of Theorem \ref{comparisons2}(3) for simplicial presheaves with respect to the local model structures was shown in \cite{Ra} and for the motivic homotopy theory in \cite{Gu}.
\end{remark} 

The Quillen adjunction $\F[-] \colon \mathsf{sSet}_0 \rightleftarrows \mathsf{sCoCoalg}^0_{\F} \colon \mathcal{P}$ (in any of three cases -- Proposition \ref{quillen_adj}) factors as the composition of two adjunctions:
$$\delta \colon \mathsf{sSet}_0 \rightleftarrows \mathsf{sSet}(G)_0 \colon (-)^G$$
$$\overline{\F}[-]^G \colon \mathsf{sSet}(G)_0 \rightleftarrows \mathsf{sCoCoalg}^0_{\F} \colon \mathcal{P}_G$$
and these are Quillen adjunctions in each one of the three cases (Theorem \ref{Gsset_modelcat}).  Since the derived unit transformation of the last Quillen adjunction is a natural isomorphism by Theorem \ref{comparisons2} (in each one of the three cases), we obtain the following identification of the derived unit transformation of the composite Quillen adjunction. 

\begin{corollary} \label{perfectfield_corollary}
Let $\F$ be a perfect field with algebraic closure $\F \subseteq \overline{\F}$ and (profinite) absolute Galois group $G$. 
\begin{enumerate}
\item The derived unit transformation of the Quillen adjunction $$\F[-] \colon (\mathsf{sSet}_0, \F\text{-}\cateq) \rightleftarrows (\mathsf{sCoCoalg}^0_{\F}, \mathbb{\Omega}\text{-}\qi) \colon \mathcal{P}$$
is canonically identified with the derived unit transformation of the Quillen adjunction 
$$\delta \colon (\mathsf{sSet}_0, \F\text{-}\cateq) \rightleftarrows (\mathsf{sSet}(G)_0, \F\text{-}\cateq) \colon (-)^G.$$

\item The derived unit transformation of the Quillen adjunction $$\F[-] \colon (\mathsf{sSet}_0, \pi_1\text{-}\F\text{-}\eq) \rightleftarrows (\mathsf{sCoCoalg}^0_{\F}, \widehat{\mathbb{\Omega}}\text{-}\qi) \colon \mathcal{P}$$
is canonically identified with the derived unit transformation of the Quillen adjunction 
$$\delta \colon (\mathsf{sSet}_0, \pi_1\text{-}\F\text{-}\eq) \rightleftarrows (\mathsf{sSet}(G)_0, \pi^G_1\text{-}\F\text{-}\eq) \colon (-)^G.$$

\item The derived unit transformation of the Quillen adjunction $$\F[-] \colon (\mathsf{sSet}_0, \F\text{-}\eq) \rightleftarrows (\mathsf{sCoCoalg}^0_{\F}, \qi) \colon \mathcal{P}$$ is canonically identified with the derived unit transformation of the Quillen adjunction 
$$\delta \colon (\mathsf{sSet}_0, \F\text{-}\eq) \rightleftarrows (\mathsf{sSet}(G)_0, \F\text{-}\eq) \colon (-)^G.$$
\end{enumerate}
In other words, the derived unit transformation is identified in each case with the canonical map into the homotopy $G$-fixed points $X \to (\delta(X))^{hG}$ (where $(-)^{hG}$ is interpreted in the appropriate way in each model category). 
\end{corollary}

\begin{remark} In \cite[Proposition 1.5]{Go} the space of homotopy fixed points $\delta(X)^{hG}$ corresponding to Corollary \ref{perfectfield_corollary}(3) is described in certain special cases. For instance, if $G$ is a finite group and $X$ is fibrant in $(\mathsf{sSet}_0, \F\text{-eq.})$, then $\delta(X)^{hG}$ is the usual homotopy fixed point space of $X$.  

When $\F= \F_p$ and $X$ is simply-connected, then $\delta(X)^{hG}$ is the free loop space of the $p$-completion of $X$. If $\mathbb{F}$ has characteristic zero and $X$ is simply-connected, then $\delta(X)^{hG}$ is the rational localization of $X$. Analogous results hold also in the context of Corollary \ref{perfectfield_corollary}(2) using Lemma \ref{technical_lemma}.
\end{remark}

\appendix
\section{Cylinder objects for simplicial coalgebras} \label{sec_appendix}

\numberwithin{theorem}{section}
\numberwithin{equation}{theorem}

Let $R$ be a commutative ring. For any $C \in \mathsf{sCoCoalg}^0_{R}$, we define $\text{Cyl}(C)$ as in the pushout diagram \eqref{cylinder} in the proof of Theorem \ref{coalg_modelcat}. In this appendix, we prove the following result that was used in the proof of Theorem~\ref{coalg_modelcat}(1). 

\begin{proposition} \label{projcobarquasiiso} For any $C \in \mathsf{sCoCoalg}^0_{R}$, the natural projection map $$q\colon \emph{Cyl}(C) \to C$$ is an $\mathbb{\Omega}$-quasi-isomorphism. 
\end{proposition} 

For simplicity, denote by $\q=\mathcal{N}_*(q)\colon \mathcal{N}_*(\text{Cyl}(C)) \to \mathcal{N}_*(C)$ the map on normalized chains induced by the projection $q\colon \text{Cyl}(C) \to C$ and by $\y =\mathcal{N}_*(i_1) \colon \mathcal{N}_*(C) \to  \mathcal{N}_*(\text{Cyl}(C))$ the map induced by the inclusion $i_1 \colon C \to \text{Cyl}(C)$. To prove the above proposition, we will:
\begin{itemize}
\item construct a chain homotopy $$H \colon \mathcal{N}_*(\text{Cyl}(C)) \to \mathcal{N}_{*+1}(\text{Cyl}(C))$$ between the composition $\y \circ \q$ and the identity map $\text{id}_{ \mathcal{N}_*(\text{Cyl}(C))}$, and then

\item extend $H$ to a chain homotopy $$H_{\mathbb{\Omega}} \colon \mathbb{\Omega}(\text{Cyl}(C)) \to \mathbb{\Omega}(\text{Cyl}(C))$$ between the identity map $\text{id}_{\mathbb{\Omega}(\text{Cyl}(C))}$ and the composition $\mathbb{\Omega}(i_1) \circ \mathbb{\Omega}(q)$. This involves verifying that $H$ is a $(\y \circ \q, \text{id})$-\textit{coderivation}, i.e. that the equation $$\big(H \otimes (\y \circ \q) + \text{id} \otimes H\big) \circ \mathbf{ \Delta} = \mathbf{\Delta} \circ H$$ is satisfied. 
\end{itemize}

Using the notation introduced in the proof of Theorem~\ref{coalg_modelcat}, we may represent any element in $\text{Cyl}(C)_n$ by a linear combination of elements of the form $x \otimes [0^{k+1} 1^{n-k}]$, where $k \in \{-1, 0, \dots, n\}$. For simplicity, we will write $ x \otimes [0^{k+1} 1^{n-k}]= (x, [0^{k+1}1^{n-k}])$. Note that, when $0 \leq k \leq n-1$,
$$ [0^{k+1} 1^{n-k}] = s_{n-1}\dots \widehat{s_k} \dots s_0[01],$$
where  the $s_j$'s denote simplicial degeneracy maps, $[01] \in (\mathbb{\Delta}^1)_1$ is the unique non-degenerate $1$-simplex, and $\widehat{s_k}$ means we omit this map from the composition.  

\smallskip

Define a degree $+1$ map $$H \colon \mathcal{N}_*(\text{Cyl}(C))  \to \mathcal{N}_{*+1}(\text{Cyl}(C))$$ given on any representative $(x, [0^{k+1}1^{n-k}]) \in \text{Cyl}(C)_n$ by the formula
\begin{eqnarray}\label{Hdefinition}
H(x, [0^{k+1}1^{n-k}]) := \sum_{j=0}^k (-1)^j (s_{k-j}x, [0^{k-j+1}1^{n-k+j+1}])= \\
\notag
\sum_{j=0}^k (-1)^j (s_{k-j}x, s_n \dots \widehat{s_{k-j}} \dots s_0[01]) .
\end{eqnarray}

\begin{remark} Since we are working with normalized chains we have $H(x,[1^{n+1}])=0$ for any $x \in C_n$, $n \geq 0$. Also note $H(x,[0])=H(x,[1])=0$. Moreover, if $n>0$,  we have
$$H(x, [0^{n+1}])= (\mathcal{N}_*(p) \circ EZ)(x \otimes [01]),$$
where $$EZ \colon \mathcal{N}_*(C) \otimes \mathcal{N}_*(R[\mathbb{\Delta}^1]) \to \mathcal{N}_*(C \otimes R[\mathbb{\Delta}^1])$$ is the Eilenberg-Zilber map and $p\colon C \otimes R[\mathbb{\Delta}^1] \to \text{Cyl}(C)$ the map in the pushout defining $\text{Cyl}(C)$. In other words, on simplices at the $0$-th end of the cylinder, $H$ is given by ``crossing a simplex with the simplicial interval" and subdividing appropriately.  
\end{remark} 

An easy computation yields that $H$ satisfies the chain homotopy equation
$$ H \circ \partial + \partial \circ H = \y \circ \q - \text{id},$$ where $\partial \colon \mathcal{N}_*(C) \to \mathcal{N}_{*-1}(C)$ is the normalized chains differential. Then we define $$H_{\mathbb{\Omega}} \colon \mathbb{\Omega} (\text{Cyl}(C)) \to  \mathbb{\Omega} (\text{Cyl}(C))$$ by $$H_{\mathbb{\Omega}} := \sum_{i,j \geq 0}  (s^{-1} \circ( \y \circ \q )\circ s^{+1})^{\otimes i} \circ (s^{-1} \circ H \circ s^{+1}) \circ ( \text{id})^{\otimes j}.$$ We will show that $$D H_{\mathbb{\Omega}}  +  H_{\mathbb{\Omega}}D= \text{id} - \mathbb{\Omega}(i_1) \circ \mathbb{\Omega}(q),$$ where $D \colon \mathbb{\Omega}(\text{Cyl}(C)) \to \mathbb{\Omega}(\text{Cyl}(C))$ is the differential of the dg algebra $\mathbb{\Omega}(\text{Cyl}(C))$, but we first make a general observation. 

\medskip

The following notion was introduced in \cite[Section 1.11]{Mu}.
\begin{definition}Let $N=(N, \partial_N, \mathbf{\Delta}_N)$ and $N'=(N', \partial_{N'}, \mathbf{\Delta}_{N'})$ be two dg coalgebras and $f,g : N \to N'$ two morphisms of dg coalgebras. A degree $+1$ map  $F: N \to N'$ between underlying graded modules  is said to be an $(f,g)$-\textit{coderivation} if  the equation $$(F \otimes f + g \otimes F ) \circ \mathbf{\Delta}_N = \mathbf{\Delta}_{N'} \circ F$$ is satisfied.
\end{definition}

\begin{lemma} \label{fgcoderivation} Let $N=(N, \partial_N, \mathbf{\Delta}_N)$ and $N'=(N', \partial_{N'}, \mathbf{\Delta}_{N'})$ be two conilpotent dg $R$-coalgebras. Suppose $f,g: N \to N'$ are two morphisms of dg coalgebras and let $F: N \to N'$ be a chain homotopy between $f$ and $g$ (of degree $+1$). If  $F$ is an $(f,g)$-coderivation, then the map $$F_{\mathbb{\Omega}}: \mathsf{Cobar}(N) \to \mathsf{Cobar}(N')$$ defined by $$F_{\mathbb{\Omega}}:= \sum_{i,j \geq 0} \overline{f}^{\otimes i} \circ \overline{F} \circ \overline{g}^{\otimes j},$$ where
$\overline{f}= s^{-1} \circ f \circ s^{+1}, \overline{F}= s^{-1} \circ F \circ s^{+1}$, and $\overline{g}= s^{-1} \circ g \circ s^{+1}$, is a chain homotopy between  $\mathsf{Cobar}(g)= \sum_n \overline{g}^{\otimes n}$ and $\mathsf{Cobar}(f)= \sum_n \overline{f}^{\otimes n}$.
\end{lemma}
\begin{proof} Let $D_N$ and $D_{N'}$ be the differentials of $\mathsf{Cobar}(N)$ and $\mathsf{Cobar}(N')$, respectively. 
After using that $f$ and $g$ are maps of dg coalgebras to cancel terms, we obtain
\begin{eqnarray*}
D_{N'}F_{\mathbb{\Omega}} + F_{\mathbb{\Omega}}D_N=
\\
\sum - \overline{f}^{\otimes i} \otimes (\overline{\partial}_{N'} \circ \overline{F}) \otimes \overline{g}^{\otimes j} + \sum \overline{f}^{\otimes i} \otimes (\overline{\mathbf{\Delta}}_{N'} \circ\overline{F} ) \otimes \overline{g}^{\otimes j}
\\
+ \sum -\overline{f}^{\otimes i} \otimes ( \overline{F}\circ \overline{\partial}_N ) \otimes \overline{g}^{\otimes j} + \sum \overline{f}^{\otimes i} \otimes \Big((\overline{f} \otimes \overline{F} + \overline{F} \otimes \overline{g}) \circ \overline{\mathbf{\Delta}}_N\Big) \otimes \overline{g}^{\otimes j}.
\end{eqnarray*}
Since $F$ is a $(f,g)$-coderivation, the second and fourth terms in the above sum cancel and we obtain
\begin{eqnarray*}
D_{N'}F_{\mathbb{\Omega}} + F_{\mathbb{\Omega}}D_N=
\\
\sum - \overline{f}^{\otimes i} \otimes (\overline{\partial}_{N'} \circ \overline{F}) \otimes \overline{g}^{\otimes j} + \sum - \overline{f}^{\otimes i} \otimes  (\overline{F} \circ \overline{\partial}_{N})  \otimes \overline{g}^{\otimes j} =
\\
 \sum - \overline{f}^{\otimes i+1} \otimes \overline{g}^{\otimes j} + \sum \overline{f}^{\otimes i} \otimes \overline{g}^{\otimes j+1}=
\\
 \sum - \overline{f}^{\otimes n} + \sum \overline{g}^{\otimes n} = \mathsf{Cobar}(g) - \mathsf{Cobar}(f).
\end{eqnarray*}
\end{proof}
In order to apply Lemma~\ref{fgcoderivation} with $f = \y \circ \q$, $g= \text{id}_{\mathcal{N}_*(\text{Cyl}(C))}$, and $F=H$ as defined in \ref{Hdefinition}, we must verify the following.
\begin{proposition} \label{Hcoderivation} The map $H \colon \mathcal{N}_*(\emph{Cyl}(C)) \to \mathcal{N}_{*+1}(\emph{Cyl}(C))$ defined in ~\ref{Hdefinition} is a $(\y \circ \q, \emph{id})$-coderivation. 
\end{proposition} 
\begin{proof} 

We prove that 
\begin{eqnarray} \label{coderivationeqn}
(H \otimes (\y \circ \q) + \text{id} \otimes H) \circ \mathbf{\Delta} = \mathbf{\Delta} \circ H,
\end{eqnarray}
where $\mathbf{\Delta} \colon \mathcal{N}_*(\text{Cyl}(C)) \to \mathcal{N}_*(\text{Cyl}(C)) \otimes \mathcal{N}_*(\text{Cyl}(C))$ is the Alexander-Whitney coproduct. On any $(x,\sigma) \in \text{Cyl}(C)_n$, where $\sigma \in \mathbb{\Delta}^1_n$, the Alexander-Whitney coproduct is given by
\begin{eqnarray}\label{AWformula}
\mathbf{\Delta}(x,\sigma) = \sum_{p=1}^{n+1} (d_{p} \dots d_nx' , d_{p} \dots d_n\sigma) \otimes (d_0^{p-1}x'', d_0^{p-1}\sigma),
\end{eqnarray}
where we have used (generalized) Sweedler's notation writing $\Delta_n(x) =x' \otimes x''$ for the coproduct of the cocommutative coalgebra $C_n$. 

Note that on any $n$-simplex $\sigma=[0^{k+1} 1^{n-k}] \in \mathbb{\Delta}^1_n$, we can split the Alexander-Whitney coproduct into two sums
\begin{eqnarray} \label{AWonasimplex}
\notag
 \sum_{p=1}^{n+1} d_{p} \dots d_n[0^{k+1}1^{n-k}] \otimes d_0^{p-1}[0^{k+1}1^{n-k}] =
 \\ 
 \sum_{p=1}^{k+1} [0^p] \otimes [0^{k+2-p}1^{n-k}] + \sum_{p=k+2}^{n+1}[0^{k+1}1^{p-k-1}] \otimes [1^{n+2-p}].
\end{eqnarray}

\noindent For any $(x, [0^{k+1} 1^{n-k}]) \in \text{Cyl}(C)_n$, the right hand side of \ref{coderivationeqn} is given by
\begin{eqnarray*}
\mathbf{\Delta}\big( H(x, [0^{k+1} 1^{n-k}])\big) =
\\  \sum_{p=1}^{n+2} \ \sum_{j=0}^k (-1)^j(d_{p} \dots d_{n+1}s_{k-j} x', d_{p} \dots d_{n+1}[0^{k-j+1}1^{n+1-k+j}])
\\ \otimes (d_0^{p-1}s_{k-j}x'', d_0^{p-1}[0^{k-j+1}1^{n+1-k+j}]).
\end{eqnarray*}
Using \ref{AWonasimplex}, we split the above sum into two sums (I) and (II):
\\
\\
sum (I) is given by
\begin{eqnarray*}
 \sum_{p=1}^{k-j+1} \ \sum_{j=0}^k (-1)^j (d_{p} \dots d_{n+1}s_{k-j} x', [0^p]) \otimes (d_0^{p-1}s_{k-j}x'', [0^{k-j+2-p}1^{n+1-k+j}]),
\end{eqnarray*}
 and sum (II) by
 \begin{eqnarray*}
\sum_{p=k-j+2}^{n+2} \ \sum_{j=0}^k (-1)^j (d_{p} \dots d_{n+1}s_{k-j}x', [0^{k-j+1}1^{p-k+j-1}]) \otimes (d_0^{p-1}s_{k-j}x'', [1^{n+3-p}]).
\end{eqnarray*}
The simplicial identities, together with the fact that $H(y , [1^{r+1}])=0$ for any  $y \in C_r$, yield that, up to the Koszul sign rule,  sum (I) equals:
\begin{eqnarray*}
(\text{id} \otimes H) \mathbf{\Delta}(x, [0^{k+1} 1^{n-k}] )=
\\
(\text{id} \otimes H) (\sum_{p=1}^{k+1} (d_{p} \dots d_nx', [0^p]) \otimes (d_0^{p-1}x'', [0^{k+2-p}1^{n-k}] ))=
\\
\sum_{j=0}^{k+1-p} \ \sum_{p=1}^{k+1} (-1)^{p-1+j} (d_{p}\dots d_n x', [0^p]) \otimes (s_{k+1-p-j}d_0^{p-1}x'', [0^{k+2-p-j}1^{n-k+j+1}]).
\end{eqnarray*}
Finally, we use \ref{AWformula}, \ref{AWonasimplex}, and the formula $$(\y \circ \q)(y, [0^{r}1^s])=(y, [1^{r+s}])$$ to compute
\begin{eqnarray*}
(H \otimes (\y \circ \q)) \mathbf{\Delta}(x, [0^{k+1}, 1^{n-k}] ) =
\\
\sum_{p=1}^{n+1} H(d_{p} \dots d_n x' , d_{p} \dots d_n[0^{k+1}1^{n-k}]) \otimes (\y \circ \q) (d_0^{p-1}x'', d_0^{p-1}[0^{k+1}1^{n-k}])=
\\
\sum_{p=1}^{k+1}H(d_p \dots d_nx', [0^p]) \otimes (\y \circ \q)(d_0^{p-1}x'', [0^{k+2-p}1^{n-k}])
\\
+ \sum_{p=k+2}^{n+1} H(d_p \cdots d_nx', [0^{k+1}1^{p-k-1}])\otimes (\y \circ \q)(d_0^{p-1}x'', [1^{n+2-p}])=
\\
\sum_{p=1}^{k+1} \ \sum_{j=0}^{p-1} (-1)^j (s_{p-1-j}d_p \dots d_nx', [0^{p-j}1^{j+1}]) \otimes (d_0^{p-1}x'', [1^{n+2-p}])
\\
+\sum_{p=k+2}^{n+1} \ \sum_{j=0}^k (-1)^j (s_{k-j}d_p \dots d_nx', [0^{k+1-j}1^{p-k+j}]) \otimes (d_0^{p-1}x'',[1^{n+2-p}]).
\end{eqnarray*}
If we reindex the above two sums by setting $i=j-p+1+k$ and $q=p+1$ in the first sum and $q=p+1$ in the second, we obtain
\begin{eqnarray*}
\sum_{q=2}^{k+2} \ \sum_{i=k+2-q}^k (-1)^{q+i-k} (s_{k-i}d_{q-1} \dots d_nx', [0^{k+1-i}1^{q-k+i-1}]) \otimes (d_0^{q-2}x'', [1^{n+3-q}])
\\
+ \sum_{q=k+3}^{n+2} \ \sum_{j=0}^k (-1)^j (s_{k-j}d_{q-1} \dots d_nx', [0^{k+1-j}1^{q-k+j-1}])\otimes (d_0^{q-2}x'', [1^{n+3-q}]).
\end{eqnarray*}
This is exactly sum (II) after using the simplicial identities. 
 \end{proof}

We may now conclude the main result of this appendix.
\medskip

\noindent \textit{Proof of Proposition~\ref{projcobarquasiiso}:} Lemma~\ref{fgcoderivation} together with Proposition~\ref{Hcoderivation} imply that $H_{\mathbb{\Omega}}: \mathbb{\Omega} (\text{Cyl}(C)) \to \mathbb{\Omega} (\text{Cyl}(C))$ is a chain homotopy between the maps $\text{id}_{\mathbb{\Omega} (\text{Cyl}(C))}$ and $\mathsf{Cobar}(\y \circ \q)= \mathbb{\Omega}(i_1 \circ q)= \mathbb{\Omega}(i_1) \circ \mathbb{\Omega}(q)$. Since $\mathbb{\Omega}(q) \circ \mathbb{\Omega}(i_1) = \text{id}_{\mathbb{\Omega} (\text{Cyl}(C))}$, it follows that both maps of dg algebras $\mathbb{\Omega}(q)$ and $\mathbb{\Omega}(i_1)$ are chain homotopy inverses to each other and, consequently, quasi-isomorphisms. \qed

\end{document}